\documentclass[11pt, letterpaper]{article}    	
 \usepackage[margin=.75in]{geometry}
\usepackage{setspace}
\usepackage{graphicx}	
\usepackage{subcaption}
\usepackage{float}
\usepackage[title,titletoc,toc]{appendix}

\usepackage[draft]{todonotes}   

\usepackage{paralist}		
\usepackage{enumitem, hyperref}
\makeatletter
\def\namedlabel#1#2{\begingroup
    #2%
    \def\@currentlabel{#2}%
    \phantomsection\label{#1}\endgroup
}

\usepackage{amssymb}
\usepackage{amsmath, amsthm, amssymb}

\newtheorem{thm}{Theorem}[section]

\newtheorem{prop}[thm]{Proposition}
\newtheorem{claim}{Claim}
\newtheorem{lemma}[thm]{Lemma}
\newtheorem{cor}[thm]{Corollary}

\theoremstyle{definition}

\newtheorem{defi}[thm]{Definition}

\newcommand{\E}{\mathbf{E}}
\newcommand{\I}{\mathbf{I}}
\newcommand{\X}{\mathbf{X}}
\newcommand{\Y}{\mathbf{Y}}
\newcommand{\rva}{\mathbf{a}}
\newcommand{\rvb}{\mathbf{b}}

\newcommand{\Prb}{\mathbf{Pr}}

\newcommand{\Pa}{P}
\newcommand{\nPa}{\hat{P}}
\newcommand{\pa}{p}

\newcommand{\Ac}{A}
\newcommand{\ac}{\gamma}
\newcommand{\Ke}{K}
\newcommand{\ke}{\eta}
\newcommand{\Lo}{L}
\newcommand{\nlo}{q}

\newcommand{\Cco}{\mathrm{C}}
\newcommand{\cde}{\varphi_1}
\newcommand{\cac}{\varphi_2}
\newcommand{\rd}{i}

\newcommand{\er}{\varepsilon}

\newcommand{\Hc}{\mathcal{H}}
\newcommand{\Tc}{\mathcal{T}}
\newcommand{\Fc}{\mathcal{F}}
\newcommand{\Ic}{\mathcal{I}}
\newcommand{\A}{\mathcal{A}}
\newcommand{\B}{\mathcal{B}}
\newcommand{\C}{\mathcal{C}}
\newcommand{\D}{\mathcal{D}}
\newcommand{\bO}{\mathcal{O}}

\usepackage{tikz}
\newcounter{casenum}

\newcommand*{\rom}[1]{\expandafter{\romannumeral #1\relax}}

\usepackage{authblk}
\makeatletter
\def\@cite#1#2{{\normalfont[{\bfseries#1\if@tempswa , #2\fi}]}}
\makeatother

\title{The chromatic number of triangle-free hypergraphs}

\author{Lina Li\thanks{Combinatorics and Optimization Department,
University of Waterloo, Waterloo, Ontario N2L 3G1, Canada \texttt{lina.li@uwaterloo.ca}.}
\quad Luke Postle\thanks{Combinatorics and Optimization Department,
University of Waterloo, Waterloo, Ontario N2L 3G1, Canada {\tt lpostle@uwaterloo.ca}. Partially supported by NSERC
under Discovery Grant No. 2019-04304 and the Canada Research Chair program.}}

\begin{document}

\maketitle
\begin{abstract}
A triangle in a hypergraph $\Hc$ is a set of three distinct edges $e, f, g\in\Hc$ and three distinct vertices $u, v, w\in V(\Hc)$ such that $\{u, v\}\subseteq e$, $\{v, w\}\subseteq f$, $\{w, u\}\subseteq g$ and $\{u, v, w\}\cap e\cap f\cap g=\emptyset$.
Johansson~\cite{johansson1996asymptotic} proved in 1996 that $\chi(G)=\bO(\Delta/\log\Delta)$ for any triangle-free graph $G$ with maximum degree $\Delta$. Cooper and Mubayi~\cite{cooper2015list} later generalized the Johansson's theorem to all rank $3$ hypergraphs.
In this paper we provide a common generalization of both these results for all hypergraphs, showing that if $\Hc$ is a rank $k$, triangle-free hypergraph, then the list chromatic number
\[
\chi_{\ell}(\Hc)\leq \bO\left(\max_{2\leq \ell \leq k} \left\{\left(  \frac{\Delta_{\ell}}{\log \Delta_{\ell}} \right)^{\frac{1}{\ell-1}} \right\}\right),
\]
where $\Delta_{\ell}$ is the maximum $\ell$-degree of $\Hc$.
The result is sharp apart from the constant.
Moreover, our result implies, generalizes and improves several earlier results on the chromatic number and also independence number of hypergraphs, while its proof is based on a different approach than prior works in hypergraphs (and therefore provides alternative proofs to them).
In particular, as an application, we establish a bound on chromatic number of sparse hypergraphs in which each vertex is contained in few triangles, and thus extend results of Alon, Krivelevich and Sudakov~\cite{alon1999coloring} and Cooper and Mubayi~\cite{cooper2016coloring} from hypergraphs of rank 2 and 3, respectively, to all hypergraphs.
\end{abstract}

\section{Introduction}\label{sec:intro}
A hypergraph is a pair $(V,E)$ where $V$ is a set whose elements are called \textit{vertices}, and $E$ is a family of subsets of $V$ called \textit{edges}.
A hypergraph has \textit{rank $k$} if every edge contains between $2$ and $k$ vertices, and is \textit{$k$-uniform} if every edge contains exactly $k$ vertices. 
A \textit{proper coloring} of $\Hc$ is an assignment of colors to the vertices so that no edge is monochromatic. 
The smallest number of colors that are required for a proper coloring of $\Hc$, is called the \textit{chromatic number} of $\Hc$ and denoted by $\chi(\Hc)$.
Given a set $L(v)$ of colors for every vertex $v\in V(\Hc)$, a \textit{proper list coloring} of $\Hc$ is a proper coloring, where every vertex $v$ receives a color from $L(v)$. 
The \textit{list chromatic number} of $\Hc$, denoted by $\chi_{\ell}(\Hc)$, is the minimum number $c$ so that if $|L(v)| \geq c$ for all $v$, then $\Hc$ has a proper list coloring.
It is not hard to see that $\chi(\Hc) \leq \chi_{\ell}(\Hc)$.

The study of the chromatic number of graphs (i.e., 2-uniform hypergraphs) has a rich history.
A straightforward greedy coloring algorithm shows that any graph $G$ with maximum degree $\Delta$ has chromatic number $\chi(G) \leq \Delta +1$, while the celebrated Brooks' theorem~\cite{brooks1941colouring} states that for connected graphs equality holds only for cliques and odd cycles.
Moving beyond Brooks' theorem, a natural question to consider is: what structural constraints one can be put on a graph to decrease its chromatic number?
In particular, Vizing~\cite{vizing1968some} proposed a question in 1968 which asked for the best possible upper bound for the chromatic number of a triangle-free graph.
Improving on results of Catlin~\cite{catlin1978bound}, Lawrence~\cite{lawrence1978covering}, Borodin and Kostochka ~\cite{borodin1977upper}, Kostochka~\cite{kostochkaletter}, and Kim~\cite{kim1995brooks}, in 1996 Johansson~\cite{johansson1996asymptotic} showed that 
\begin{equation}\label{johansson}
\chi(G)=\bO(\Delta/\log\Delta)
\end{equation}
for any triangle-free graph $G$ with maximum degree $\Delta$, and this bound is known to be tight up to a constant factor by constructions of Kostochka and Masurova~\cite{kostochka1977inequality}, and Bollob{\'a}s~\cite{bollobas1978chromatic}.
Indeed, Johansson~\cite{johansson1996asymptotic} proved a stronger result by showing that the list chromatic number $\chi_{\ell}(G)\leq (9+o(1))\Delta/\log\Delta.$
Pettie and Su~\cite{pettie2015distributed} subsequently improved the above constant from $9$ to $4$.
Later, Molloy~\cite{molloy2019list} further reduced the constant to $1$ while Bernshteyn~\cite{bernshteyn2019johansson} then provided a shorter proof of Molloy's result.

Analogous problems have also been investigated for hypergraphs by many researchers over the years. 
For a rank $k$ hypergraph $\Hc$ and an positive integer $i \leq k$, the \textit{$i$-degree} of a
vertex $v$ is the number of size $i$ edges containing $v$.
Using the Lov{\'a}sz Local Lemma, one can easily show that $\chi(\Hc)=\bO\left(\Delta^{1/(k-1)}\right)$ for any $k$-uniform hypergraph $\Hc$ with maximum $k$-degree $\Delta$, see Erd\H{o}s and Lov{\'a}sz~\cite{erdos1975problems}.
Similarly as for the graph case, one may ask what local constraints can be imposed on a hypergraph in order to significantly improve its chromatic number beyond this easy bound. We say a hypergraph is \textit{linear} if any two of its edges intersect in at most one vertex, and a \textit{loose triangle} in a linear hypergraph is a set of three pairwise intersecting edges containing no common point.
Frieze and Mubayi~\cite{frieze2008chromatic} first generalized Johansson's theorem (that is,~\eqref{johansson}) to all $3$-uniform linear hypergraphs as follows.
\begin{thm}[Frieze and Mubayi~\cite{frieze2008chromatic}]\label{thm:FM1}
If $\Hc$ is a  linear $3$-uniform hypergraph which does not contain any loose triangles, then
\[
\chi(\Hc)=\bO\left((\Delta/\log\Delta)^{1/2}\right),
\]
where $\Delta$ is the maximum $3$-degree of $\Hc$.
\end{thm}
It was subsequently realized by the same group in~\cite{frieze2013coloring} that for linear hypergraphs, the triangle-free condition in Theorem~\ref{thm:FM1} can be removed while the same conclusion still holds. 
Meanwhile, they also showed that such a linear hypergraph result can be generalized to any uniformity, by proving that if $\Hc$ is a $k$-uniform linear hypergraph with maximum $k$-degree $\Delta$, then 
\begin{equation}\label{linearthm}
\chi(\Hc)=\bO\left((\Delta/\log\Delta)^{1/(k-1)}\right).
\end{equation}
On the other hand, Cooper and Mubayi~\cite{cooper2015list} removed the restriction to linear systems from~Theorem~\ref{thm:FM1}, and then generalized Johansson's theorem from graphs to all rank $3$ hypergraphs. 
In order to formally state their result, we first introduce a definition of `triangle' for general hypergraphs, which was used in~\cite{cooper2015list}.
\begin{defi}[Triangle]\label{def:tri}
A \textit{triangle} in a hypergraph $\Hc$ is a set of three distinct edges $e, f, g\in\Hc$ and three distinct vertices $u, v, w\in V(H)$ such that $\{u, v\}\subseteq e$, $\{v, w\}\subseteq f$, $\{w, u\}\subseteq g$ and $\{u, v, w\}\cap e\cap f\cap g=\emptyset$.
\end{defi}
Note that similarly to a more classical definition of the hypergraph triangle, the \textit{Berge triangle}\footnote{A \textit{Berge triangle} in a hypergraph $\Hc$ is a set of three distinct edges $e, f, g\in\Hc$ such that there
exists three distinct vertices $u, v, w\in V(H)$ with $\{u, v\}\subseteq e$, $\{v, w\}\subseteq f$, $\{w, u\}\subseteq g$.}, the notion of triangle in Definition~\ref{def:tri} refers to a family of hypergraphs. 
However, Definition~\ref{def:tri} is weaker than the definition of Berge triangles, in the sense that the triangle family it refers to is a subfamily of Berge triangles.
For example, there are three different triangles in a $3$-uniform hypergraph: the loose triangle $C_3=\{abc, cde, efa\}$, $F_5=\{abc, bcd, aed\}$, and $K^-_4=\{abc, bcd, abd\}$. On the other hand, $\{abc,bcd,ace\}$ is a Berge triangle but not a triangle. 

We say a hypergraph is \textit{triangle-free} if it does not contain any triangle as a subgraph.
As in Johansson's theorem~\cite{johansson1996asymptotic}, the main result of Cooper and Mubayi can be stated in terms of list chromatic number.
\begin{thm}[Cooper and Mubayi~\cite{cooper2015list}]\label{CooperMubayi}
Let $\Hc$ be a rank $3$, triangle-free hypergraph with maximum $3$-degree $\Delta_3$ and maximum $2$-degree $\Delta_2$. Then
\[
\chi_{\ell}(\Hc) \leq c\cdot \max\left\{
\left(\frac{\Delta_3}{\log\Delta_3}\right)^{1/2},\ \frac{\Delta_2}{\log\Delta_2}
\right\},
\]
where $c$ is a fixed constant, not depending on $\Hc$.
\end{thm}

\subsection{Our main result}
Given a rank $k$ hypergraph $\Hc$, an integer $2\leq \ell \leq k$, and a set $S$ of vertices (where $1 \leq |S| <\ell$), we define $\deg_{\ell}(S, \Hc)$ to be the number of size $\ell$ edges containing $S$.
In particular, when $S$ consists of a single vertex $v$, then $\deg_{\ell}(S, \Hc)$ is exactly the $i$-degree of $v$.
The \textit{maximum $\ell$-degree} of $\Hc$, denoted by $\Delta_{\ell}(\Hc)$, is the maximum of $\mathrm{deg}_{\ell}(v, \Hc)$ over all vertices $v$ in $\Hc$; the \textit{maximum $(s,\ell)$-codegree of $\Hc$}, denoted by $\delta_{s, \ell}(\Hc)$, is the maximum of $\deg_{\ell}(S, \Hc)$ over all $s$-vertex sets $S$ in $\Hc$.
When the underlying graph is clear from the context, we simply write $\Delta_{\ell}$ and $\delta_{s, \ell}$ instead.

We extend Cooper and Mubayi's theorem to all hypergraphs as follows.
\begin{thm}\label{mainthm1}
Let $k\geq 3$ be an integer,
and $\Hc$ be a rank $k$, triangle-free hypergraph. Then
\[
\chi_{\ell}(\Hc)\leq c\cdot \max_{2\leq \ell \leq k} \left\{\left(  \frac{\Delta_{\ell}}{\log \Delta_{\ell}} \right)^{\frac{1}{\ell-1}} \right\},
\]
where $c$ depends only on $k$, not on $\Hc$.
\end{thm}
It is shown in~\cite[Theorem~5]{frieze2008chromatic} that there exists $k$-uniform triangle-free hypergraph with maximum $k$-degree $\Delta$ and chromatic number at least $c'\left(\Delta/\log \Delta\right)^{1/(k-1)}$ for some absolute constant $c'$, which only depends on $k$.
Therefore, Theorem~\ref{mainthm1} is sharp apart from the constant $c$.

In fact, we will derive Theorem~\ref{mainthm1} as a corollary of the following weaker theorem.
\begin{thm}\label{mainthm2}
Let $k\geq 3$ be an integer,
and $\Hc$ be a rank $k$, triangle-free hypergraph. If there exists $\Delta$ such that 
\begin{itemize}
\item[(i)] $\Delta_{\ell} \leq \Delta^{1 - \frac{k - \ell}{k-1}} (\log \Delta)^{\frac{k - \ell}{ k-1}}$ for $2\leq \ell \leq k$;
\item[(ii)] $\delta_{s, \ell} \leq (\Delta/\log\Delta)^{\frac{\ell - s}{k-1}}$ for $2\leq s<\ell \leq k$,
\end{itemize}
then
\[
\begin{split}
\chi_{\ell}(\Hc) &\leq c \cdot\left( \frac{\Delta}{\log \Delta}\right)^{\frac{1}{k-1}},
\end{split}
\]
where $c$ depends only on $k$, not on $\Hc$.
\end{thm}
Here, one may interpret $\Delta$ as a rescaled bound on the degrees of $\Hc$, based on the uniformity. Specifically, when $\Hc$ is $k$-uniform, we have $\Delta_k \leq \Delta$.
%
%

\subsection{Applications to sparse hypergraphs coloring}
Alon, Krivelevich and Sudakov~\cite{alon1999coloring} extended~\eqref{johansson} by showing that for a graph $G$ with maximum degree $\Delta$, if every vertex
$u$ is in at most $\Delta^2/f$ triangles, then 
\begin{equation}\label{ref:AKS}
\chi(G) = \bO(\Delta/\log f),
\end{equation}
where $\Delta\rightarrow\infty$.
This was later generalized to rank $3$ hypeprgraphs due to the work of Cooper and Mubayi~\cite{cooper2016coloring}. To state their result, we first recall some terminology from~\cite{cooper2016coloring}.

Given two hypergraphs $\Fc_1$ and $\Fc_2$, a map $\phi : V(\Fc_1) \rightarrow V(\Fc_2)$ is an \textit{isomorphism} if
for all $E\subset V(\Fc_1)$, $\phi(E)\in \Fc_2$ if and only if $E\in \Fc_1$. If there exists an isomorphism $\phi : V (\Fc_1) \rightarrow V (\Fc_2)$, we say $\Fc_1$ is \textit{isomorphic} to $\Fc_2$ and denoted it by $\Fc_1\cong_{\phi}\Fc_2$.
For two hypergraphs $\Fc, \Hc$ and a vertex $v\in V(\Fc)$, let
\[
\Delta_{\Fc, v}(\Hc)=\max_{u\in V(\Hc)}|\{\Fc'\subseteq \Hc: \Fc'\cong_{\phi} \Fc \text{ and } \phi(u)=v\}|
\]
and 
\[
\Delta_{\Fc}(\Hc) = \min_{v\in V(\Fc)}\Delta_{\Fc, v}(\Hc).
\]

Cooper and Mubayi~\cite{cooper2016coloring} proved the following theorem.
\begin{thm}[Cooper and Mubayi~\cite{cooper2016coloring}]\label{CMspar}
Let $\Hc$ be a rank $3$ hypergraph with maximum $3$-degree $\Delta_3$ and maximum 2-degree $\Delta_2$.
Let $\Tc$ denote the family of rank $3$ triangles.
If 
\[
\Delta_{T}(\Hc)\leq \left(\max\left\{\Delta_3^{1/2}, \Delta_2\right\}\right)^{v(T)-1}/f
\]
for all $T\in \Tc$, then
\[
\chi(\Hc)\leq \bO\left(\max \left\{\left(\frac{\Delta_3}{\log f} \right)^{1/2},\ \frac{\Delta_{2}}{\log f} \right\}\right).
\]
\end{thm}
The main idea behind both~\eqref{ref:AKS} and Theorem~\ref{CMspar} is the following: if a graph/hypergraph $\Hc$ is \textit{sufficiently sparse} (i.e., has bounded degrees and codegrees), and every vertex lies in \textit{not many} triangles, then we can partition $\Hc$ into \textit{a few} graphs/hypergraphs such that each of them is \textit{triangle-free}; after that, we just apply the known results for triangle-free graphs on each part individually. 
As the main contribution of their paper, Cooper and Mubayi~\cite{cooper2016coloring} established such partition lemma (see Section~\ref{sec:spars} for details) in a even more general set-up: the hypergraph can be of any rank, and triangles can be replaced by other families of some fixed hypegraphs.
Therefore, the only missing ingredient for extending~\eqref{ref:AKS} to any rank is in proving the corresponding result for the chromatic number of triangle-free hypergraphs.

By using our main theorem then, we generalize the results of~Alon, Krivelevich and Sudakov~\cite{alon1999coloring} and Cooper and Mubayi~\cite{cooper2016coloring} to all hypergraphs as follows.

\begin{thm}\label{sparsethm}
Fix $k\geq 3$. Let $\Hc$ be a rank $k$ hypergraph with maximum $\ell$-degree at most $\Delta_{\ell}$ for each $2\leq \ell\leq k$.
Denote by $\Tc$ the family of rank $k$ triangles.
If 
\[
\Delta_{T}(\Hc)\leq \left(\max_{2\leq\ell\leq k}\Delta_{\ell}^{1/(\ell-1)}\right)^{v(T)-1}/f
\]
for all $T\in \Tc$, then
\[
\chi(\Hc)\leq \bO\left(\max_{2\leq \ell \leq k} \left\{\left(  \frac{\Delta_{\ell}}{\log f} \right)^{\frac{1}{\ell-1}} \right\}\right).
\]
\end{thm}
Notice that the hypotheses of Theorem~\ref{sparsethm} are satisfied when $\Hc$ is linear, $k$-uniform and $f =\Delta_k^{1 - 1/(k-1)}$, so Theorem~\ref{sparsethm} implies \eqref{linearthm}.

Using Theorem~\ref{sparsethm}, we also extend a result of Cooper-Mubayi~\cite[Theorem 5]{cooper2016coloring} from $k$-uniform hypergraphs to all hypergraphs with the following theorem.
\begin{thm}\label{CMspar2}
Fix $k\geq 3$. Let $\Hc$ be a rank $k$ hypergraph with maximum $\ell$-degree at most $\Delta_{\ell}$ for each $2\leq \ell\leq k$.
Suppose that for all $2\leq s<\ell \leq k$, the maximum $(s, \ell)$-codegree
\[
\delta_{s, \ell}(\Hc)\leq \left(\max_{2\leq\ell\leq k}\Delta_{\ell}^{1/(\ell-1)}\right)^{\ell-s}/f,
\]
and additionally for the graph triangle $T_0$,
\[
\Delta_{T_0}(\Hc)\leq \left(\max_{2\leq\ell\leq k}\Delta_{\ell}^{1/(\ell-1)}\right)^{2}/f.
\]
Then we have
\[
\chi(\Hc)\leq \bO\left(\max_{2\leq \ell \leq k} \left\{\left(  \frac{\Delta_{\ell}}{\log f} \right)^{\frac{1}{\ell-1}} \right\}\right).
\]
\end{thm}
Observe that given such an $\Hc$, for any rank $k$ triangle $T$ (except for $T_0$), one can easily use the codegree conditions to show that 
\[
\Delta_{T}(\Hc)\leq \bO\left(\left(\max_{2\leq\ell\leq k}\Delta_{\ell}^{1/(\ell-1)}\right)^{v(T)-1}/f \right).
\]
Theorem~\ref{sparsethm} then immediately yields Theorem~\ref{CMspar2}.

\subsection{Applications to independence number of hypergraphs}
Closely related to coloring problems are questions about the independence number of hypergraphs.
The \textit{independence number $\alpha(\Hc)$} of a hypergraph $\Hc$ is the size of a largest set of vertices containing no edge of $\Hc$.
Using Turan's theorem, one can easily show that a $n$-vertex $k$-uniform hypergraphs with maximum degree $\Delta_k$ has $\alpha(\Hc)=\Omega\left(n/\Delta_k^{1/(k-1)}\right)$.
A seminal result of Ajtai, Koml{\'o}s, Pintz, Spencer, and Szemer{\'e}di~\cite{ajtai1982extremal} showed that this trivial lower bound could be improved by forbidding certain small subgraphs.

For $\ell\geq 2$, a \textit{(Berge) cycle} of length $\ell$ in $\Hc$ is a collection of $\ell$ edges $E_1,\cdots, E_{\ell}\in\Hc$ such that there exists $\ell$ distinct vertices $v_1,\cdots, v_{\ell}$ with $v_i\in E_i\cap E_{i+1}$ for $i\in[\ell-1]$ and $v_{\ell} \in E_{\ell}\cap E_1$.
\begin{thm}[Ajtai, Koml{\'o}s, Pintz, Spencer and Szemer{\'e}di~\cite{ajtai1982extremal}]\label{AKPSS}
Let $\Hc$ be a $k$-uniform hypergraph with maximum degree $\Delta_k$ that contains no cycles of length 2, 3, and 4. Then
\[
\alpha(\Hc)\geq c\cdot n\left(\frac{\log\Delta_k}{\Delta_k}\right)^{1/(k-1)}
\]
where $c$ depends only on $k$, not on $\Hc$.
\end{thm}
Ajtai, Erd\H{o}s, Koml{\'o}s and Szemer{\'e}di~\cite{ajtai1981turan} proposed the problem on determining whether Theorem~\ref{AKPSS} could also be extended to other families of hypergraphs. 
In particular, Spencer~\cite{nesetril2012mathematics} conjectured that the same conclusion holds for linear hypergraphs, and this was later proved by Duke, Lefmann and R{\"o}dl~\cite{duke1995uncrowded}.

Our main result (Theorem~\ref{mainthm1}) immediately yields the following strengthed version of Theorem~\ref{AKPSS}, showing that the same lower bound holds even if $\Hc$ just contains no triangles.
\begin{thm}\label{mainthmind}
Let $\Hc$ be a $k$-uniform triangle-free hypergraph with maximum degree $\Delta_k$. 
Then 
\[
\alpha(\Hc)\geq c\cdot n\left(\frac{\log\Delta_k}{\Delta_k}\right)^{1/(k-1)}
\]
where $c$ depends only on $k$, not on $\Hc$.
\end{thm}

Rather than considering $F$-free hypergraphs, Kostochka, Mubayi and Verstra{\"e}te~\cite{kostochka2014independent} proved the following general result on the independence number of $k$-uniform hypergraphs given the maximum $(k-1, k)$-codegree.
\begin{thm}[Kostochka, Mubayi and Verstra{\"e}te~\cite{kostochka2014independent}]\label{KMVCo}
Fix $k\geq 3$. There exists $c_k>0$ such that if $\Hc$ is an $k$-uniform hypergraph on $n$ vertices with the maximum $(k-1, k)$-codegree $\delta_{k-1, k}(\Hc):=d< n/(\log n)^{3(r-1)^2}$, then
\[
\alpha(\Hc) \geq c_k\left(\frac{n}{d}\log\frac{n}{d}\right)^{\frac{1}{k-1}},
\]
where $c_k>0$ and $c_k \sim k/e$ as $k\rightarrow\infty$.
\end{thm}
Our next theorem improves and extends Kostochka, Mubayi and Verstra{\"e}te’s result on the independence
number to non-uniform hypergraphs and chromatic number. We also weaken the hypothesis
by not requiring any upper bound condition on codegrees. 
On the other hand, as $k\rightarrow \infty$, Theorem~\ref{KMVCo} is best possible
including the value of the constant $c_k$ (see the matching lower bound construction in~\cite{kostochka2014independent}), whereas we do not optimize the constant in our result below.
\begin{thm}
Fix $k\geq 3$, and let $\Hc$ be a rank $k$ hypergraph on $n$ vertices with the maximum $(\ell-1, \ell)$-codegree $\delta_{\ell-1, \ell}(\Hc)\leq d_{\ell}$. Set
\[
f:=\min_{2\leq \ell \leq k}\left\{\left(n/d_\ell\right)^{\frac{1}{\ell-1}}\right\},
\] 
and assume additionally that for the graph triangle $T_0$,
\[
\Delta_{T_0}(\Hc)\leq n^2/f^3.
\]
Then we have
\[
\chi(\Hc)\leq \bO\left(\max_{2\leq \ell \leq k} \left\{\left(\frac{n^{\ell-2}d_{\ell}}{\log f} \right)^{\frac{1}{\ell-1}} \right\}\right).
\]
In particular, 
\[
\alpha(\Hc) \geq \Omega\left(\min_{2\leq\ell\leq k}\left\{\left(\frac{n}{d_{\ell}}\log f\right)^{\frac{1}{\ell-1}}\right\}\right).
\]
\end{thm}
\begin{proof}
Observe that for every $2\leq s<\ell \leq k$, we have $\Delta_{\ell}(\Hc)\leq n^{\ell-2}d_{\ell}$, and
\[
\delta_{s, \ell}(\Hc)\leq n^{\ell-s-1}d_{\ell} = (n^{\ell-2}d_{\ell})^{\frac{\ell-s}{\ell-1}}/\left(n/d_{\ell}\right)^{1- \frac{\ell-s}{\ell-1}}\leq (n^{\ell-2}d_{\ell})^{\frac{\ell-s}{\ell-1}}/f.
\]
Moreover, for the graph triangle $T_0$, we have
\[
\Delta_{T_0}(\Hc)\leq n^2/f^3 = \left(\max_{2\leq \ell\leq k}\left\{(n^{\ell-2}d_{\ell})^{\frac{1}{\ell-1}}\right\}\right)^{2}/f
\]
Then by Theorem~\ref{CMspar2}, we obtain that
\[
\chi(\Hc)\leq \bO\left(\max_{2\leq \ell \leq k} \left\{\left(\frac{n^{\ell-2}d_{\ell}}{\log f} \right)^{\frac{1}{\ell-1}} \right\}\right). 
\]

\end{proof}

\section{Informal proof overview and Organization of paper}\label{sec:overview}
Like many previous works, our proof technique is rooted in the celebrated \emph{nibble} method, while incorporating several innovative and critical modifications to tackle the challenges of hypergraph coloring problems.

The \emph{nibble} method (also known as the \emph{semi-random} method) was first introduced by R{\"o}dl~\cite{rodl1985packing} in 1985, to resolve the Erd\H{o}s-Hanani conjecture about the existence of asymptotically optimal designs.
It was later discovered by many researchers (e.g.,~\cite{cooper2015list, frieze2008chromatic, frieze2013coloring, jamall2011brooks, johansson1996asymptotic, kim1995brooks}) that the R{\"o}dl nibble method is indeed a powerful tool for addressing graph/hypergraph coloring problems.
Roughly speaking, the nibble method is a probabilistic approach for constructing a combinatorial substructure, such as a proper coloring, by building it iteratively through random selection. This approach typically involves two phases:
\begin{itemize}
    \item[I.] \textbf{Iterative coloring phase}: Construct a `good' partial coloring by iteratively coloring a small portion of vertices in each round using random selection. This iterative step is referred to as a \textit{nibble}.
    \item[II.] \textbf{Finishing phase}: Extend the partial coloring from Phase I to a complete coloring.
\end{itemize}
A common challenge with this approach is to design a `smart' nibble strategy that ensures, by the end of Phase I, the remaining uncolored graph is sufficiently \textit{sparse} while each of its vertices still has many \textit{usable} colors. This enables the finishing phase to be `easily' handled, for example, using the greedy algorithm or the Lov{\'a}sz Local Lemma.


To better illustrate the challenges of hypergraph coloring problem and the novelty of our work, we begin by discussing the \textit{graph} case.  
To apply the nibble method, in each nibble (i.e., an iterative coloring step), every uncolored vertex maintains a \textit{palette} that records all the \textit{usable} colors—namely, colors that do not and potentially will not violate the coloring rules.
Clearly, once a vertex $u$ is colored, this color will no longer be \textit{usable} for all of $u$’s neighbors (to achieve a proper coloring) and will subsequently be removed from their palettes. 
We define the \textit{$c$-degree} of a vertex $u$ as the number of its neighbors whose palettes contain the color $c$, which will be the most critical graph parameter that needs to be tracked throughout the random process.
In Kim’s algorithm~\cite{kim1995brooks} for coloring girth-5 graphs, $c$-degrees can be effectively bounded after each nibble using standard concentration inequalities, due to a special \textit{independence} of girth-5 graphs: whether a color $c$ remains in the palette of one neighbor of a vertex 
$u$ has \textit{negligible} impact on the palettes of other neighbors of $u$.  
However, even for triangle-free graphs, this independence cannot be guaranteed, which makes the coloring problem more challenging. 
In fact, it was noted by Jamall in~\cite{jamall2011brooks} that when the graph has 4-cycles, $c$-degrees are not necessarily concentrated around its expectation with Kim's algorithm.

In the literature, in efforts to resolve the problem of coloring triangle-free graphs, there are essentially four different approaches:
\begin{itemize}
\item \textbf{Approach 1}: Proposed by Johansson~\cite{johansson1996asymptotic}, this approach involves modifying the nibble strategy to control the entropy of the remaining palettes, ensuring that every color in the palette is chosen nearly uniformly at each iteration. This uniformity has proven to be useful and crucial in bounding $c$-degrees.
\item \textbf{Approach 2}: Jamall~\cite{jamall2011brooks} claimed that, although each $c$-degree does not concentrate, the average $c$-degree (over all $c$ in
the palette) does concentrate. Using this insight, Jamall~\cite{jamall2011brooks} provided an alternative proof to Johansson's result (i.e., \eqref{johansson}) with a constant of 67. Pettie and Su~\cite{pettie2015distributed} further developed this approach and improved Johansson's constant from 9 to 4.


\item \textbf{Approach 3}: Despite differences in concentration details, both approaches above rely on the Lov{\'a}sz Local Lemma in their analysis of nibbles, where there is some `slackness' in its application.
Molloy~\cite{molloy2019list} and later Bernshteyn~\cite{bernshteyn2019johansson} provided dramatically simpler proofs by employing the entropy compression method and the lopsided Lov{\'a}sz Local Lemma, respectively, to replace the iterated applications of the classical Lov{\'a}sz Local Lemma. This approach allows them to improve the constant to 1.

\item \textbf{Approach 4}: Very recently, Bernshteyn, Brazelton, Cao, and Kang~\cite{bernshteyn2023counting} and, independently, Hurley and Pirot~\cite{hurley2021first}, and Martinsson~\cite{martinsson2021simplified} obtained other alternative proofs of Molloy’s bound using a recent technique developed by Rosenfeld~\cite{rosenfeld2020another}, in place of the entropy compression method.
\end{itemize}

Going back to hypergraphs, all prior work (\cite{cooper2015list, frieze2008chromatic, frieze2013coloring}) essentially employed Approach 1, i.e., Johansson’s entropy approach. 
In particular, Cooper and Mubayi~\cite{cooper2015list} extended the entropy approach to all rank $3$ hypergraphs.
However, their proof does not readily generalize to higher ranks as far as we can see, due to the increasing complexity of concentration analysis that accumulates in higher-rank hypergraphs. 
Meanwhile, it remains unclear how Approach 3 could be adapted to hypergraphs. 
As for Approach 4, Wanless and Wood~\cite{wanless2022general} provide a general framework for hypergraph coloring using Rosenfeld's technique; however, it is unclear how to embed this framework into the context of triangle-free or sparse hypergraphs.

In contrast to the aforementioned studies on hypergraphs, our work is based on Approach 2, following the ideas of Jamall and Pettie-Su, which focus on concentrating the average of $c$-degrees. Not surprisingly, when extending this approach to hypergraphs, several new challenges arise.
\begin{itemize}
\item[(i)] {\bf Tracking auxiliary hypergraphs with various uniformity.} 
The first challenge is that with higher ranks, the coloring algorithm necessarily becomes much more complicated. For example, when a vertex $u$ is colored with some color $c$, it is no longer practical to immediately remove $c$ from the palettes of its neighbors, as this could result in the loss of too many usable colors.
Instead, we need to be more cautious: for each hyperedge $e$ containing $u$ that might receive the same color $c$, we replace this edge with a new edge 
$e' =e-\{u\}$, which now has lower uniformity. We then update the coloring rule to state that vertices in 
$e'$ cannot all receive color $c$. By doing this, we keep track of potential usable colors without eliminating them prematurely. If a coloring is `good' under the new rule, it will also be `good' under the original rule.

To facilitate this, we introduce a collection of \textit{auxiliary hypergraphs} $\Hc_{c, \ell}$ at each stage of the algorithm that tracks the coloring rule associated with edges of size $\ell$ and color $c$. 
We note that essentially this means we start and maintain a hypergraph whose edges have colors (with the possibility of multiple instances of the same edge in different colors), and hence we in fact have proved the \emph{color-degree} version of Theorem~\ref{mainthm1} though we omit its statement. 
For examples of the color-degree setting, we refer the interested readers to recent works such as Alon-Assadi~\cite{alon2020palette}, Cambie–Kang~\cite{cambie2022independent}, Kang–Kelly~\cite{kang2022colorings}, Glock–Sudakov~\cite{glock2022average}, and Anderson–Bernshteyn–Dhawan~\cite{anderson2025coloring}.
This also extends to the more general form, where edges only have certain colors forbidden from being monochromatic.

\item[(ii)] {\bf Defining $c$-degrees in hypergraphs.}
Another fundamental issue is how to generalize the notion of \textit{$c$-degree} to hypergraphs given edges of various uniformity. 
A straightforward approach would be to define $c$-degree for each uniformity individually: for example, we could define the \textit{$(c, \ell)$-degree} of $u$ as its $\ell$-degree in $\Hc_{c, \ell}$.   
However, unlike in graphs, these $(c, \ell)$-degrees may not necessarily decrease during an iteration. This is because, as mentioned above, new edges of lower uniformity are introduced in addition to the removal of old edges.

One of the main contributions of our paper is the introduction of the \textit{weighted sum of $(c, \ell)$-degrees} (see \eqref{def:cdegree} for the exact formula).
This new definition serves as the role of `$c$-degree' in hypergraphs and lays the foundation for all subsequent analyses.
The weights are carefully chosen to balance the contributions from each uniformity, ensuring that the average of these $c$-degrees over colors is \textit{monotone decreasing} in expectation and well-concentrated.
We specifically remark that this weighted sum is a \emph{linear} combination of the $(c, \ell)$-degrees, rather than a sum of polynomial roots (as appeared in the statement of Theorem~\ref{mainthm1}) which might seem more natural.
This linearity is crucial not only for determining the expectation but also for proving the concentration of the new weighted sum in the next iteration of the algorithm. 

\item[(iii)] {\bf Manually reducing codegrees.}
The structural intricacy of hypergraphs also introduces more dependencies among trials and variables compared to the graph case, making concentration analysis more challenging.
We will see later in the algorithm (see Section~\ref{sec:coloringAlgDetail}) that the $c$-degrees of a vertex are determined by the colorings of its neighbors and second neighbors. Higher codegrees in hypergraphs introduce additional dependencies among variables, which may prevent the average $c$-degrees from being well-concentrated.

Unlike Cooper and Mubayi~\cite{cooper2015list} who tracked all codegrees of hypergraphs, another novelty of our work is the development of a \textit{codegree reduction algorithm} (see Section~\ref{sec:coredu} for details), which directly reduces all codegrees by contracting multiple hyperedges into one hyperedge of smaller uniformity.
This reduction process obviates the need to track codegrees---which becomes increasingly sophisticated with higher rank---while preserving all \textit{essential} coloring information and properties.

\item[(iv)] {\bf Using a new concentration technique.} Addressing technical difficulties in concentration analysis is a core component of our work. To this end, instead of relying on classical concentration tools, we use a linear version of Talagrand's concentration inequality with exceptional events developed in Delcourt and Postle~\cite{DP} (Theorem~\ref{talagrand}). This new version crucially provides a linear (as opposed to quadratic) dependence on the so-called Lipschitz constant under certain assumptions, allowing us to establish concentration for variables with larger Lipschitz constants, in contrast to classical Talagrand's concentration inequality.
Moreover, our application of this technique is both intricate and innovative, including breaking target random variables into several variables for which Theorem~\ref{talagrand} are applicable, and constructing the set of exceptional events through iterative applications of Theorem~\ref{talagrand}. 

\item[(v)] {\bf Establishing `almost independence'.}
Similar to the graph case, the triangle-free condition alone does not guarantee the independence of random coloring events, which is important for bounding $c$-degrees in expectation. However, by using Janson's Inequality (Theorem~\ref{Janson}), we demonstrate that triangle-freeness is sufficient to ensure a certain degree of `almost independence' (see Lemma~\ref{lem:dep}), which is enough for our purpose. 

We also point out that while the triangle-freeness is crucial for reducing the chromatic number of hypergraphs, the vast majority of our proof does not rely on this condition.
The only place we use the triangle-freeness is in Section~\ref{sec:dep}, where it establishes the `almost independence' during the algorithm. In other words, the triangle-free assumption in Theorem~\ref{mainthm1} could be substituted with any condition that ensures the conclusion of Lemma~\ref{lem:dep}.
\end{itemize}

Altogether, it is the combination of the appropriate definitions (e.g., colored edges, weighted $c$-degrees) and innovative techniques (e.g., codegree reduction, the new version of Talagrand’s inequality) that leads to our proof of Theorem~\ref{mainthm1}.

\medskip
In the next section, we present some related probabilistic tools.
In Section~\ref{sec:algo}, we describe our codegree reduction algorithm and the main coloring algorithm.
Section~\ref{sec:mainthm2} contains an analysis of our coloring algorithm; in particular, we state our Key Lemma (Lemma~\ref{lem:de}) and show how it is used to prove Theorem~\ref{mainthm2}.
Sections~\ref{sec:mainlemma}, ~\ref{sec:expe} and~\ref{sec:concede} are devoted to proving Lemma~\ref{lem:de}.
We then derive Theorem~\ref{mainthm1} from Theorem~\ref{mainthm2} in Section~\ref{sec:mainthm1}, and prove Theorem~\ref{sparsethm} in Section~\ref{sec:spars}.
Finally, we conclude the paper with some open problems in Section~\ref{sec:conclu}.

\section{Probabilistic tools}\label{sec:prob}
\subsection{The Lov{\'a}sz Local Lemma}
\begin{thm}[The Asymmetric Local Lemma~\cite{molloy2002graph}]\label{locallemma}
Consider a set $\mathcal{E}=\{\A_1, \ldots, \A_n\}$ of (typically bad) events such that each $\A_i$ is mutually independent of $\mathcal{E} -(\mathcal{D}_i  \cup \A_i)$, for some $\mathcal{D}_i \subset \mathcal{E}$.
If for each $1 \leq i \leq n$
\begin{itemize}
\item $\Prb(\A_i) \leq 1/4$, and
\item $\sum_{\A_j\in \mathcal{D}_i}\Prb(\A_j) \leq 1/4$,
\end{itemize}
then with positive probability, none of the events in $\mathcal{E}$ occur.
\end{thm}

\subsection{Concentration inequalities}
\label{sec:coneq}
%
One of the key tools in our proof is a linear version of Talagrand’s concentration inequality (Theorem~\ref{talagrand}) from the recent work of Delcourt and Postle~\cite{DP}. However, we do not require the full strength of their concentration inequality. Instead, we present the following special case (which is more user-friendly), where the target random variable can be interpreted as a sum of $\{0, 1\}$ random variables.
Before proceeding, we first introduce some definitions.

\begin{defi}[$r$-verifiable]\label{def:verfiable}
Let $\{(\Omega_i, \Sigma_i, \mathbb{P}_i)\}^{n}_{i=1}$ be probability spaces, $(\Omega, \Sigma, \mathbb{P})$ be their product space, $\Omega^*\subseteq \Omega$ be a set of \textit{exceptional outcomes}, and $\Y:\Omega \rightarrow \{0, 1\}$ be a $\{0, 1\}$-random variable.
Let $r\geq 0$.
We say $\Y$ is \textit{$r$-verifiable} with verifier $R:\{\omega\in\Omega\setminus\Omega^*: \Y(\omega)=1\}\rightarrow 2^{[n]}$ with respect to $\Omega^*$ if
\begin{itemize}
\item $|R(\omega)|\leq r$ for every $\omega\in\Omega\setminus\Omega^*$ with $\Y(\omega)=1$, and
\item $\Y(\omega')=1$ for all $\omega'=(\omega'_1, \ldots, \omega'_n)\in \Omega\setminus\Omega^*$ such that $\omega_i=\omega'_i$ for each $i\in R(\omega)$.
\end{itemize}
\end{defi}

\begin{defi}[$(r, d)$-observable]\label{def:obser}
Let $\{(\Omega_i, \Sigma_i, \mathbb{P}_i)\}^{n}_{i=1}$ be probability spaces, $(\Omega, \Sigma, \mathbb{P})$ be their product space, and $\Omega^*\subseteq \Omega$ be a set of \textit{exceptional outcomes}.
Let $r, d\geq 0$.
We say a random variable $\X$ in $\Omega$ is \textit{$(r, d)$-observable} with respect to $\Omega^*$ if
\begin{itemize}
\item  \[
\X=\sum_{j=1}^{m}\Y_j,
\]
where for every $j\in[m]$, $\Y_j$ is a $\{0, 1\}$-random variable in $\Omega$ that is $r$-verifiable with verifiers $R_j$, and
\item  for  every $\omega \in \Omega\setminus\Omega^*$ and $i\in [n]$,
\[
|\{j\in [m]:\ i\in R_j(\omega) \ \text{and} \ \Y_j(\omega)=1\}|\leq d.
\]
\end{itemize}
\end{defi}

Now we state their concentration inequality, as follows. 
\begin{thm}[\text{Delcourt-Postle~\cite[Theorem~4.4]{DP}}]
\label{talagrand}
Let $\{(\Omega_i, \Sigma_i, \mathbb{P}_i)\}^{n}_{i=1}$ be probability spaces, $(\Omega, \Sigma, \mathbb{P})$ be their product space, and $\Omega^*\subseteq \Omega$ be a set of exceptional outcomes. 
Let $r, d\geq 0$, and $\X: \Omega \rightarrow \mathbb{R}_{\geq 0}$ be a non-negative random variable.
If $\X$ is $(r, d)$-observable with respect to $\Omega^*$, then for any $\tau>96\sqrt{rd\E[\X]} + 128rd + 8\Prb[\Omega^*](\sup \X)$,
\[
\Prb(|\X - \E[\X]|>\tau) \leq 4\exp\left(-\frac{\tau^2}{8rd(4\E[\X] + \tau)}\right) + 4\Prb(\Omega^*).
\]
\end{thm}
We also need the following classical Chernoff bound (see~\cite{alon2016probabilistic}).
\begin{lemma}[Chernoff bound]\label{chernoff}
Let $\X_1, \ldots, \X_n$ be independent $\{0, 1\}$-random variables such that $\Prb(\X_i=1)=p$.
Let $\X=\sum_i\X_i$.
Then
\begin{itemize}
\item Upper tail: 
$\Prb(\X \geq (1  +\delta)\E[\X]) \leq \exp\left(-\delta^2\E[\X]/(2 + \delta)\right)$ for all $\delta >0$;
\item Lower tail: 
$\Prb(\X \leq (1  -\delta)\E[\X]) \leq \exp\left(-\delta^2\E[\X]/2\right)$ for all $0<\delta <1$;
\end{itemize}
\end{lemma}
Lastly, we present a simple but useful proposition about conditional probability.
\begin{prop}\label{lem:prb}
For any two events $\A$ and $\B$,
\[
\Prb(\A)\leq \Prb(\A\mid \B) + \Prb(\overline{\B}).
\]
\end{prop}
\begin{proof}
\[
\Prb(\A)= \Prb(\A\mid \B)\cdot\Prb(\B)+ \Prb(\A\mid \overline{\B})\cdot\Prb(\overline{\B})
\leq \Prb(\A\mid \B) + \Prb(\overline{\B}).
\]
\end{proof}

\subsection{Janson's Inequality}
Let $\Omega$ be a finite universal set and let $\mathbf{R}$ be a random subset of $\Omega$ obtained by choosing each element $v\in\Omega$ independently with
\[
\Prb (v\in \mathbf{R}) = p_v.
\]
Let $\{A_i\}_{i\in I}$ be subsets of $\Omega$, where $I$ is a finite index set.
Let $\A_{i}$ be the event $A_i\subseteq \mathbf{R}$. 
(That is, each point $v\in\Omega$ ``flip a coin" to determine if it is in $\mathbf{R}$, and $\A_i$ is the event that the coins for all $v\in A_i$ came up ``heads".)

For $i, j\in I$ we write $i \sim j$ if $i\neq j$ and $A_i\cap A_j\neq \emptyset$.
Note that when $i\neq j$ and not $i\sim j$, then $\A_i$, $\A_j$ are independent events. We define
\begin{equation}\label{Janson:delta}
\Delta^*:=\sum_{i \sim j}\Prb(\A_i\wedge\A_j),
\end{equation}
where the sum is over ordered pairs $(i, j)$.
We set
\begin{equation}\label{Janson:M}
M:=\prod_{i\in I}\Prb(\overline{\A_i}).
\end{equation}
The following result was given by Janson, {\L}uczak and Ruci{\'n}ski~\cite{janson1988exponential}.
\begin{thm}[Janson's Inequality]\label{Janson}
Let $\{\A_i\}_{i\in I}$, $\Delta^*$, $M$ be as above and assume that there is an $\varepsilon>0$ so that $\Prb(\A_i)\leq \varepsilon$ for all $i\in I$. Then
\[
M\leq \Prb\left(\bigwedge_{i\in I}\overline{\A_i}\right) \leq M\exp\left(\frac{1}{1-\varepsilon}\frac{\Delta^*}{2}\right).
\]

\end{thm}

\subsection{A correlation inequality}
We also use the following correlation inequality from~\cite{alon2016probabilistic}.
Let $p=(p_1, \ldots, p_n)$ be a real vector, where $0\leq p_i \leq 1$. Consider the probability space whose element are all members of the power set $\mathcal{P}(N)$, where, for each $A\subseteq N$, $\Prb(A)=\prod_{i\in A}p_i\prod _{j\notin A}(1-p_j)$.
Clearly this probability distribution is obtained if we choose a random $A\subseteq N$ by choosing each element $i\in N$ independently with probability $p_i$.
For each $\A \subseteq \mathcal{P}(N)$, Let us denote by $\Prb_p[\A]$ the probability that a randomly chosen subset of $N$ lies in $\A$, i.e., $\sum_{A\in\A}\Prb(A)$. 

A family $\A$ of subsets of $N$ is \textit{monotone decreasing} if $A \in \A$ and $A'\subseteq\A$ indicates $A'\in\A$. Similarly, it is \textit{monotone increasing} if $A\in\A$ and $A\subseteq A'$ indicates $A'\in\A$. 
\begin{thm}\cite[Theorem 6.3.2]{alon2016probabilistic}\label{FKG}
Let $\A$ and $\B$ be two monotone increasing families of subsets of $N$ and let $\C$ and $\D$ be two monotone decreasing families of subsets of $N$. Then, for any real vector $p=(p_1, \ldots, p_n)$, $0\leq p_i \leq 1$,
\begin{align*} 
\Prb_p(\A\cap\B) \geq \Prb_p(\A)\cdot\Prb_p(\B), \\ 
\Prb_p(\C\cap\D) \geq \Prb_p(\C)\cdot\Prb_p(\D), \\ 
\Prb_p(\A\cap\C) \leq \Prb_p(\A)\cdot\Prb_p(\C).
\end{align*}
\end{thm}

\section{Coloring algorithm}\label{sec:algo}
In this section, we introduce our coloring algorithm for the iterative coloring phase of the nibble method, which will be applied to prove Theorem~\ref{mainthm2}.
We begin this section by establishing our basic setup. 
After that, in Section~\ref{sec:alg:warm}, we present a `baby version' of our algorithm, focusing on the key coloring steps while omitting technical details, to help readers gain some understanding of the overall approach.
In Section~\ref{sec:coredu}, we introduce our \textit{codegree reduction algorithm}, which will serve as a subroutine at each
step of our coloring algorithm. 
In Section~\ref{sec:coloringAlgDetail}, we provide a detailed description of our main coloring algorithm. 
Finally in Section~\ref{sec:notation}, we summarize all the notation and parameters used throughout the algorithm and gather some simple facts for later analysis.
\medskip

The input to our coloring algorithm is a rank $k$, triangle-free hypergraph $\Hc$, that satisfies all the degree and codegree assumptions in Theorem~\ref{mainthm2}.
Let $\mathcal{P}=\{P(u)\}_{u\in V(\Hc)}$ be a list assignment of $\Hc$, where each $P(u)$ represents the initial color \textit{palette}, i.e., the set of usable colors for $u$.
Let 
\[
\Cco:=\cde^{-1}(\Delta/\log\Delta)^{1/(k-1)}
\]
be the number of colors in the initial palettes, where $\cde:=1/(60\cdot2^k)$ is a small constant (chosen to meet the requirements of the analysis).
The goal is to show that for any list assignment $\mathcal{P}$ with $|\Pa(u)|= \Cco$ for all vertices $u$, our coloring algorithm always generates a proper partial coloring of $\Hc$, which can then be easily extended to a proper coloring of $\Hc$.

\begin{table}[H]
\centering
\begin{tabular}{lll}
\hline
\hline
Notation & Value & Description \\
\hline
$k$ & $\ge 3$ & rank \\
$\Hc$ &  & rank $k$, triangle-free hypergraph \\
$\Delta_{\ell}(\Hc)$ &  $\leq \Delta^{1 - \frac{k - \ell}{k-1}} (\log \Delta)^{\frac{k - \ell}{ k-1}}$ & maximum $\ell$-degree\\
$\delta_{s, \ell}(\Hc)$ &  $\leq(\Delta/\log\Delta)^{\frac{\ell - s}{k-1}}$ & maximum codegree\\
$\Pa(u)$ &  & Initial color palette of vertex $u$\\
$\Cco$ & $\cde^{-1}(\Delta/\log\Delta)^{1/(k-1)}$ & number of colors in the palette\\
$\cde$  & $1/(60\cdot2^k)$  & constant\\
\hline
\hline
\end{tabular}
\caption{Basic hypergraph parameters}
\label{table:basic}
\end{table}


\subsection{Warm-up}\label{sec:alg:warm}

As mentioned in Section~\ref{sec:overview}, our coloring algorithm is based on the work of Pettie and Su~\cite{pettie2015distributed} for triangle-free graphs, while incorporating several important modifications to generalize the method to hypergraphs. Before we proceed, we first introduce the following definition.

\begin{defi}[Compatible coloring]
Let $V$ be a set, $\{\Pa(u)\}_{u\in V}$ be a collection of color palettes for $V$, and $\{\Fc_{c, \ell}\}_{c, 2\le \ell\le k}$ be a family of hypergraphs defined on $V$, where each $\Fc_{c, \ell}$ is an $\ell$-uniform hypergraph with color index $c$. 

We say a (partial) coloring $\Psi$ on $V$ is \textit{$(\{\Pa(u)\},~\{\Fc_{c, \ell}\})$-compatible}, if for every $u$, $\Psi(u)\in \Pa(u)$, and none of the edges in 
$\Fc_{c, \ell}$ are monochromatic with color $c$ under 
$\Psi$.

\end{defi}

Intuitively, in each iteration, we partially color some vertices to obtain a larger partial coloring $\Psi$. We then update the color palettes $\Pa(u)$ for remaining vertices (say $U$).
Most importantly, we establish and track a collection of restriction hypergraphs $\Hc_{c, \ell}$ with uniformity $\ell$ and color index $c$, which are used to capture all the coloring restrictions needed to achieve a proper coloring of $\Hc$, in the sense that
\[
\text{any extension of $\Psi$ will be a proper coloring of $\Hc$, if it is $\left(\{\Pa(u)\}, \{\Hc_{c,\ell}\}\right)$-compatible.}
\]
Here is a rough outline of one step in our random iterative coloring procedure:
\begin{itemize}
    \item For each uncolored vertex $u$, \textit{activate} each color in its current palette $\Pa(u)$ independently with some small probability $p$. 
        \item For a vertex $u$, a color $c$ will be considered \textit{lost} from its palette $\Pa(u)$, if there is an edge in $\cup_{\ell}\Hc_{c, \ell}$ containing $u$, such that $c$ is activated on all other vertices in the edge.
    \item \textit{Color} vertex $u$ with any color from $\Pa(u)$ that is activated and not lost.
\end{itemize}
Note that the $\left(\{\Pa(u)\}, \{\Hc_{c,\ell}\}\right)$-compatibility guarantees that the updated partial coloring remains proper, as none of the edges can have a color $c$ that is both activated and not lost on all of its vertices.
We then update the palette to include only the remaining colors in 
$\Pa(u)$—those that are not lost—and adjust the restriction hypergraphs as follows:
\[
\begin{split}
\Hc_{\ell, c}\rightarrow &\left\{e\in\Hc_{\ell, c}:\ e\subseteq U,\ c\in \Pa(u) \text{ for all } u\in e\right\} \\
 &+ \left\{S\subseteq \binom{U}{\ell}: S\subsetneq e\in \bigcup_{s> \ell}\Hc_{s, c},\ c\in \bigcap_{u\in S}\Pa(u)\ \& ~\Psi(v)=c \text{ for all } v\in e\setminus S\right\}.
 \end{split}
\]
The motivation for defining $\Hc_{c, \ell}$ as described is to ensure that, if none of the edges in the updated $\Hc_{c,\ell}$ are monochromatic with $c$, then the same holds for the previous $\Hc_{c,\ell}$. This crucial property maintains compatibility and allows us to track all coloring restrictions by focusing solely on the uncolored vertices.

In practice, to better control the algorithm, we need to make several \textit{fine-tunings} to our palettes and restriction hypergraphs, such as reducing codegrees of hypergraphs and filtering more colors than necessary. These details will be described in Section 4.3.

\subsection{Codegree reduction algorithm}\label{sec:coredu}
Let $\Fc$ be a rank $k$ hypergraph and let $f: \mathbb{Z} \times \mathbb{Z} \rightarrow \mathbb{Z}$ be a function. 
The purpose of this algorithm is to generate a new hypergraph $\Fc'$ from $\Fc$ with $V(\Fc')=V(\Fc)$ and smaller codegrees (in terms of $f$), such that any proper coloring of $\Fc'$ is still proper for $\Fc$.
\newline

\noindent\textbf{Codegree reduction algorithm.}
The input to the algorithm is a rank $k$ hypergraph $\Fc$ and a function $f: \mathbb{Z} \times \mathbb{Z} \rightarrow \mathbb{Z}$. We start the algorithm with $\Fc^0:=\Fc$. 
\begin{itemize}
\item[1.] In the $i$-th iteration round, for every vertex $u$ and $2\leq k-i < \ell \leq k$, let 
\[
F_{k-i, \ell}(u):=\left\{ S\subseteq V:\ |S|=k-i,\ u\in S,~\text{and}\  \deg_{\ell}(S, \Fc^{i-1}) \geq f(k-i, \ell)\right\}.
\]
We then define $\Fc^i$ as follows: let $V(\Fc^i) = V(\Fc)$ and 
\[
E(\Fc^i) := E(\Fc^{i-1}) - \bigcup_u\bigcup_{\ell>k-i}\bigcup_{S\in F_{k-i, \ell}(u)}\{e\in \Fc^{i-1}:\  e \supseteq S,\ |e|=\ell \} + \bigcup_u\bigcup_{\ell>k-i}F_{k-i, \ell}(u),
\]
and move to the next iteration. 
\item[2.] We stop the algorithm after $k-2$ steps, and set the output as $\Fc'=\Fc^{k-2}$. 
\end{itemize}

Intuitively, if a set $S$ in $\Fc$ is contained in too many $\ell$-edges, we simply delete all $\ell$-edges containing $S$ and add $S$ itself as a new edge (now with lower uniformity). 
It is important to note that in each iteration, we only `repair' codegrees $\delta_{s, \ell}$ for a given $s$. The ordering in which we run the algorithm (i.e., iterating as $s$ decreases) ensures that the codegrees we have already `repaired' will not increase in subsequent steps.

\begin{defi}[$f$-reduction]
For two hypergraphs $\Fc,\ \Fc'$, and a function $f: \mathbb{Z}\times\mathbb{Z}\rightarrow \mathbb{Z}$, we say $\Fc'$ is an \textit{$f$-reduction} of $\Fc$, if $\Fc'$ is generated from the \textit{codegree reduction algorithm} with the input $\{\Fc, f\}$.
\end{defi}
\begin{prop}\label{prop:redu}
Let $\Fc$ be a rank $k$ hypergraph and let $f: \mathbb{Z}\times\mathbb{Z}\rightarrow \mathbb{Z}$ be a function such that $f(\ell, \ell)=1$ for every $\ell$, and 
\[f(s_1, \ell) < f(s_2, \ell) \quad \text{if $s_1> s_2$}.\]
Then the $f$-reduction $\Fc'$ of $\Fc$ satisfies the following properties:
\begin{itemize}
\item[(1)] For every $2\leq s < \ell \leq k$, 
$\delta_{s, \ell}(\Fc')\leq f(s, \ell);$
\item[(2)] Any proper coloring of $\Fc'$ is also proper for $\Fc$;

\item[(3)] If $\Fc$ is triangle-free, then $\Fc'$ is also triangle-free.
\end{itemize}
\end{prop}
\begin{proof}
Property (1) follows directly from the construction of the $f$-reduction.
For property (2), let $\Psi$ be a proper coloring of $\Fc'$, i.e., every edge in $\Fc'$ uses at least two colors under $\Psi$. 
By the construction of $\Fc'$, for every edge $e\in\Fc$, there exists some $e'\in \Fc'$ such that $e'\subseteq e$. This immediately implies that $\Psi$ is also a proper coloring of $\Fc$.

For property (3), suppose that $\Fc$ is triangle-free. Let $\Fc^0, \ldots, \Fc^{k-2}$ be the hypergraphs generated by our reduction algorithm, and recall that $\Fc=\Fc^0$ and $\Fc'=\Fc^{k-2}$. 
To show that $\Fc'$ is triangle-free, we will prove by induction on $i$ that each $\Fc^i$ is triangle-free.

The base case $i=0$ is trivially true as $\Fc^0=\Fc$. 
Assume by induction that $\Fc^{i-1}$ is triangle-free.
Suppose, for the sake of contradiction, that there exists a triangle $\{e, f, g\}$ in $\Fc^{i}$ with vertices $u, v, w$ such that $\{u, v\}\subseteq e$, $\{v, w\}\subseteq f$, $\{w, u\}\subseteq g$ and $\{u, v, w\}\cap e\cap f\cap g=\emptyset$. In particular, we have $w\notin e$.
Triangle-freeness then indicates that a least one of these three edges is not in $\Fc^{i-1}$. To simplify the presentation, we further assume that $e$ is the only edge that is not in $\Fc^{i-1}$; other cases follow by applying the same argument to each missing edge, which we omit here.

Since $e\in\Fc^{i}\setminus \Fc^{i-1}$, we have $|e|=k-i$, and moreover, there exists an integer $\ell > k-i$ such that there are at least $f(k-i, \ell)$ edges in $\Fc^{i-1}$ containing $e$. 
On the other hand, the construction of $\Fc^{i-1}$ ensures that $\deg_{\ell}(e\cup w, \Fc^{i-1}) \le \delta_{k-i+1, \ell}(\Fc^{i-1}) \leq f(k-i+1, \ell)$, which is strictly less than $f(k-i, \ell)$ by the definition of $f$.
Therefore, there must be at least one edge in $\Fc^{i-1}$, say $e'$, such that $e'\supseteq e$ and $w\notin e'$. Then $\{e', f, g\}$ forms a triangle in $\Fc^{i-1}$, which contradicts our induction assumption. 
\end{proof}

\subsection{Coloring algorithm}\label{sec:coloringAlgDetail}
We now describe our main coloring algorithm as follows.

\subsubsection{Initial set-up of the algorithm.} \label{sec:alg:initial}
Denote by $U_0$ the set of \textit{uncolored vertices}, which initially equals $V(\Hc)$, and let $\Psi_0$ be a partial coloring of $\Hc$ with support set $V(\Hc)\setminus U_0$ (which is currently empty and thus vacuously proper).
For every vertex $u\in U_0$, let $\Pa_{0}(u):=\Pa(u)$, the \textit{color palette} of $u$ at the start of the algorithm. Our first crucial parameter is set as
\[
\pa_0:=\Cco,\]
which intuitively refers to the \textit{ideal size of the palettes}. In particular, for all $u\in U_0$,
\[
|P_0(u)| =C \ge (1 - o(1))p_0.
\] 
For any color $c$ and $2\leq \ell \leq k-1$, let 
\[
\Hc^0_{\ell, c}:=\{e\in\Hc :\ |e|=\ell, \text{ and } c\in \Pa_{0}(u) \text{ for all } u\in e\}.
\] 
As mentioned in Section~\ref{sec:alg:warm}, the hypergraph $\Hc^0_{\ell, c}$ intuitively captures all the coloring restrictions of size $\ell$ associated with the color $c$, such that \[
\text{any $\left(\{\Pa_0(u)\}, \{\Hc^0_{c, \ell}\}\right)$-compatible coloring on $V(\Hc)$ is a proper coloring of $\Hc$.}\]

For a vertex $u\in U_0$ and a color $c$, the \textit{$c$-degree of $u$} (at the start of the algorithm) is defined as
\[
d_0(u, c):=\sum_{\ell=2}^k(\cde \pa_{0})^{k-\ell} d^{0}_{\ell}(u, c),
\]
where $d^0_{\ell}(u, c):=\deg(u,\Hc^0_{\ell, c})$.
We then set the second crucial parameter \[t_0:=(k-1)\Delta,
\]
which intuitively refers to the \textit{ideal bound of $c$-degrees}.
Note that, by the choice of parameters (see Table~\ref{table:basic}), for all $u\in U_0$ and $c\in \Pa_0(u)$,
\begin{equation}\label{cde:zero}
d_0(u, c)\leq \sum_{\ell=2}^k(\cde \Cco)^{k-\ell}\Delta_{\ell}(\Hc)
\leq \sum_{\ell=2}^k(\cde \Cco)^{k-\ell}\Delta^{1 - \frac{k - \ell}{k-1}} (\log \Delta)^{\frac{k - \ell}{ k-1}}=\sum_{\ell=2}^k \Delta =t_0 \le 2t_0,
\end{equation}
and
\[
\delta_{s, \ell}(\Hc^{0}_{\ell, c}) \leq \delta_{s, \ell}(\Hc) \leq \left(\frac{\Delta}{\log\Delta}\right)^{\frac{\ell-s}{k-1}}=(\cde\pa_0)^{\ell-s}.
\]

$\{U_0, \Psi_0, \{\Pa_{0}(u)\}, p_0, \{\Hc^0_{\ell,c}\}, \{d_0(u, c)\}, t_0\}$ are the key parameters that we will track and update as the algorithm progresses. In particular, the ratio of the ideal palette size to the ideal $c$-degree bound (see the exact definition in~\eqref{def:degPalRadio}) controls the algorithm's progress. 
Our ultimate goal is to generate a proper partial coloring of $\Hc$ while decreasing this ratio throughout the algorithm, as a small ratio is key to carrying out the finishing phase of the nibble method.

\subsubsection{Iteration of the algorithm.}\label{sec:alg:itera}
At the beginning of the $\rd$-th iteration round, we are given: 
\begin{itemize}
\item[(1)] \textbf{Uncolored vertices}: $U_{\rd-1}$,
\item[(2)] \textbf{Proper partial coloring}: $\Psi_{i-1}$ with support set $V(\Hc)\setminus U_{\rd-1}$,
\item[(3)] \textbf{Color palettes}: $\Pa_{\rd-1}(u)$ for every $u\in U_{i-1}$,
\item[(4)] \textbf{Ideal palette size}: $\pa_{\rd-1}$,
\item[(5)] \textbf{Restriction hypergraphs}: $\Hc^{\rd-1}_{\ell, c}$, defined on the vertex set $U_i$, for every $2\le \ell \le k$ and color $c$, 
\item[(6)] \textbf{$c$-degrees}: for every $u\in U_{i-1}$ and $c\in \Pa_{\rd-1}(u)$, the \textit{$c$-degree of $u$} (at the $(i-1)$-iteration) is defined as follows:
\begin{equation}\label{def:cdegree}
    d_{\rd-1}(u, c):=\sum_{\ell=2}^k(\cde \pa_{\rd-1})^{k-\ell} d^{\rd-1}_{\ell}(u, c)
\end{equation}
where $d^{\rd-1}_{\ell}(u, c):=\deg_{\ell}(u, \Hc^{\rd-1}_{\ell, c})$,
\item[(7)] \textbf{Ideal bound of $c$-degrees}: $t_{\rd-1}$,
\end{itemize}
which satisfy the following induction assumptions:
\begin{equation}\label{ite:asump1}
d_{\rd-1}(u, c) \leq 2t_{\rd-1}, \quad \text{for all }u\in U_{\rd-1} \text{ and } c\in \Pa_{\rd-1}(u),
\end{equation}
and
\begin{equation}\label{ite:asumpcode}
\delta_{s, \ell}(\Hc^{\rd-1}_{\ell, c})\leq (\cde\pa_{\rd-1})^{\ell-s} \quad \text{for all } c.
\end{equation}
Observe that \eqref{ite:asump1} further implies 
\begin{equation}\label{ite:asump2}
d^{\rd-1}_{\ell}(u, c)\leq 2t_{\rd-1}/(\cde \pa_{\rd-1})^{k-\ell} \quad \text{for each } \ell.
\end{equation}
Again, these $\Hc^{\rd-1}_{\ell, c}$ hypergraphs are designed to capture all the \textit{crucial} coloring restrictions on the remaining vertices, in the sense that
\begin{equation}\label{ite:asumpRH}
\text{any extension of $\Psi_{i-1}$ will be a proper coloring of $\Hc$, if it is $\left(\{\Pa_{i-1}(u)\}, \{\Hc^{i-1}_{c,\ell}\}\right)$-compatible.}
\end{equation}

The iterative algorithm proceeds through the following nine steps:
\begin{enumerate}
\item \textbf{Activate colors}: For every vertex $u\in U_{\rd-1}$, independently \textit{activate} each of the color $c\in \Pa_{\rd-1}(u)$ with probability \[\pi_{\rd}:=\cac\frac{(\cde\pa_{\rd-1})^{k-2}}{4t_{\rd-1}}, \]
where $\cac:=1/(60\cdot k^4\cdot 2^k)$ is a chosen small constant.
For formal notation, let
\begin{equation}\label{def:act}
\ac^{\rd}_{u, c}:=\begin{cases}
1, & \text{ if } c \text{ is activated on } u \\
0, & \text{ otherwise, }
\end{cases}
\end{equation}
denote the indicator random variable for this activation operation, and define the \textit{set of activated colors} of $u$ as \[\Ac_{\rd}(u):=\{c\in\Pa_{\rd-1}(u): \ac^{\rd}_{u, c}=1\}.\]

\item \textbf{Remove potentially unusable colors.} We say a color $c$ is \textit{lost} at a vertex $u$, if there exists an edge $e\in\cup_{\ell\geq 2}\Hc^{\rd-1}_{\ell, c}$ s.t. 
$u\in e$ and $\ac^{\rd}_{u, c}=1$ for all $v\in e\setminus\{u\}$. 
Denote by $\Lo_{\rd}(u)$ the \textit{set of lost colors} of $u$, and by $\nlo^{\rd}_{u, c}$ the probability that a color $c$ is \textit{not} lost at $u$. Note that by the union bound,
\[
\begin{split}
\nlo^{\rd}_{u, c}&= \Prb(c\notin\Lo_{\rd}(u))
\geq 1 - \sum_{\ell=2}^{k}d^{\rd -1}_{\ell}(u, c)\pi_{\rd}^{\ell-1}
\geq 1 - \sum_{\ell=2}^{k}\frac{2t_{\rd-1}}{(\cde\pa_{\rd-1})^{k-\ell}}\pi_{\rd}^{\ell-1}\\
&= 1- \frac{2t_{\rd-1}}{(\cde\pa_{\rd-1})^{k-1}}\sum_{\ell=2}^{k}(\cde\pa_{\rd-1}\pi_{\rd})^{\ell-1}
\geq 1- \frac{2t_{\rd-1}}{(\cde\pa_{\rd-1})^{k-1}}2\cde\pa_{\rd-1}\pi_{\rd} = 1- \cac,
\end{split}
\]
where the second inequality follows from the induction hypothesis~\eqref{ite:asump2}, and the last inequality uses a fact that $\cde\pa_{\rd-1}\pi_{\rd} \le 1/2$, which will follow from the termination condition of the algorithm (see later in~\eqref{indas4} for details).
\item \textbf{Temporary color palettes.}
For ease of notation, let
\[
\beta:=1 - \cac,
\]
and for every vertex $u\in U_{\rd-1}$, independently \textit{select} each color $c\in \Pa_{\rd-1}(u)$ with probability $\beta/\nlo^{\rd}_{u, c}$.
For formal notation, define the indicator variable
\begin{equation}\label{def:sel}
\ke^{\rd}_{u, c}=\begin{cases}
1, & \text{ if } c \text{ is selected on } u \\
0, & \text{ otherwise, }
\end{cases}
\end{equation}
and let the \textit{set of selected colors} of $u$ be $\Ke_{\rd}(u):=\{c\in\Pa_{\rd-1}(u): \ke^{\rd}_{u, c}=1\}$.
Then, for every vertex $u\in U_{\rd-1}$, the \textit{temporary palette} of $u$ at the $i$-th iteration is defined as
\[\nPa_{\rd}(u):=\Ke_{\rd}(u)\setminus \Lo_{\rd}(u).\] 
Note that for every $c\in\Pa_{i-1}(u)$,
\begin{equation}\label{eq:ncolor}
\Prb(c\in\nPa_{\rd}(u))
=\Prb(c\notin\Lo_{\rd}(u))\Prb(\ke^{\rd}_{u, c}=1)
=\beta,
\end{equation}
and therefore $\E[\nPa_{\rd}(u)] = \beta|\Pa_{\rd-1}(u)|$.
While this additional selection might seem wasteful, as it removes colors that do not cause conflicts in the proper coloring, it establishes more `uniform-sized' palettes across all vertices, which simplifies the analysis of the algorithm.
\item \textbf{Update partial coloring.} 
For a vertex $u\in U_{i-1}$, we \textit{permanently color} $u$ with any color, say $u_c$, from $\Ac_{\rd}(u)\cap \nPa_{\rd}(u)$, if $\Ac_{\rd}(u)\cap \nPa_{\rd}(u) \neq \emptyset$.
Once a vertex is permanently colored, it is immediately removed from $U_{\rd-1}$.
Let $U_{\rd}$ be the set of remaining vertices in $U_{\rd-1}$, i.e., 
\[
U_{\rd}:=\{u\in U_{\rd-1}:~\Ac_{\rd}(u)\cap \nPa_{\rd}(u) = \emptyset\}.
\]
We then obtain a new partial coloring $\Psi_i$ of $\Hc$ with support set $V(\Hc)\setminus U_i$, where
\[
\Psi_i(u):=\begin{cases}
u_c, & \text{ for } u\in U_{i-1}\setminus U_{i}; \\
\Psi_{i-1}(u), & \text{ for } u\in V(\Hc)\setminus U_{i-1}.
\end{cases}
\]
Note that none of edges $e$ in $\Hc^{i-1}_{\ell, c}$ are monochromatic with color $c$ under $\Psi_i$.
If it did, it would mean all vertices in $e$ have color $c$ \textit{activated}, but then according to step 2, color $c$ must be \textit{lost} at each vertex $u$ in 
$e$, making it unavailable in $\nPa_{\rd}(u)$.
This, along with \eqref{ite:asumpRH}, indicates that $\Psi_i$ is a proper partial coloring.

\item \textbf{Temporary restriction hypergraphs.}
For every color $c$ and $\ell$, define the \textit{temporary $(c, \ell)$-restriction hypergraph} as
\[
\begin{split}
\hat{\Hc}^{\rd}_{\ell, c}:=&\left\{e\in\Hc^{\rd-1}_{\ell, c}:\ e\subseteq U_{\rd},\ c\in \nPa_{\rd}(u) \text{ for all } u\in e\right\} \\
 &+ \left\{S\subseteq \binom{U_{\rd}}{\ell}: S\subsetneq e\in \bigcup_{s> \ell}\Hc^{\rd-1}_{s, c},\ c\in \bigcap_{u\in S}\nPa_{\rd}(u)\ \&~\Psi_i(v)=c \text{ for all } v\in e\setminus S\right\}.
 \end{split}
\]
Observe that, by the definition of $\hat{\Hc}^{\rd}_{\ell, c}$ and the fact that $\nPa_{\rd}(u) \subseteq \Pa_{\rd-1}(u)$, any $(\{\nPa_{\rd}(u)\}, \{\hat{\Hc}^{\rd}_{\ell, c}\})$-compatible coloring on $U_{\rd}$, together with $\Psi_{\rd}\setminus \Psi_{\rd-1}$, forms a $(\{\Pa_{\rd-1}(u)\}, \{\Hc^{\rd-1}_{\ell, c}\})$-compatible coloring on $U_{\rd-1}$.
This, along with \eqref{ite:asumpRH}, shows that our $\hat{\Hc}^{\rd}_{\ell, c}$ hypergraphs capture all the crucial coloring restrictions on the remaining vertices: \begin{equation}\label{fact:compa1}
\text{any extension of $\Psi_i$ will be a proper coloring of $\Hc$, if it is $(\{\nPa_{\rd}(u)\}, \{\hat{\Hc}^{\rd}_{\ell, c}\})$-compatible.}
\end{equation}

\item \textbf{Set new ideal palette size and ideal $c$-degree bound.} Set 
\begin{equation}
\pa_{\rd}:=\beta\pa_{\rd-1}, \quad t_{\rd}:=\alpha'_{\rd}\beta^{k-1}t_{\rd-1},
\end{equation}
where $\alpha'_{\rd}=1 - \beta\pi_{\rd}\pa_{\rd-1}/6$. 
These represent the \textit{ideal palette size} and the \textit{ideal $c$-degree bound} at the $i$-th iteration, respectively. 
Since $\pi_{\rd}\pa_{\rd-1} \le 1$ (as detailed later in~\eqref{indas4}), $\alpha'>0$ is a valid choice.
Moreover, $\alpha'_{\rd} < 1$ indicates that, the $c$-degrees are expected to decrease slightly faster than the palette size, and thus their ratio (as defined later in \eqref{def:degPalRadio}) is indeed monotone decreasing.

\item \textbf{Filter out `heavy' colors and finalize the color palettes.}
As discovered by Pettie-Su~\cite{pettie2015distributed}, we need to \textit{filter out} colors with large $c$-degrees to better control the algorithm.
We first let the \textit{temporary $c$-degree} of $u$ be
\begin{equation}
\hat{d}_{\rd}(u, c):=\sum_{\ell=2}^k(\cde \pa_{\rd})^{k-\ell}\hat{d}^{\rd}_{\ell}(u, c),
\end{equation}
where $\hat{d}^{\rd}_{\ell}(u, c): = \deg_{\ell}(u, \hat{\Hc}^{\rd}_{\ell, c})$.
Then we define the \textit{new color palette} of $u$ (at the $i$-th iteration) as
\begin{equation}
\Pa_{\rd}(u) := \{c\in \nPa_{\rd}(u):\ \hat{d}_{\rd}(u, c)\leq 2t_{\rd}\}.
\end{equation}
We further assume, without loss of generality, that 
\begin{equation}\label{def:palet}
|\Pa_{\rd}(u)|:=\min\{|\{c\in \nPa_{\rd}(u):\ \hat{d}_{\rd}(u, c)\leq 2t_{\rd}\}|,\ \pa_{\rd}\},
\end{equation}
by arbitrarily deleting some extra colors from $\Pa_{\rd}(u)$.
This step provides an upper bound on the size of palettes, which will be used in the analysis of degree concentration.
Indeed, as will be shown later in Proposition~\ref{prop:pbou}, with high probability, we have $|\Pa_{\rd}(u)|=(1 - o(1)) \pa_{\rd}$ for each $u\in U_{\rd}$.
Note that since $\Pa_{\rd}(u)\subseteq \nPa_{\rd}(u)$,
\begin{equation}\label{fact:compa2}
\text{any $(\{\Pa_{\rd}(u)\}, \{\hat{\Hc}^{\rd}_{\ell, c}\})$-compatible coloring on $U_i$, is $(\{\nPa_{\rd}(u)\}, \{\hat{\Hc}^{\rd}_{\ell, c}\})$-compatible.}
\end{equation}

\item \textbf{Reduce codegrees and finalize restriction hypergraphs.} 
As mentioned in Section~\ref{sec:overview}, for the purpose of concentration analysis, we need to control the codegrees of our restriction hypergraphs using the \textit{codegree reduction algorithm} described in Section~\ref{sec:coredu}.
For integers $s, \ell\geq 0$, let $f(s, \ell):=(\cde\pa_{\rd})^{\ell-s}$, which satisfies the assumption of Proposition~\ref{prop:redu}.
For every color $c$, let
\[
\Fc'_{c}:=\left\{e\in\bigcup_{\ell=2}^{k}\hat{\Hc}^{\rd}_{\ell, c}:\ c\in \Pa_{\rd}(v) \text{ for all } v\in e\right\},
\] and denote by $\Fc_c$ the $f$-reduction of $\Fc'_{c}$.
Then for every color $c$ and $\ell$, we define the \textit{$(c, \ell)$-restriction hypergraph} at the $i$-th iteration as
\[
\Hc^{\rd}_{\ell, c}:=\{e\in\Fc_c: |e|=\ell\}.
\]
By Proposition~\ref{prop:redu}, for every $c$ and $2\leq s<\ell\leq k$,
\begin{equation}\label{reduce:code}
\delta_{s, \ell}(\Hc^{\rd}_{\ell, c})\leq (\cde\pa_{\rd})^{\ell-s},
\end{equation}
and
\[
\text{any $(\{\Pa_{\rd}(u)\}, \{\Hc^{\rd}_{\ell, c}\})$-compatible coloring on $U_i$, is $(\{\Pa_{\rd}(u)\}, \{\hat{\Hc}^{\rd}_{\ell, c}\})$-compatible.}
\]
This, together with~\eqref{fact:compa1} and~\eqref{fact:compa2}, leads to that
\begin{equation}\label{outout:compat}
\text{any extension of $\Psi_i$ will be a proper coloring of $\Hc$, if it is $(\{\Pa_{\rd}(u)\}, \{\Hc^{\rd}_{\ell, c}\})$-compatible;}
\end{equation}
or, intuitively speaking, these $\Hc^{\rd}_{\ell, c}$ hypergraphs capture all the \textit{crucial} coloring restrictions on the remaining vertices.
Without loss of generality, we further assume that
\begin{equation}\label{reduce:de}
\text{
for each $c$, there is no pair of edges $e_1, e_2\in \sum_{\ell=2}^{k}\Hc^{\rd}_{\ell, c}$ such that $e_1\subsetneq e_2$,
}
\end{equation}
as otherwise we would always delete the larger edge $e_2$ without increasing the codegrees, or affecting compatibility.

\item \textbf{Finalize $c$-degrees.} For every $u\in U_{i}$ and $c\in \Pa_{\rd}(u)$, we define the \textit{$c$-degree of $u$} (at the $i$-iteration) as 
\begin{equation}
d_{\rd}(u, c):=\sum_{\ell=2}^k(\cde \pa_{\rd})^{k-\ell}d^{\rd}_{\ell}(u, c)
\end{equation}
where $d^{\rd}_{\ell}(u, c):=\deg_{\ell}(u, \Hc^{\rd}_{\ell, c})$.
It is important to note that, due to the choice of weights, the $f$-reduction in step 8 does not increase the value of our weighted sum of degrees, as $(\cde\pa_{\rd})^{k-s}\cdot 1 - (\cde\pa_{\rd})^{k-\ell}\cdot (\cde\pa_{\rd})^{\ell-s}=0$.
The same holds true for the assumption~\eqref{reduce:de} as we only delete edges. Therefore, for all $c \in \Pa_{\rd}(u)$, we have
\begin{equation}\label{cde:bound}
d_{\rd}(u, c)\leq \hat{d}_{\rd}(u, c)\leq 2t_i.
\end{equation}

After step 9, we move to the next iteration until the termination condition is met.
\end{enumerate}

\subsubsection{The termination condition.}
For each $\rd\geq 0$, let 
\begin{equation}\label{def:degPalRadio}
\zeta_{\rd}:=\frac{t_{\rd}}{(\cde\pa_{\rd})^{k-1}},\end{equation} which measures the ratio between the ideal $c$-degree bound and the ideal palette size. 
We terminate this semi-random coloring algorithm after $T$ iterations, where $T$ is the first integer such that 
\begin{equation}\label{ratioT}
\zeta_{T}\leq 1/8k.
\end{equation}
Observe that $\zeta_{0}=\frac{(k-1)\Delta}{(\cde C)^{k-1}}=(k-1)\log\Delta$, and
\[
\zeta_{\rd}=\frac{\alpha'_{\rd}\beta^{k-1}t_{\rd-1}}{(\cde\beta\pa_{\rd-1})^{k-1}}=\alpha'_{\rd}\frac{t_{\rd-1}}{(\cde\pa_{\rd-1})^{k-1}}
=\left(1 - \frac{\beta}{6}\left(\cac\frac{(\cde\pa_{\rd-1})^{k-2}}{4t_{\rd-1}}\right)\pa_{\rd-1}\right)\frac{t_{\rd-1}}{(\cde\pa_{\rd-1})^{k-1}}
=\zeta_{\rd -1} - \frac{\beta\cac}{24\cde},
\]
which is a strictly decreasing function.
Therefore, we have
\[T\leq \frac{24(k-1)\cde}{(1 - \cac)\cac}\log\Delta.\] 

\subsubsection{Summary}
In summary, for a rank $k$ hypergraph $\Hc$ that satisfies all the degree and codegree conditions in Theorem~\ref{mainthm2}, and any list assignment $\mathcal{P}$ with $|\Pa(u)|= \Cco$ for all vertices $u$, our coloring algorithm generates a sequence of outputs $\{U_{\rd}, \Psi_\rd, \{\Pa_\rd(u)\}, \{\Hc^\rd_{\ell,c}\}\}_{\rd=0}^{T}$ satisfying the following proposition.


\begin{prop}\label{prop:output}
Let $\Hc$ a rank $k$ hypergraph that satisfies all the degree and codegree conditions in Theorem~\ref{mainthm2}, and $\mathcal{P}$ be a list assignment with $|\Pa(u)|= \Cco$ for all vertices $u$. 

Suppose that $\{U_{\rd}, \Psi_\rd, \{\Pa_\rd(u)\}, \{\Hc^\rd_{\ell,c}\}\}_{\rd=0}^{T}$ is a sequence of outputs generated from our coloring algorithm with input $\{\Hc, \mathcal{P}\}$. 
For every $0\le \rd \le T$, $\{U_{\rd}, \Psi_\rd, \{\Pa_\rd(u)\}, \{\Hc^\rd_{\ell,c}\}\}$ satisfies the following properties:
\begin{enumerate}
    \item[(i)] $d_{\rd}(u, c)=\sum_{\ell=2}^k(\cde \pa_{\rd})^{k-\ell}d^{\rd}_{\ell}(u, c) \leq 2t_{\rd}, \quad \text{for all }u\in U_{\rd} \text{ every } c\in \Pa_{\rd}(u)$;
\item[(ii)] $\delta_{s, \ell}(\Hc^{\rd}_{\ell, c})\leq (\cde\pa_{\rd})^{\ell-s} \quad \text{for every } 2\le s < \ell \le k \text{ and } c;$
\item[(iii)] any extension of $\Psi_{\rd}$ will be a proper coloring of $\Hc$, if it is $\left(\{\Pa_{\rd}(u)\}, \{\Hc^{\rd}_{\ell, c}\}\right)$-compatible.
\end{enumerate}
\end{prop}
\begin{proof}
    The case $i=0$ is established in Section~\ref{sec:alg:initial}. For $i\ge 1$, it follows from the mechanics of the algorithm, see~\eqref{cde:bound}, \eqref{reduce:code} and~\eqref{outout:compat}.
\end{proof}

In particular, when the algorithm terminates, we obtain a proper partial coloring  $\Psi_T$ of $\Hc$, although some vertices, $U^T$, may remain uncolored. 
However, since $\zeta_{T}$ is now sufficiently small, it is not difficult to properly color the remaining vertices in $U^T$ using the standard Local Lemma, see details in Section~\ref{sec:mainthm2}.

\subsection{Notation and parameters}\label{sec:notation}
As outlined above, the algorithm is parameterized by the ideal $c$-degree bound ${t_{\rd}}$ and the ideal palette sizes ${p_{\rd}}$. In practice, the \textit{actual} palette sizes and $c$-degrees after $\rd$ rounds may deviate from these ideal values. To account for these deviations, we define 
\[p'_{\rd}:=(1 - \er/8)^{\rd}\pa_{\rd} \quad \text{and} \quad t'_{\rd}:=(1 + \er)^{\rd}t_{\rd}\]
as the \textit{approximate} versions of $p_{\rd}$ and $t_{\rd}$, incorporating a small error control parameter $\er:=4\Delta^{-\theta}\log^{2k}\Delta$, where $\theta=1/4k$. Additionally, for ease of notation, we introduce another parameter \[\alpha_{\rd}:=1 - \beta\pi_{\rd}\pa_{\rd-1}/5,\] which will be used to measure the decreasing rate of $c$-degrees. 

As mentioned in Section~\ref{sec:overview}, rather than concentrating on $c$-degrees, we will show that at the $i$-iteration, the \textit{average $c$-degree} of a vertex $u$, defined as
\[
\Lambda_{\rd}(u):=\sum_{c\in \Pa_{\rd}(u)}\frac{d_{\rd}(u, c)}{|\Pa_{\rd}(u)|},\]
can be well-concentrated.
Moreover, following an idea of Pettie-Su~\cite{pettie2015distributed}, we introduce the following concept to balance the tradeoff between the palette size and the average $c$-degree:
\[
\D_{\rd}(u):=\lambda_{\rd}(u)\Lambda_{\rd}(u) + (1 - \lambda_{\rd}(u))2t_{\rd},
\]
where $\lambda_{\rd}(u):=\min\left\{1, |\Pa_{\rd}(u)|/p'_{\rd}\right\}$.

We summarize all notation and parameters in Table~\ref{table2} for the convenience of readers.
\begin{table}[H]
\centering
\begin{tabular}{lll}
\hline
\hline
Notation & Value & Description \\
\hline
$U_{\rd}$ & & the set of uncolored vertices. \\
$\Psi_{\rd}$ & & proper partial coloring on $U_i$\\
$\Pa_{\rd}(u)$ &  & Color palette of the vertex $u$\\
$\Hc^{\rd}_{\ell, c}$ & & $(c, \ell)$-restriction hypergraph\\
$d^{\rd}_{\ell}(u, c)$ & & number of edges of $\Hc^{\rd}_{\ell, c}$ that are incident to $u$\\
$d_{\rd}(u, c)$ & $\sum_{\ell=2}^k(\cde \pa_{\rd})^{k-\ell}d^{\rd}_{\ell}(u, c)$& $c$-degree of $u$\\
$\Lambda_{\rd}(u)$ &$\sum_{c\in \Pa_{\rd}(u)}\frac{d_{\rd}(u, c)}{|\Pa_{\rd}(u)|}$ &average $c$-degree of a vertex $u$\\
$\D_{\rd}(u)$ & $\lambda_{\rd}(u)\Lambda_{\rd}(u) + (1 - \lambda_{\rd}(u))2t_{\rd}$ & adjusted average $c$-degree of a vertex $u$\\
$\lambda_{\rd}(u)$ &$\min\left\{1, |\Pa_{\rd}(u)|/p'_{\rd}\right\}$& \\
\hline
$\nPa_{\rd}(u)$ &  & temporary palette of the vertex $u$\\
$\hat{\Hc}^{\rd}_{\ell, c}$ & & temporary $(c, \ell)$-restriction hypergraph\\
$\hat{d}^{\rd}_{\ell}(u, c)$ & & number of edges of $\hat{\Hc}^{\rd}_{\ell, c}$ that are incident to $u$\\
$\hat{d}_{\rd}(u, c)$ & $\sum_{\ell=2}^k(\cde \pa_{\rd})^{k-\ell}\hat{d}^{\rd}_{\ell}(u, c)$& temporary $c$-degree of $u$\\
\hline
$\ac^{\rd}_{u, c}$ & & indicator variable for the event that $c$ is activated on $u$\\
$\ke^{\rd}_{u, c}$ & & indicator variable for the event that $c$ is selected on $u$\\
$\Ac_{\rd}(u)$ & & the set of activated color of the vertex $u$\\
$\Lo_{\rd}(u)$ & & the set of lost color of the vertex $u$\\
$\Ke_{\rd}(u)$ & & the set of selected color of the vertex $u$\\
$\pi_{\rd}$ & $\cac\frac{(\cde\pa_{\rd-1})^{k-2}}{4t_{\rd-1}}$  & probability of color-activation\\
$\beta$ & $1 - \frac{4t_{\rd -1}\pi_{\rd}}{(\cde \pa_{\rd-1})^{k-2}}=1 - \cac$   & ideal probability of retaining a color\\
$\alpha_{\rd}$ &  $1 - \beta\pi_{\rd}\pa_{\rd-1}/5$  & decreasing rate\\
$\alpha'_{\rd}$ & $1 - \beta\pi_{\rd}\pa_{\rd-1}/6$  & adjusted decreasing rate\\
\hline
$\pa_{0}$ & $\Cco=\cde^{-1}(\Delta/\log\Delta)^{1/(k-1)}$ & ideal palette size at the start of the algorithm\\
$\pa_{\rd}$ & $\beta\pa_{\rd-1}$ & ideal palette size at $i$-th iteration\\
$\pa'_{\rd}$ & $(1 - \er/8)^{\rd}\pa_{\rd}$  & approximate palette size \\
$t_{0}$ & $(k-1)\Delta$  & ideal $c$-degree bound at the start of the algorithm\\
$t_{\rd}$ & $\alpha'_{\rd}\beta^{k-1}t_{\rd-1}$  & ideal $c$-degree bound at $i$-th iteration\\
$t'_{\rd}$ & $(1 + \er)^{\rd}t_{\rd}$   & approximate $c$-degree bound\\
$\zeta_{\rd}$ & $\frac{t_{\rd}}{(\cde\pa_{\rd})^{k-1}}$ & a ratio\\
$T$  &  $\leq\frac{24(k-1)\cde}{(1 - \cac)\cac}\log\Delta$  & total number of iterations\\
\hline
$\cde$  & $1/(60\cdot2^k)$  & constant\\
$\cac$  & $1/(60\cdot k^4\cdot 2^k)$  & constant\\
$\theta$  & $1/4k$  & constant\\
$\er$ &  $4\Delta^{-\theta}\log^{2k}\Delta$  & error term \\
\hline
\hline
\end{tabular}
\caption{Notation and parameters}
\label{table2}
\end{table}
To conclude this section, we present three useful facts that follow directly from the definitions of the parameters:
\begin{equation}\label{indas6}
 1/8k \leq \zeta_{\rd} \leq (k-1)\log\Delta \quad \text{for all }0\leq \rd<T,
\end{equation}
\begin{equation}\label{indas4}
 \Omega(1/\log\Delta)\le \cde\pi_{\rd}\pa_{\rd-1} \le \pi_{\rd}\pa_{\rd-1}= \cac/(4\cde\zeta_{\rd-1}) \le \cac\cdot 2k/\cde \le 1/2 \quad \text{for all }1\leq \rd\leq T,
\end{equation}
and
\begin{equation}\label{indas7}
\pa_{\rd}\geq \beta^{T}\pa_0 \geq (1 - \cac)^{\frac{24(k-1)\cde}{(1 - \cac)\cac}\log\Delta}\cde^{-1}(\Delta/\log\Delta)^{1/(k-1)}\geq\Delta^{\frac{1}{2(k-1)}} \quad \text{for all $0\leq i\leq T$}.
\end{equation}
Note that \eqref{indas6} and \eqref{indas7} further imply
\begin{equation}\label{indas5}
\frac{t_{\rd-1}}{(\cde\pa_{\rd-1})^{k-2}} = \cde\pa_{\rd-1}\zeta_{\rd-1}= \Omega\left(\Delta^{1/2(k-1)}\right) \quad \text{for all }1\leq \rd\leq T,
\end{equation}
and therefore
\begin{equation}\label{pibound}
\pi_{\rd}=O\left(\Delta^{-1/2(k-1)}\right) \quad \text{for all }1\leq \rd\leq T.
\end{equation}

\section{Proof of Theorem~\ref{mainthm2}}\label{sec:mainthm2}
Throughout the rest of the paper, let $\{U_{\rd}, \Psi_\rd, \{\Pa_\rd(u)\}, \{\Hc^\rd_{\ell,c}\}\}_{\rd=0}^{T}$ denote a sequence of outputs generated from our coloring algorithm with input $\{\Hc, \mathcal{P}\}$, where $\Hc$ is a rank $k$ triangle-free hypergraph that satisfies all the degree and codegree conditions in Theorem~\ref{mainthm2}, and $\mathcal{P}$ is a list assignment with $|\Pa(u)|= \Cco=\cde^{-1}(\Delta/\log\Delta)^{1/(k-1)}$ for all vertices $u$. 

We now state our key lemma that yields the proof of Theorem~\ref{mainthm2}.

\begin{lemma}[Key Lemma]\label{lem:de}
Let $1\leq \rd\leq T$. If $\D_{\rd-1}(u) \leq t'_{\rd-1}$ for all $u\in U_{\rd-1}$, then there exists $\{U_{\rd}, \Psi_{\rd}, \{\Pa_{\rd}(u)\}, \{\Hc^\rd_{\ell,c}\}\}$ such that \[\text{$\D_{\rd}(u) \leq t'_{\rd}$ for all $u\in U_{\rd}$.}\]
\end{lemma}
Lemma~\ref{lem:de} immediately implies the following corollary.
\begin{cor}\label{cor:de}
For every $0\leq \rd\leq T$, there exists $\{U_{\rd}, \Psi_{\rd}, \{\Pa_{\rd}(u)\}, \{\Hc^\rd_{\ell,c}\}\}$ such that 
$\D_{\rd}(u) \leq t'_{\rd}$ for all $u\in U_{\rd}$.
\end{cor}
\begin{proof}
For $i=0$, note that for every vertex $u\in U_0$, we have $\lambda_0(u)=\min\left\{1, |\Pa_{0}(u)|/p'_{0}\right\}=1$, and thus \[\D_{0}(u)=\Lambda_{0}(u)=\sum_{c\in \Pa_{0}(u)}\frac{d_{0}(u, c)}{|\Pa_{0}(u)|}\leq t_0 = t'_0,\]
where the inequality follows from~\eqref{cde:zero}.
The case $i\geq 1$ then follows by induction on $i$, using Lemma~\ref{lem:de}.
\end{proof}

The proof of Lemma~\ref{lem:de} is lengthy and technical, and will be distributed across Sections~\ref{sec:mainlemma},~\ref{sec:expe},and~\ref{sec:concede}. 
In the rest of this section, we explain how our main theorem, Theorem~\ref{mainthm2}, derives from it.
We start with the following simple proposition. 
\begin{prop}\label{prop:pbou}
Let $0\leq \rd \leq T$ and $u\in U_i$.
If $\D_{\rd}(u)\leq t'_{\rd}$, then
\begin{equation*}
|\Pa_{\rd}(u)|
\geq (1 -  (1 + \er)^{\rd}/2)\pa'_{\rd}
= (1 -  (1 + \er)^{\rd}/2)(1 - \er/8)^{\rd}\pa_{\rd}.
\end{equation*}
\end{prop}
\begin{proof}
By the definition of $\D_{\rd}(u)$, if $\D_{\rd}(u)\leq t'_{\rd}$, then we have $(1 - \lambda_{\rd}(u))2t_{\rd}\leq \D_{\rd}(u)\leq t'_{\rd}.$
This, together with the definition of $\lambda_{\rd}$, implies that
\[
\frac{|\Pa_{\rd}(u)|}{p'_{\rd}}\geq\lambda_{\rd}(u)\geq 1 - \frac{t'_{\rd}}{2t_{\rd}}=1 -  \frac{(1 + \er)^{\rd}}{2},
\]
which completes the proof.
\end{proof}

The mechanics of the algorithm---specifically,  filtering out colors with large $c$-degrees---ensures that the 
$c$-degrees of hypergraphs decrease at a fairly fast speed, while Corollary~\ref{cor:de} and Proposition~\ref{prop:pbou} together imply that, at termination, there will still be a sufficiently large number of usable colors for each uncolored vertex. We can then color the remaining vertices using the standard approach, the Local Lemma.
\begin{proof}[Proof of Theorem~\ref{mainthm2}]
Applying our (semi-random) coloring algorithm (as described in Section~\ref{sec:algo}) to $\Hc$, we obtain a proper partial coloring $\Psi_T$ of $\Hc$, with potentially some uncolored vertices in $U^T$.
Additionally, by Corollary~\ref{cor:de} and Proposition~\ref{prop:pbou}, there exists $\Psi_T$ so
that for all $u\in U_T$,
\[|\Pa_{T}(u)|\ge (1 -  (1 + \er)^{T}/2)(1 - \er/8)^{T}\pa_{T}\ge (1 -  (1 + 2\er T)/2)(1 - \er T/8)\pa_{T}\geq \pa_{T}/4.\]
where the last two inequalities follow from $\er \ll \er T \ll 1$ as $\Delta$ becomes sufficiently large (see Table~\ref{table2}).

For every vertex $u\in U^T$, we color it with colors in $\Pa_{T}(u)$ uniformly at random. 
The goal is to show that with positive probability, there exists a $\left(\{\Pa_{T}(u)\}, \{\Hc^T_{\ell, c}\}\right)$-compatible coloring. By Proposition~\ref{prop:output} (iii), this coloring, combined with $\Psi_T$, forms a proper coloring of $\Hc$.

For a color $c$, an integer $2\leq \ell\leq k$, and an edge $e_{\ell}\in \Hc^T_{\ell, c}$, let $\A_{e_{\ell}, c}$ be the event that all the vertices in $e_{\ell}$ receive the color $c$. Then, it is equivalent to prove that, with positive probability, none of the events $\A_{e_{\ell}, c}$ occur.
First, observe that for any event $\A_{e_{\ell}, c}$ with $e_\ell\in \Hc^T_{\ell, c}$,
\[\Prb\left[\A_{e_{\ell}, c}\right]\le \prod_{u\in e_{\ell}}\frac{1}{|\Pa_{T}(u)|}
\leq \left(\frac{1}{\pa_{T}/4}\right)^{\ell} \ll \frac14,\]
where the last inequality uses~\eqref{indas7}.
Next, let $\mathcal{D}_{e_\ell,c}$ be the collection of events $\A_{e, c'}$ that are dependent on $\A_{e_{\ell}, c}$. Note that two events $\A_{e_{\ell}, c}$, $\A_{e, c'}$ are dependent if and only if $e_{\ell} \cap e \neq \emptyset$.
Therefore, we have
\[
\begin{split}
\sum_{\A\in\mathcal{D}_{e_\ell,c}}\Prb\left[\A\right] 
&\leq \sum_{u\in e_{\ell}}\sum_{c\in \Pa_{T}(u)}\sum_{\ell=2}^k\sum_{\substack{e\in \Hc^T_{\ell,c}\\ u\in e}}\Prb\left[\A_{e, c}\right]
=\sum_{u\in e_{\ell}}\sum_{c\in \Pa_{T}(u)}\sum_{\ell=2}^k\sum_{\substack{e\in \Hc^T_{\ell,c}\\ u\in e}}\frac{1}{|\Pa_{T}(u)|}\prod_{v\in e\setminus\{u\}}\frac{1}{|\Pa_{T}(v)|}\\
&\leq \sum_{u\in e_{\ell}}\sum_{c\in \Pa_{T}(u)}
\frac{1}{|\Pa_{T}(u)|}\sum_{\ell=2}^kd^{T}_{\ell}(u, c)\left(\frac{4}{\pa_{T}}\right)^{\ell -1} 
\leq \sum_{u\in e_{\ell}}\sum_{c\in \Pa_{T}(u)}
\frac{1}{|\Pa_{T}(u)|}\sum_{\ell=2}^kd^{T}_{\ell}(u, c)\left(\frac{1}{\cde\pa_{T}}\right)^{\ell -1}\\
&=\sum_{u\in e_{\ell}}\sum_{c\in \Pa_{T}(u)}
\frac{1}{|\Pa_{T}(u)|}\left(\frac{1}{\cde\pa_{T}}\right)^{k -1} d_{T}(u, c)
\leq k\frac{2t_{T}}{(\cde\pa_{T})^{k -1}} =2k\zeta_{T} \leq 1/4,
\end{split}
\]
where the second last inequality follows from~Proposition~\ref{prop:output} (i), and the last inequality uses~\eqref{ratioT}.
The Local Lemma (Theorem~\ref{locallemma}) then implies that there exists a coloring in which none of the events $\A_{e_{\ell}, c}$ occur, thereby completing the proof.
\end{proof}

\section{Proof of Key Lemma (Lemma~\ref{lem:de})}\label{sec:mainlemma}
In this section, we prove our Key Lemma, Lemma~\ref{lem:de}, assuming two lemmas that will be proved in later sections. Throughout this section (and the next two sections), we fix an integer $1\leq i\leq T$, restrict our analysis to the $\rd$-th iteration of the algorithm, and assume that
\begin{equation}\label{indas1}
\D_{\rd-1}(u) \leq t'_{\rd-1} \quad \text{for all } u\in U_{\rd-1}.
\end{equation}
Our goal is to show that, with positive probability, at the $i$-th iteration, our coloring algorithm produces an output $\{U_{\rd}, \Psi_{\rd}, \{\Pa_{\rd}(u)\}, \{\Hc^\rd_{\ell,c}\}\}$ that satisfies 
$\D_{\rd}(u) \leq t'_{\rd}$ for all $u\in U_{\rd}$.



\subsection{Probability space and random variables}\label{sec:probRV}

For every vertex $u\in U_{i-1}$ and color $c\in \Pa_{\rd-1}(u)$, we let $(\Omega_{u,c}^{{\rm ac}}, \Sigma_{u,c}^{{\rm ac}}, \mathbb{P}_{u,c}^{{\rm ac}})$ be the probability space associated with $\ac^{\rd}_{u, c}$ (see~\eqref{def:act} for definition), and let $(\Omega_{u,c}^{{\rm se}}, \Sigma_{u,c}^{{\rm se}}, \mathbb{P}_{u,c}^{{\rm se}})$ be the probability space associated with $\ke^{\rd}_{u, c}$ (see~\eqref{def:sel} for definition). 
\begin{defi}[Probability Space of Procedure]\label{def:probspace}
We define the product probability space
$$(\Omega,\Sigma,\mathbb{P}) =  \prod_{u\in U_{i-1}}\prod_{c\in \Pa_{\rd-1}(u)} (\Omega_{u,c}^{{\rm ac}}, \Sigma_{u,c}^{{\rm ac}}, \mathbb{P}_{u,c}^{{\rm ac}}) \times (\Omega_{u,c}^{{\rm se}}, \Sigma_{u,c}^{{\rm se}}, \mathbb{P}_{u,c}^{{\rm se}}).
$$
\end{defi}

For every $u\in U_{i-1}$, define the \textit{neighborhood} of $u$ as
\[
N_{\rd-1}(u):=\left\{v\in U_{\rd-1} \mid \{u, v\}\subseteq e \text{ for some } e\in \bigcup_{c\in\Pa_{\rd-1}(u)}\bigcup_{\ell=2}^{k}\Hc^{\rd-1}_{\ell, c}\right\},
\]
and for every integer $d\geq 2$, define the \textit{$d$-th neighborhood} of $u$ iteratively as $N^{d}_{\rd-1}(u):=\bigcup_{v\in N^{d-1}_{\rd-1}(u)}N_{\rd-1}(v)$.
Observe that 
\begin{equation}\label{bound:nbd}
\text{$|N_{\rd-1}(u)|\leq \sum_{\ell=2}^{k}\ell\cdot\Delta_{\ell}(\Hc)\le k^2\Delta$, and then $|N^d_{\rd-1}(u)|\leq (k^2\Delta)^d$ for every $d\ge 2$.}
\end{equation}

By the definition of $\D_{\rd}(u)$ (see Table~\ref{table2}) and~\eqref{cde:bound}, to establish an upper bound on $\D_{\rd}(u)$, it is important to analyze how $\sum_{c\in\nPa_{\rd}(u)}\hat{d}_{\rd}(u, c)$ evolves as the algorithm proceeds.
Recall from Section~\ref{sec:alg:itera} (step 5) that for each $\ell$ and $c$, the edges of $\hat{\Hc}^{\rd}_{\ell, c}$ come from two sources: 
\begin{itemize}
    \item the edges in $\Hc^{\rd-1}_{\ell, c}$,     
    \item and the $\ell$-sets contained in some edge of larger uniformity, i.e., in some $e\in \bigcap_{s>\ell}\Hc^{\rd-1}_{s, c}$.  
\end{itemize}
Specifically, for the first type, an edge is included in $\hat{\Hc}^{\rd}_{\ell, c}$, if and only if all its vertices remain uncolored and still have $c$ in their (temporary) palettes. For the second type, an $\ell$-set $L$ is added to $\hat{\Hc}^{\rd}_{\ell, c}$, if and only if there exists some edge $e$ containing $L$, such that all vertices in $e\setminus L$ are colored by $c$, and all vertices in $L$ still have $c$ in their (temporary) palettes.

In summary, for every $u\in U_{\rd}$, we have
\[
\begin{split}
\hat{d}^{\rd}_{\ell}(u, c)
&\leq 
\sum_{\substack{e\in\Hc^{\rd-1}_{\ell, c}\\ u\in e}} \I\left[\left(\forall x\in e\setminus\{u\},~x\in U_{\rd}\right) \wedge \left(\forall y\in e,~c\in\nPa_{\rd}(y)\right)\right]\\
&\quad +\sum_{s=\ell+1}^{k}\sum_{\substack{e\in \Hc^{\rd-1}_{s, c}\\ u\in e}}\sum_{Q\in\binom{e-u}{s-\ell}}
\I\left[\left(\forall x\in Q,~\Psi_i(x)=c\right)\wedge \left( \forall y\in e-Q,~c\in \nPa_{\rd}(y)\right)\right].
\end{split}
\]
Recall from Section~\ref{sec:alg:itera} (step 4) that $\Psi_i(x)=c$, only if $c$ is in both $\nPa_{\rd}(x)$ and $\Ac_{\rd}(x)$.
Then we further have
\[
\begin{split}
\hat{d}^{\rd}_{\ell}(u, c)
&\leq \sum_{\substack{e\in\Hc^{\rd-1}_{\ell, c}\\ u\in e}} \I\left[\left(\forall x\in e\setminus\{u\},~x\in U_{\rd}\right) \wedge \left(\forall y\in e,~c\in\nPa_{\rd}(y)\right)\right]\\
&\quad +\sum_{s=\ell+1}^{k}\sum_{\substack{e\in \Hc^{\rd-1}_{s, c}\\ u\in e}}\sum_{Q\in\binom{e-u}{s-\ell}}
\I\left[\left(\forall x\in Q,~c\in\Ac_{\rd}(x) \right)\wedge\left(\forall y\in e,~c\in \nPa_{\rd}(y)\right)\right]
\end{split}
\]
To simplify the notation, we introduce the following random variables:
\begin{equation}\label{def:keedge}
\X_{u,\ell}:=\sum_{c\in\Pa_{\rd-1}(u)}\sum_{\substack{e\in\Hc^{\rd-1}_{\ell, c}\\ u\in e}} \I\left[\left(\forall x\in e\setminus\{u\},~x\in U_{\rd}\right) \wedge \left(\forall y\in e,~c\in\nPa_{\rd}(y)\right)\right],
\end{equation}
and
\begin{equation}\label{def:inedge}
\X_{u,\ell, s}:=\sum_{c\in\Pa_{\rd-1}(u)}\sum_{\substack{e\in \Hc^{\rd-1}_{s, c}\\ u\in e}}\sum_{Q\in\binom{e-u}{s-\ell}}
\I\left[\left(\forall x\in Q,~c\in\Ac_{\rd}(x) \right)\wedge\left(\forall y\in e,~c\in \nPa_{\rd}(y)\right)\right].
\end{equation}
Observe that
\begin{equation}\label{def:sumde}
\begin{split}
\sum_{c\in\nPa_{\rd}(u)}\hat{d}_{\rd}(u, c)
&=\sum_{c\in\nPa_{\rd}(u)}\sum_{\ell=2}^k\hat{d}^{\rd}_{\ell}(u, c)(\cde\pa_{\rd})^{k-\ell}
=\sum_{\ell=2}^k(\cde\pa_{\rd})^{k-\ell}\sum_{c\in\nPa_{\rd}(u)}\hat{d}^{\rd}_{\ell}(u, c)\\
&\leq \sum_{\ell=2}^k(\cde\pa_{\rd})^{k-\ell}\left(
\X_{u,\ell} + \sum_{s=\ell+1}^{k}\X_{u,\ell, s}\right):=\X_{u}.
\end{split}
\end{equation}
We conclude this section with the following proposition, whose proof follows directly from the mechanics of our algorithm and is omitted here.
\begin{prop}\label{prop:dep}
\begin{itemize}
    \item[(i)] For every $u\in U_{i-1}$, the random variable $|\nPa_{\rd}(u)|$ is fully determined by the indicator variables $\{\ac^{\rd}_{v, c}:\ v\in N_{\rd-1}(u),~c\in \Pa_{\rd-1}(u)\}$ and $\{\ke^{\rd}_{u, c}:~c\in \Pa_{\rd-1}(u)\}$.
\item[(ii)] For every $u\in U_i$, the random variable $\sum_{c\in\nPa_{\rd}(u)}\hat{d}_{\rd}(u, c)$ is fully determined by the indicator variables $\{\ac^{\rd}_{v, c}:\ v\in N^2_{\rd-1}(u),~c\in \Pa_{\rd-1}(v)\}$ and $\{\ke^{\rd}_{v, c}:\ v\in N_{\rd-1}(u),~c\in \Pa_{\rd-1}(v)\}$.
\end{itemize}
\end{prop}

\subsection{Proof of Lemma~\ref{lem:de}}
The proof of Lemma~\ref{lem:de} is established on the following three lemmas. The first lemma states that the size of the temporary palette $\nPa_{\rd}(u)$ is well-concentrated. 
\begin{lemma}\label{lem:tempale}
For every $u\in U_{\rd-1}$, $\Prb(|\nPa_{\rd}(u)|\geq (1 - \er/8)\beta|\Pa_{\rd-1}(u)|) \geq 1- \exp\left(-\Omega(\log^2\Delta)\right).$
\end{lemma}

\begin{proof}
Recall from~\eqref{eq:ncolor} that $\E|\nPa_{\rd}(u)|=\beta|\Pa_{\rd-1}(u)|$. By Chernoff bounds (Lemma~\ref{chernoff}), we have
\[
\begin{array}{lll}
\Prb\left(|\nPa_{\rd}(u)|\leq (1 - \er/8)\beta|\Pa_{\rd-1}(u)|\right)& 
\leq \exp\left(-(\er/8)^2\beta|\Pa_{\rd-1}(v)|/2\right) &   \\
&=\exp\left(- \Omega\left( \er^2
\beta\pa_{\rd-1}\right)\right) & \text{Proposition~\ref{prop:pbou}}~\&~ \eqref{indas1}\\
&=\exp\left(- \Omega\left( \er^2
\pa_{\rd-1}\right)\right) &  \beta=\Theta(1)\\
&=\exp\left(-\Omega(\log^2\Delta)\right) & \eqref{indas7}~\&~\er=4\Delta^{-1/4k}\log^{2k}\Delta,
\end{array}
\]
completing the proof.
\end{proof}
The second lemma provides an upper bound on the expected sum of temporary $c$-degrees $\hat{d}_{\rd}(u, c)$ over all colors in $\nPa_{\rd}(u)$.
\begin{lemma}\label{lem:expe}
For every $u\in U_{\rd}$,
\[
\E\left[\X_u\right] \le \alpha'_{\rd}\beta^{k}|\Pa_{\rd-1}(u)|\Lambda_{\rd-1}(u).
\]
\end{lemma}
The third lemma shows that the sum $\sum_{c\in\nPa_{\rd}(u)}\hat{d}_{\rd}(u, c)$ is well-concentrated around its expectation.
\begin{lemma}\label{lem:concen}
For every $u\in U_{\rd}$,
\[
\Prb\left(\X_u - \E\left[\X_u\right]\leq (\er/4)\alpha'_{\rd}\beta^{k}|\Pa_{\rd-1}(u)|t_{\rd-1}\right)
\geq 1- \exp\left(-\Omega(\log^2\Delta)\right).
\]
\end{lemma}


The proof of Lemmas~\ref{lem:expe} and~\ref{lem:concen} are the most technical components of our work and will be deferred to Sections~\ref{sec:expe} and \ref{sec:concede}, respectively.

\begin{proof}[Proof of Lemma~\ref{lem:de}]
Note by Proposition~\ref{prop:dep} that the event $\mathcal{A}_u:=\{\D_{\rd}(u)> t'_{\rd}\}$ is fully determined by the indicator variables $\{\ac^{\rd}_{v, c}:\ v\in N^2_{\rd-1}(u),~c\in \Pa_{\rd-1}(v)\}$ and $\{\ke^{\rd}_{v, c}:\ v\in N_{\rd-1}(u)\cup\{u\},~c\in \Pa_{\rd-1}(v)\}$. Therefore, two events $\mathcal{A}_u$ and $\mathcal{A}_{u'}$ are dependent only if $u'\in N^4_{\rd-1}(u)$, where $|N^4_{\rd-1}(u)| \leq (k^2\Delta)^4$ by~\eqref{bound:nbd}.
By applying the Local Lemma to the bad events $\{\mathcal{A}_u\}$, it suffices to show that for every $u\in U_{i}$,
\[
\Prb(\D_{\rd}(u)\leq t'_{\rd})\geq 1- \exp\left(-\Omega(\log^2\Delta)\right).
\]

Fix an arbitrary $u\in U_{i}$. By Lemmas~\ref{lem:tempale},~\ref{lem:expe} and~\ref{lem:concen}, and inequality~\eqref{def:sumde}, with probability at least $1- \exp\left(-\Omega(\log^2\Delta)\right)$, the following holds:
\begin{gather}
\sum_{c\in\nPa_{\rd}(u)}\hat{d}_{\rd}(u, c) \leq \alpha'_{\rd}\beta^{k}|\Pa_{\rd-1}(u)|\left(\Lambda_{\rd-1}(u) + \frac{\er}{4}t_{\rd-1}\right)\label{eq:re2};\\
|\nPa_{\rd}(u)|\geq (1 - \er/8)\beta|\Pa_{\rd-1}(u)|\label{eq:re1}.
\end{gather}
Let $\hat{\Lambda}_{\rd}(u):=\sum_{c\in \nPa_{\rd}(u)}\frac{\hat{d}_{\rd}(u, c)}{|\nPa_{\rd}(u)|}$. Then we have
\begin{equation}
\label{eq:Lamh1}
\begin{aligned}
\hat{\Lambda}_{\rd}(u) 
 & \leq (1 + \er/2)\alpha'_{\rd}\beta^{k-1}\left(\Lambda_{\rd-1}(u) + \frac{\er}{4}t_{\rd-1}\right) & \text{by}~\eqref{eq:re2}\&~\eqref{eq:re1}\\
 & \leq \alpha'_{\rd}\beta^{k-1}\Lambda_{\rd-1}(u) + (\er/2+\er/4+\er^2/8)\alpha'_{\rd}\beta^{k-1}t'_{\rd-1} & \Lambda_{\rd-1}(u)\leq \D_{\rd-1}(u) \leq t'_{\rd-1}\\
 & \leq \alpha'_{\rd}\beta^{k-1}\Lambda_{\rd-1}(u) + \er(1+ \er)^{\rd-1}t_{\rd}  & \alpha'_{\rd}\beta^{k-1}t'_{\rd-1}=(1+ \er)^{\rd-1}t_{\rd}.
\end{aligned}
\end{equation}
Here, $\Lambda_{\rd-1}(u)\leq \D_{\rd-1}(u)$ follows from the definitions of $\Lambda_{\rd-1}(u)$ and $\D_{\rd-1}(u)$ (see Table~\ref{table2}) and Proposition~\ref{prop:output} (i), while $\D_{\rd-1}(u) \leq t'_{\rd-1}$ is an assumption of our lemma (see~\eqref{indas1}).
In particular, 
\begin{equation}\label{eq:Lamh2}
\hat{\Lambda}_{\rd}(u) \leq \alpha'_{\rd}\beta^{k-1}t'_{\rd-1} + \er(1+ \er)^{\rd-1}t_{\rd}
=(1+ \er)^{\rd-1}t_{\rd} + \er(1+ \er)^{\rd-1}t_{\rd} = (1+ \er)^{\rd}t_{\rd} \leq 2t_{\rd},
\end{equation}
where the last inequality uses $\er i \le \er T \ll 1.$
Now instead of directly bounding $\D_{\rd}(u)$, we consider 
\[\hat{\D}_{\rd}(u):=\hat{\lambda}_{\rd}(u)\hat{\Lambda}_{\rd}(u) + (1 - \hat{\lambda}_{\rd}(u))2t_{\rd},\]
where $\hat{\lambda}_{\rd}(u):=\min\left\{1, |\nPa_{\rd}(u)|/\pa'_{\rd}\right\}$.
Compared to $\hat{\D}_{\rd}(u)$, $\D_{\rd}(u)$ can be viewed as the average $c$-degree of the palette obtained by replacing colors in $\hat{P}_{\rd}(u)$ with larger 
$c$-degrees by dummy colors, each with a 
$c$-degree of exactly $2t_{\rd}$. 
Since the average only goes down in this process, we immediately have
\[
\D_{\rd}(u)\leq \hat{\D}_{\rd}(u).
\]
Additionally, note by~\eqref{eq:re1} that
\[
\frac{\nPa_{\rd}(u)}{\pa'_{\rd}} \geq \frac{(1 - \er/8)\beta|\Pa_{\rd-1}(u)|}{(1 - \er/8)\beta\pa'_{\rd-1}}=
\frac{|\Pa_{\rd-1}(u)|}{\pa'_{\rd-1}},
\]
and this, combined with $\nPa_{\rd}(u) \supseteq \Pa_{\rd}(u)$, gives
\begin{equation}\label{eq:lamh}
\hat{\lambda}_{\rd}(u)\geq \lambda_{\rd}(u),\ \lambda_{\rd-1}(u).
\end{equation}
Finally, we obtain that
\begin{align*}
\D_{\rd}(u) &\leq \hat{\D}_{\rd}(u)=\hat{\lambda}_{\rd}(u)\hat{\Lambda}_{\rd}(u) + (1 - \hat{\lambda}_{\rd}(u))2t_{\rd} & \\
& \leq \lambda_{\rd-1}(u)\hat{\Lambda}_{\rd}(u) + (1 - \lambda_{\rd-1}(u))2t_{\rd} & \text{by}~\eqref{eq:Lamh2}~\&~\eqref{eq:lamh}\\
& \leq \lambda_{\rd-1}(u)\left(\alpha'_{\rd}\beta^{k-1}\Lambda_{\rd-1}(u) + \er(1+ \er)^{\rd-1}t_{\rd}\right) + (1 - \lambda_{\rd-1}(u))2t_{\rd} & \text{by}~\eqref{eq:Lamh1}\\
& \leq \alpha'_{\rd}\beta^{k-1}\left(\lambda_{\rd-1}(u)\Lambda_{\rd-1}(u) + (1 - \lambda_{\rd-1}(u))2t_{\rd-1}\right)+\er(1+ \er)^{\rd-1}t_{\rd} & t_{\rd}=\alpha'_{\rd}\beta^{k-1}t_{\rd-1}\\
&=\alpha'_{\rd}\beta^{k-1}\D_{\rd-1}(u)+\er(1+ \er)^{\rd-1}t_{\rd}\leq \alpha'_{\rd}\beta^{k-1}t'_{\rd-1} +\er(1+ \er)^{\rd-1}t_{\rd} & \text{defn. of }\D_{\rd-1}(u)~\&~\eqref{indas1}\\
&\leq (1 + \er)^{\rd-1}t_{\rd}+\er(1+ \er)^{\rd-1}t_{\rd}=t'_{\rd} & t'_{\rd}=(1 + \er)^{\rd}t_{\rd},
\end{align*}
which completes the proof.
\end{proof}

\section{Proof of Lemma~\ref{lem:expe}: Expectation of average $c$-degrees}\label{sec:expe}
Throughout this section, in addition to the assumptions stated at the beginning of Section~\ref{sec:mainlemma}, we further fix an arbitrary choice of $u\in U_i$.
Recall from \eqref{def:sumde} that 
\[
\X_u= \sum_{\ell=2}^k(\cde\pa_{\rd})^{k-\ell}\left(
\X_{u,\ell} + \sum_{s=\ell+1}^{k}\X_{u,\ell, s}\right).
\]
When the underlying vertex is clear, we simply write $\X_{\ell}$ for $\X_{u,\ell}$ and $\X_{\ell, s}$ for $\X_{u, \ell, s}$.
Our plan is to establish an upper bound for each $\E\left[\X_{\ell}\right]$ and $\E\left[\X_{\ell, s}\right]$, and then apply linearity of expectation to prove Lemma~\ref{lem:expe}. 

\subsection{Dependency lemma}\label{sec:dep}
Before we proceed to expectations, we first address some dependency issue arising from the analysis of the algorithm, which are needed for later calculations. As mentioned in Section~\ref{sec:overview}, this is the only place in the entire proof where triangle-freeness is required.
We begin by showing that the algorithm always maintains the triangle-freeness of hypergraphs.

\begin{prop}\label{trianglefree}
For every $0\leq \rd \leq T$ and color $c$, the hypergraph $\bigcup_{\ell=2}^{k}\Hc^{\rd}_{\ell, c}$ is triangle-free.
\end{prop}
\begin{proof}
We prove it by induction on $\rd$. The case $\rd=0$ is trivial as our input hypergraph $\Hc$ is triangle-free.
Suppose that $\bigcup_{\ell=2}^{k}\Hc^{\rd-1}_{\ell, c}$ is triangle-free for all color $c$. 

Assume by contradiction that $\bigcup_{\ell=2}^{k}\Hc^{\rd}_{\ell, c}$ is not triangle-free for some $c$. 
By the mechanics of the algorithm and Proposition~\ref{prop:redu}, a triangle cannot be created during the filtering process (step 7) or the codegree reduction process (step 8) of the algorithm. 
Therefore, there must be a triangle $\{e, f, g\}\subseteq \bigcup_{\ell=2}^{k}\hat{\Hc}^{\rd}_{\ell, c}$ such that $\{u, v\}\subseteq e$, $\{v, w\}\subseteq f$, $\{w, u\}\subseteq g$ and $\{u, v, w\}\cap e\cap f\cap g=\emptyset$. In particular, $w\notin e$.
By the induction hypothesis, at least one of these three edges is not in $\bigcup_{\ell=2}^{k}\Hc^{\rd-1}_{\ell, c}$. To simplify the discussion, we further assume that $e$ is the only edge in $\bigcup_{\ell=2}^{k}\hat{\Hc}^{\rd}_{\ell, c}$ but not in $\bigcup_{\ell=2}^{k}\Hc^{\rd-1}_{\ell, c}$; other cases will follow from a similar argument. Let $s:=|e|$.

By the definition of $\hat{\Hc}^{\rd}_{\ell, c}$ , this new edge $e$ must be created in the following situation:
during step 5 of the 
$\rd$-th iteration, there existed an edge $e'\in \bigcup_{\ell=2}^{k}\Hc^{\rd-1}_{\ell, c}$ such that all vertices in $e'\setminus e$ are colored by $c$. 
Note that $w\notin e'\setminus e$, and therefore not in $e'$, since $w$ remains uncolored at the end of round $\rd$. Then $\{e', f, g\}$ forms triangle in $\bigcup_{\ell=2}^{k}\Hc^{\rd-1}_{\ell, c}$, which contradicts our induction hypothesis.
\end{proof}

Intuitively, to bound the expectation of $\hat{d}_{\rd}(u, c)$, we need an upper bound on the probability $\Prb(c\in \nPa_{\rd}(v) \text{ for all } v\in e)$ for some hyperedge $e$. 
Note that if the events $\{c\in  \nPa_{\rd}(v)\}$ are mutually independent, then we can easily do it. 
However, we do not have such independence; in fact, these events are positively correlated, as in the hypergraph setup, vertices $v, u\in e$ can share common neighbors within $e$, even with the triangle-freeness condition.
Fortunately, the next lemma shows that although these events are not independent, they are `almost independent'.

\begin{lemma}[Dependency lemma]\label{lem:dep}
Let $1\leq \rd\leq T$, and $c$ be an arbitrary color. 
For any $e_0\in \cup_{\ell=2}^k\Hc^{\rd-1}_{\ell, c}$ and $S\subseteq e_0$, 
\begin{equation}
\Prb(\forall v\in S,~c\in \nPa_{\rd}(v)) \leq (1 + \er_0)\prod_{v\in S}\Prb(c\in \nPa_{\rd}(v)) = (1 + \er_0)\beta^{|S|},
\end{equation}
where $\er_0:=\beta\pi_{\rd}\pa_{\rd-1}/60$.
\end{lemma}
\begin{proof}
By~\eqref{eq:ncolor} and the independence of the $\ke^{i}_{v, c}$ variables, it is sufficient to prove that
\begin{equation}
\Prb(\forall v\in S,~c\notin \Lo_{\rd}(v)) \leq (1 + \er_0)\prod_{v\in S}\Prb(c\notin \Lo_{\rd}(v)).
\end{equation}

For a vertex $v\in S$, and an edge $e\in\bigcup_{\ell=2}^{k}\Hc^{\rd-1}_{\ell, c}$, denote by $\A_{e, v}$ the event that 
\[\text{$\ac^{\rd}_{u,c}=1$ for all $u\in e-v$.}\]
By the independence of the $\ac^{i}_{v, c}$ variables, we have $\Prb(\A_{e, v})=\pi_{\rd}^{|e|-1}$.
For each $v\in S\subseteq e_0$, let $\Ic_v=\left\{(e, v):\ v\in e\in\bigcup_{\ell=2}^{k}\Hc^{\rd-1}_{\ell, c}\right\}$ be an index set.
Note that by~\eqref{reduce:de},
\begin{equation}\label{dep:assum}
\text{there is no two distinct pairs $(e, v), (e', v')$ in the index sets such that $e\subsetneq e'$}.
\end{equation}
Then the definition of $\Lo_i(v)$ gives that for every $v\in S$,
\[
\Prb(c\notin \Lo_{\rd}(v)) = \Prb\left(\bigwedge_{(e,v)\in\Ic_v}\overline{\A_{e,v}}\right),
\]
and
\[
\Prb(\forall v\in S,~c\notin \Lo_{\rd}(v)) = \Prb\left(\bigwedge_{v\in S}\bigwedge_{(e,v)\in\Ic_v}\overline{\A_{e,v}}\right).
\]

For any two pairs $(e,v), (e', v')\in \bigcup_{v\in S}\Ic_v$, write $(e,v)\sim (e', v')$, if $(e, v)\neq (e', v')$ and $(e-v)\cap (e' -v')\neq \emptyset$. We will need the following claim.
\begin{claim}\label{claim:dep}
If $(e,v)\sim (e', v')$, then $|e'|>|(e-v)\cap (e' -v')|+1$.
\end{claim}
\begin{proof}
Assume by contradiction that $|e'|=|(e-v)\cap (e' -v')|+1$. Then $e'=(e-v)\cap (e' -v')+v'$, or equivalently $(e' -v') \subseteq (e-v)$. 
Observe that $v\neq v'$ and $e\neq e_0$; otherwise, we would either have $(e, v)=(e', v')$, or $e'\subsetneq e$, which contradicts~\eqref{dep:assum}.
This further indicates that $v\notin e'$, since otherwise it would contradict $(e' -v') \subseteq (e-v)$.
Then $e'$ cannot be equal to either $e_0$ or $e$, because $v\notin e'$ but $v\in e, e_0$.

Now we have $v\neq v'$ and $e, e', e_0$ are three distinct edges in $\bigcup_{\ell=2}^{k}\Hc^{\rd-1}_{\ell, c}$.
Note that by~\eqref{reduce:de}, $e'$ cannot be fully contained in $e_0$, and then the same applies to $(e-v)\cap (e' -v')=(e' -v')$.
Thus, there exists a vertex $w\in (e-v)\cap (e' -v')$ such that $w\notin e_0$. 
Then $\{e_0, e, e'\}$ forms a triangle as $\{v, v'\}\subseteq e_0$, $\{v, w\}\subseteq e$, $\{v', w\}\subseteq e'$, and $\{v, v', w\}\cap e_0\cap e\cap e'=\emptyset$.
This contradicts Proposition~\ref{trianglefree}.
\end{proof}

We now apply Theorem~\ref{Janson} (Janson's Inequality) on $\Prb\left(\bigwedge_{v\in S}\bigwedge_{(e,v)\in\Ic_v}\overline{\A_{e,v}}\right)$.
We first set
\[
M:= \prod_{v\in S}\prod_{(e,v)\in\Ic_v}\Prb\left(\overline{\A_{e,v}}\right),
\]
and note that 
\begin{equation}\label{dep:M}
M=\prod_{v\in S}\left(\prod_{(e,v)\in\Ic_v}\Prb\left(\overline{\A_{e,v}}\right)\right) 
\leq \prod_{v\in S}\Prb\left(\bigwedge_{(e,v)\in\Ic_v}\overline{\A_{e,v}}\right) 
=\prod_{v\in S}\Prb(c\notin \Lo_{\rd}(v)).
\end{equation}
Next, we define
\[
\begin{aligned}
\mu:&=\sum_{v\in S}\sum_{(e, v)\in \Ic_v}\Prb(\A_{e,v})=\sum_{v\in S}\sum_{(e, v)\in \Ic_v}\pi_{\rd}^{|e|-1} & \quad \text{defn. of $\A_{e,v}$}\\
&= \sum_{v\in S}\sum_{\ell=2}^{k}d^{\rd-1}_{\ell}(v, c)\pi_{\rd}^{\ell-1}
\leq \frac{2t_{i-1}}{(\cde\pa_{\rd-1})^{k-1}}\sum_{v\in S}\sum_{\ell=2}^{k}(\cde\pa_{\rd-1}\pi_{\rd})^{\ell-1}           & \text{Proposition~\ref{prop:output}}\\
&=\frac{2t_{i-1}}{(\cde\pa_{\rd-1})^{k-1}}\cdot 2k(\cde\pa_{\rd-1}\pi_{\rd})& \text{by~\eqref{indas4}} \\
&=k\cac.  & \pi_{\rd}=\cac\frac{(\cde\pa_{\rd-1})^{k-2}}{4t_{\rd-1}}
\end{aligned}
\]
Then we have
\[
\begin{aligned}
\Delta^*:&=\sum_{(e, v)}\sum_{(e', v')\sim (e, v)}\Prb(\A_{e, v}\wedge\A_{e',v'})=\sum_{(e, v)}\Prb(\A_{e, v})\sum_{(e', v')\sim (e, v)}\Prb(\A_{e',v'}\mid \A_{e, v}) &\\
&= \sum_{(e, v)}\Prb(\A_{e, v})\sum_{(e', v')\sim (e, v)}\pi_{\rd}^{|(e'-v')-(e-v)|} & \text{defn. of $\A_{e,v}$}\\
&\leq \sum_{(e, v)}\Prb(\A_{e, v})\sum_{q=1}^{|e-v|}\sum_{Q\in\binom{e-v}{q}}\sum_{\substack{(e', v') s.t.\\ v'\notin Q,\ Q+v'\subseteq e' \\ |e'|>q+1}}\pi_{\rd}^{|e'|-q-1} & \text{Claim~\ref{claim:dep}}\\
&\leq \sum_{(e, v)}\Prb(\A_{e, v})\sum_{q=1}^{|e-v|}\sum_{Q\in\binom{e-v}{q}}|S|\sum_{s=q+2}^{k}(\cde\pa_{\rd-1})^{s-q-1}\cdot \pi_{\rd}^{s-q-1} & \text{Proposition~\ref{prop:output}}\\
&\leq\mu\cdot 2^{k-1}\cdot 2k(\cde\pa_{\rd-1}\pi_{\rd})& \text{by~\eqref{indas4}}\\
&\leq 2^{k}k^2(\cac\cde\pa_{\rd-1}\pi_{\rd})\leq \beta\pi_{\rd}\pa_{\rd-1}/120. & \text{defn. of $\cde, \cac, \beta$}
\end{aligned}
\]
Note by \eqref{pibound} that for each $(e, v)$, we have $\Prb(\A_{e, v})\leq \pi_{\rd} \ll 1/2$.
Then applying Theorem~\ref{Janson} and~\eqref{dep:M}, we obtain that
\[
\Prb\left(\bigwedge_{v\in S}\bigwedge_{(e,v)\in\Ic_v}\overline{\A_{e,v}}\right) \leq M\exp(\Delta^*) \leq M(1 + \beta\pi_{\rd}\pa_{\rd-1}/60)\le (1 + \er_0)\prod_{v\in S}\Prb(c\notin \Lo_{\rd}(v)),
\]
where the second inequality uses $\beta\pi_{\rd}\pa_{\rd-1}/120\le 1/2$, following from~\eqref{indas4}. This completes the proof.
\end{proof}

\subsection{Expectations of $\X_{\ell}$ and $\X_{\ell, s}$}
We begin with an additional lemma that establishes some exceptional outcomes.
\begin{lemma}\label{lem:exceout1}
Denote by $\Omega^*_1:=\Omega^*_1(u)$ the set of events where there exists an vertex $v\in N^2_{i-1}(u)$ such that $|\nPa_{\rd}(v)|\leq (1 - \er/8)\beta|\Pa_{\rd-1}(v)|$. Then
\[
\Prb(\Omega^*_1) \leq \exp\left(-\Omega(\log^2\Delta)\right).
\]
\end{lemma}

\begin{proof}
Recall from~\eqref{bound:nbd} that $|N^2_{\rd-1}(u)| \le (k^2\Delta)^2$.
The proof then follows immediately from~Lemma~\ref{lem:tempale} and the union bound.
\end{proof}

\begin{lemma}\label{lem:exp:Xell}
For every $2\leq \ell \leq k$, $\E[\X_{\ell}] \leq  (1 + \er_0)\alpha_{\rd}\beta^{\ell}\sum_{c\in \Pa_{\rd-1}(u)}d^{\rd-1}_{\ell}(u, c).$
\end{lemma}
\begin{proof}
Fix an arbitrary $2\leq \ell \leq k$, and recall from~\eqref{def:keedge} that
\[
\X_{\ell}=\sum_{c\in\Pa_{\rd-1}(u)}\sum_{\substack{e\in\Hc^{\rd-1}_{\ell, c}\\ u\in e}} \I\left[\left(\forall x\in e\setminus\{u\},~x\in U_{\rd}\right) \wedge \left(\forall y\in e,~c\in\nPa_{\rd}(y)\right)\right],
\]
We start with the following claim.
\begin{claim}\label{claim:mono1}
Let $c\in\Pa_{\rd-1}(u)$, and  $e$ be an edge in $\Hc^{\rd-1}_{\ell, c}$ containing $u$. Then for any vertex $x\in e-\{u\}$,
\[
\Prb(x\in U_{\rd} \mid  \forall y\in e,~c\in\nPa_{\rd}(y))\leq \alpha_{\rd}.
\]
\end{claim}
\begin{proof}
Recall from Section~\ref{sec:alg:itera} (step 4) that a vertex $x$ remains uncolored, i.e., in $U_i$, if and only if none of the colors from $\Pa_{\rd-1}(x)$ survive in $\nPa_{\rd}(x)\cap \Ac_{\rd}(x)$. This implies that
\[
\begin{split}
\Prb\left(x\in U_{\rd} \mid \forall y\in e,~c\in\nPa_{\rd}(y)\right)
&=\Prb\left(\nPa_{\rd}(x)\cap\Ac_{\rd}(x)=\emptyset \mid \forall y\in e,~c\in\nPa_{\rd}(y)\right)\\
&\leq\Prb\left(\nPa_{\rd}(x)\cap\Ac_{\rd}(x)-\{c\}=\emptyset \mid \forall y\in e,~c\in\nPa_{\rd}(y)\right)\\
&=\Prb\left(\nPa_{\rd}(x)\cap\Ac_{\rd}(x)-\{c\}=\emptyset\right)\\
&=\Prb\left(\forall c'\in\nPa_{\rd}(x)-\{c\},~\ac^{\rd}_{x, c'}=0\right),
\end{split}
\]
where the second equality follows from the independence of the random variables $\ac$'s and $\ke$'s among different colors.
Note by Lemma~\ref{lem:exceout1} that for any outcome $\omega\in\Omega\setminus\Omega_1^*$, the following holds:
\[
|\nPa_{\rd}(x)|\geq (1 - \er/8)\beta|\Pa_{\rd-1}(x)| \quad \text{for every $x\in N^2_{i-1}(u)$}.
\]
This, together with Lemma~\ref{lem:prb}, shows that
\[
\begin{split}
\Prb\left(x\in U_{\rd} \mid \forall y\in e,~c\in\nPa_{\rd}(y)\right)
&\leq \Prb\left(\forall c'\in\nPa_{\rd}(x)-\{c\},~\ac^{\rd}_{x, c'}=0 \mid \Omega\setminus\Omega_1^*\right) + \Prb\left(\Omega_1^*\right)\\
&\leq (1 - \pi_{\rd})^{(1 - \er/8)\beta|\Pa_{\rd-1}(v)|-1} +  \exp\left(-\Omega(\log^2\Delta)\right) \\
&\leq (1 - \pi_{\rd})^{(1 -  (1 + \er)^{\rd-1}/2)(1 - \er/8)^{\rd}\beta\pa_{\rd-1}-1} +  \exp\left(-\Omega(\log^2\Delta)\right) \\
&\leq(1 - \beta\pi_{\rd}\pa_{\rd-1}/4)+  \exp\left(-\Omega(\log^2\Delta)\right)\leq(1 - \beta\pi_{\rd}\pa_{\rd-1}/5)= \alpha_{\rd},
\end{split}
\]
where the third inequality uses Proposition~\ref{prop:pbou} and assumption~\eqref{indas1}, and the last inequality follows from $\pi_{\rd}\pa_{\rd-1}=\Omega(1/\log\Delta)$ (see~\eqref{indas4}).
\end{proof}

Claim~\ref{claim:mono1} and Lemma~\ref{lem:dep} together show that for any $c\in\Pa_{\rd-1}(u)$ and $e\in\Hc^{\rd-1}_{\ell, c}$ containing $u$,
\[
\begin{split}
\Prb\left(\left(\forall x\in e\setminus\{u\},~x\in U_{\rd}\right) \wedge \left(\forall y\in e,~c\in\nPa_{\rd}(y)\right)\right)
&\leq \Prb\left(\left(x'\in U_{\rd}\right) \wedge \left(\forall y\in e,~c\in\nPa_{\rd}(y)\right)\right)
\\
&\leq \alpha_{\rd}\cdot(1 + \er_0)\beta^{|e|}= (1 + \er_0)\alpha_{\rd}\beta^{\ell},
\end{split}
\]
where $x'$ is an arbitrary vertex in $e \setminus {u}$.
Then by linearity of expectation, we obtain that
\[
\E[\X_{\ell}] \le\sum_{c\in\Pa_{\rd-1}(u)}\sum_{\substack{e\in\Hc^{\rd-1}_{\ell, c}\\ u\in e}} (1 + \er_0)\alpha_{\rd}\beta^{\ell}
=  (1 + \er_0)\alpha_{\rd}\beta^{\ell}\sum_{c\in \Pa_{\rd-1}(u)}d^{\rd-1}_{\ell}(u, c).
\]
\end{proof}

\begin{lemma}\label{lem:exp:Xsell}
For every $2\leq \ell <s\leq k$, $\E[\X_{\ell, s}] \leq (1 + \er_0)\binom{s-1}{s-\ell}\pi_{\rd}^{s- \ell}\beta^{s}\sum_{c\in\Pa_{\rd-1}(u)}d^{\rd-1}_{s}(u, c).$
\end{lemma}
\begin{proof}
Fix arbitrary $2\leq \ell <s\leq k$, and recall from~\eqref{def:inedge} that
\[
\X_{\ell, s}:=\sum_{c\in\Pa_{\rd-1}(u)}\sum_{\substack{e\in \Hc^{\rd-1}_{s, c}\\ u\in e}}\sum_{Q\in\binom{e-u}{s-\ell}}
\I\left[\left(\forall x\in Q,~c\in\Ac_{\rd}(x) \right)\wedge\left(\forall y\in e,~c\in \nPa_{\rd}(y)\right)\right].
\]
We first prove the following claim.
\begin{claim} \label{claim:expe2}
Let $c\in\Pa_{\rd-1}(u)$, $e$ be an edge in $\Hc^{\rd-1}_{s, c}$ containing $u$, and $Q$ be a set in $\binom{e-u}{s-\ell}$. Then
\[
\Prb\left(\left(\forall x\in Q,~c\in\Ac_{\rd}(x) \right)\wedge\left(\forall y\in e,~c\in \nPa_{\rd}(y)\right)\right)
\leq \Prb\left(\forall x\in Q,~c\in\Ac_{\rd}(x)\right)\cdot \Prb\left(\forall y\in e,~c\in \nPa_{\rd}(y)\right).
\]
\end{claim}
\begin{proof}
Let \[\mathcal{N}:=\left\{\I[\ac^{\rd}_{v, c}=1]:~v\in N^2_{i-1}(u)\right\}\cup\left\{\I[\ke^{\rd}_{v, c}=0]:~v\in N_{i-1}(u)\right\}.\]
Recall from Section~\ref{sec:alg:itera} (step 3) that $c$ is in some $\nPa_{\rd}(y)$, if $\ke^{\rd}_{y, c}=1$, and for every $e^*\in \bigcup_{\ell=2}^{k}\Hc^{\rd-1}_{\ell, c}$ containing $y$, there exists a vertex $z\in e^*\setminus \{y\}$ such that $\ac^{\rd}_{z, c}=0$. 
This suggests that the event $\left\{\forall y\in e,~c\in \nPa_{\rd}(y)\right\}$ can be viewed as a family of subsets of $\mathcal{N}$. Moreover, it is monotone decreasing, as $c$ is more likely to survive in a palette if there are more vertices with $\ac^{\rd}_{v, c} = 0$ or $\ke^{\rd}_{v, c} = 1$ (alternatively, fewer vertices with $\ac^{\rd}_{v, c} = 1$ or $\ke^{\rd}_{v, c} = 0$).
One the other hand, the event $\left\{\forall x\in Q,~c\in\Ac_{\rd}(x)\right\}$ is a monotone increasing family of subsets of $\mathcal{N}$.
The claim then follows directly from Theorem~\ref{FKG}.
\end{proof}

Note, by the definition of $A_{\rd}(x)$ (see Section~\ref{sec:alg:itera} (step 1)) and the independence of the $\ac^{\rd}_{x, c}$'s, that $\Prb\left(\forall x\in Q,~c\in\Ac_{\rd}(x) \right)=\pi_{\rd}^{|Q|}$.
This, together with Claim~\ref{claim:expe2} and Lemma~\ref{lem:dep}, shows that
\[
 \Prb\left(\left(\forall x\in Q,~c\in\Ac_{\rd}(x) \right)\wedge\left(\forall y\in e,~c\in \nPa_{\rd}(y)\right)\right)  
\leq (1 + \er_0)\pi_{\rd}^{|Q|}\beta^{|e|}.
\]
Applying linearity of expectation, we obtain that
\begin{equation}\label{exp:code}
\E[\X_{\ell, s}] 
\le \sum_{c\in\Pa_{\rd-1}(u)}\sum_{\substack{e\in \Hc^{\rd-1}_{s, c}\\ u\in e}}\sum_{Q\in\binom{e-u}{s-\ell}}(1 + \er_0)\pi_{\rd}^{|Q|}\beta^{|e|}
= (1 + \er_0)\binom{s-1}{s-\ell}\pi_{\rd}^{s- \ell}\beta^{s}\sum_{c\in\Pa_{\rd-1}(u)}d^{\rd-1}_{s}(u, c).
\end{equation}
\end{proof}

We are now ready to prove Lemma~\ref{lem:expe}.
\begin{proof}[Proof of Lemma~\ref{lem:expe}]
By linearity of expectation and~\eqref{def:sumde}, we have
\[
\E\left[\X_u\right]
= \sum_{\ell=2}^k(\cde\pa_{\rd})^{k-\ell}
\left(\E[\X_{\ell}]
 + \sum_{s=\ell+1}^{k}\E[\X_{\ell, s}]\right).
\]
It then follows from Lemmas~\ref{lem:exp:Xell} and~\ref{lem:exp:Xsell} that $\E\left[\X_u\right]$ is at most
\[
\begin{split}
&\leq \sum_{\ell=2}^k(\cde\pa_{\rd})^{k-\ell}
\left((1 + \er_0)\alpha_{\rd}\beta^{\ell}\sum_{c\in \Pa_{\rd-1}(u)}d^{\rd-1}_{\ell}(u, c) + \sum_{s=\ell+1}^{k}(1 + \er_0)\binom{s-1}{s-\ell}\pi_{\rd}^{s- \ell}\beta^{s}\sum_{c\in\Pa_{\rd-1}(u)}d^{\rd-1}_{s}(u, c)\right) \\
&= (1 + \er_0)\beta^{k}\sum_{c\in \Pa_{\rd-1}(u)}\sum_{\ell=2}^k(\cde\pa_{\rd-1})^{k-\ell}
\Bigg(\alpha_{\rd}d^{\rd-1}_{\ell}(u, c)
 + \sum_{s=\ell+1}^{k}\binom{s-1}{s-\ell}(\pi_{\rd}\beta)^{s- \ell}d^{\rd-1}_{s}(u, c)\Bigg),
\end{split}
\]
where the equality uses $\pa_{\rd}=\beta\pa_{\rd-1}$.
Regrouping terms, we further obtain that $\E\left[\X_u\right]$ is 
\[
\begin{split}
\le~& (1 + \er_0)\beta^{k}\sum_{c\in \Pa_{\rd-1}(u)}\Bigg[\alpha_{\rd}d^{\rd-1}_{2}(u, c)(\cde\pa_{\rd-1})^{k-2}\\
& \quad \quad  \quad \quad  \quad \quad  \quad \quad  \quad \quad\quad \quad+ \sum_{\ell=3}^kd^{\rd-1}_{\ell}(u, c)\Bigg(
\alpha_{\rd}(\cde\pa_{\rd-1})^{k-\ell}
 + \sum_{q=2}^{\ell-1}\binom{\ell-1}{\ell-q}(\pi_{\rd}\beta)^{\ell - q}(\cde\pa_{\rd-1})^{k-q}\Bigg)\Bigg]\\
=~& (1 + \er_0)\beta^{k}\sum_{c\in \Pa_{\rd-1}(u)}\Bigg[\alpha_{\rd}d^{\rd-1}_{2}(u, c)(\cde\pa_{\rd-1})^{k-2} \\
& \quad \quad  \quad \quad  \quad \quad  \quad \quad  \quad \quad\quad \quad+  \sum_{\ell=3}^kd^{\rd-1}_{\ell}(u, c)(\cde\pa_{\rd-1})^{k-\ell}\Bigg(
\alpha_{\rd}
 + \sum_{q=2}^{\ell-1}\binom{\ell-1}{\ell-q}(\cde\beta\pi_{\rd}\pa_{\rd-1})^{\ell - q}\Bigg)\Bigg].
\end{split}
\]
Recall from~\eqref{indas4} that $\cde\pi_{\rd}\pa_{\rd-1} \le 1$, and therefore
\[
\sum_{q=2}^{\ell-1}\binom{\ell-1}{\ell-q}(\cde\beta\pi_{\rd}\pa_{\rd-1})^{\ell - q} 
\leq 2^{\ell-1}\cde\beta\pi_{\rd}\pa_{\rd-1} \leq \beta\pi_{\rd}\pa_{\rd-1}/60,
\]
where the last inequality uses $\cde =1/(60\cdot2^k)$.
Hence,
\[
\begin{split}
\E\left[\X_u\right]
&\leq (1 + \er_0)(\alpha_{\rd} + \beta\pi_{\rd}\pa_{\rd-1}/60)\beta^{k}\sum_{c\in \Pa_{\rd-1}(u)}\sum_{\ell=2}^kd^{\rd-1}_{\ell}(u, c)(\cde\pa_{\rd-1})^{k-\ell}\\
&\leq \alpha'_{\rd}\beta^{k}\sum_{c\in \Pa_{\rd-1}(u)}d_{\rd-1}(u, c)
=\alpha'_{\rd}\beta^{k}|\Pa_{\rd-1}(u)|\Lambda_{\rd-1}(u),
\end{split}
\]
where the last inequality follows from $\er_0=\beta\pi_{\rd}\pa_{\rd-1}/60$, $\alpha_{\rd}=1 - \beta\pi_{\rd}\pa_{\rd-1}/5$ and $\alpha'_{\rd}=1 -\beta\pi_{\rd}\pa_{\rd-1}/6.$

\end{proof}

\section{Proof of Lemma~\ref{lem:concen}: Concentration of average $c$-degrees}\label{sec:concede}
Similarly to expectations, we fix an arbitrary choice of $u \in U_i$ and omit $u$ from notations when the underlying vertex is clear.
We will show that each $\X_{\ell}$ and $\X_{\ell, s}$ is well-concentrated around its expectation using the linear version of Talagrand's inequality (Theorem~\ref{talagrand}), and then combine them using the union bound to prove Lemma~\ref{lem:concen}.

More formally, recall from Table~\ref{table2} that $\theta=1/4k$. For every $2\le \ell\le k$, define
\begin{equation}\label{def:tau1}
\tau_{\ell} :=\alpha'_{\rd}\beta^{k}|\Pa_{\rd-1}(u)|\left(\frac{t_{\rd-1}}{(\cde\pa_{\rd-1})^{k-\ell}}\right)\Delta^{-\theta}.
\end{equation}
which will be the error term in concentrations with respect to uniformity $\ell$.
\begin{lemma}\label{lem:con:Xell}
 For every $2\le \ell\le k$,   
\[
 \Prb\left(|\X_{\ell} - \E[\X_{\ell}] |\leq 2(2\log\Delta)^{2(\ell-1)}\tau_{\ell}\right) \geq 1 - \exp\left(-\Omega\left(\log^2\Delta\right)\right).
\]
\end{lemma}

\begin{lemma}\label{lem:con:Xsell}
 For every $2\leq \ell <s\leq k$, 
\[
\Prb\left(|\X_{\ell, s} - \E[\X_{\ell, s}]|\leq 2\tau_{\ell}\right) \geq 1- \exp\left(-\Omega(\log^2\Delta)\right).
\]
\end{lemma}

We begin this section by showing how Lemma~\ref{lem:concen} is derived from Lemmas~\ref{lem:con:Xell} and~\ref{lem:con:Xsell}. After that, in Section~\ref{sec:con:excp}, we collect some useful lemmas that will later serve as exceptional outcomes in applications of Theorem~\ref{talagrand} for our concentration analysis. We then prove Lemma~\ref{lem:con:Xell} in Section~\ref{sec:con:ell} and Lemma~\ref{lem:con:Xsell} in Section~\ref{sec:con:ells}, respectively.

\begin{proof}[Proof of Lemma~\ref{lem:concen}]
By Lemmas~\ref{lem:con:Xell},~\ref{lem:con:Xsell}, and \eqref{def:sumde}, with probability at least $1- \exp\left(-\Omega(\log^2\Delta)\right)$,
\[
\begin{split}
\X_u - \E\left[\X_u\right] 
& = \sum_{\ell=2}^k(\cde\pa_{\rd})^{k-\ell}\left(
\left(\X_{\ell} - \E\left[\X_{\ell}\right]\right) + \sum_{s=\ell+1}^{k}\left(\X_{\ell, s}  - \E\left[\X_{\ell,s}\right]\right)\right)\\
&\le \sum_{\ell=2}^k(\cde\pa_{\rd})^{k-\ell}\tau_{\ell}\left(
2(2\log\Delta)^{2(\ell-1)}  + \sum_{s=\ell+1}^{k}2\right)\\
& = \alpha'_{\rd}\beta^{k}|\Pa_{\rd-1}(u)|t_{\rd-1}\Delta^{-\theta}\sum_{\ell=2}^k\beta^{k-\ell}\left(
2(2\log\Delta)^{2(\ell-1)}  + \sum_{s=\ell+1}^{k}2\right)\\
& \le \alpha'_{\rd}\beta^{k}|\Pa_{\rd-1}(u)|t_{\rd-1}\Delta^{-\theta}(\log\Delta)^{2k} = (\er/4)\alpha'_{\rd}\beta^{k}|\Pa_{\rd-1}(u)|t_{\rd-1},
\end{split}
\]
where the last inequality holds as $\Delta$ grows sufficiently large. This completes the proof.
\end{proof}

\subsection{Exceptional outcomes}\label{sec:con:excp}
The first two lemmas provide exceptional outcomes that will be used in the concentration analysis of $\X_{\ell}$.

\begin{lemma}\label{lem:exceout2}
Denote by $\Omega^*_2:=\Omega^*_2(u)$ the set of events where there exists a vertex $v\in N^2_{i-1}(u)$ such that $|\Ac_{\rd}(v)|\geq \log^2\Delta$. Then
\[
\Prb(\Omega^*_2) \leq \exp\left(- \Omega\left(\log^2\Delta\right)\right).
\]
\end{lemma}
\begin{proof}
Recall from Section~\ref{sec:alg:itera} (step 1) that for every vertex $v\in N^2_{\rd-1}(u)$, 
\[
\E[|A_{\rd}(v)|] = \sum_{c\in \Pa_{\rd-1}(v)}\E[\ac^{\rd}_{v, c}]=
\pi_{\rd}|\Pa_{\rd-1}(v)|\leq \pi_{\rd}\pa_{\rd-1} \le 1,
\] 
where the inequalities follow from~\eqref{def:palet} and~\eqref{indas4} respectively.
Set $\delta:=\frac{\log^2\Delta}{\E[A_{\rd}(v)] } - 1$. 
By Lemma~\ref{chernoff} (Chernoff bounds), we obtain that
\[
\begin{split}
\Prb\left[|\Ac_{\rd}(v)|\geq \log^2\Delta \right]
&=\Prb\left[|\Ac_{\rd}(v)|\geq \left(1 + \delta\right)\E[A_{\rd}(v)]\right]
\leq \exp\left(-\delta^2\E[|A_{\rd}(v)|]/(2 + \delta)\right)\\
&= \exp\left(- \Omega\left( \delta \E[A_{\rd}(v)] \right)\right)
=\exp\left(- \Omega\left(\log^2\Delta\right)\right).
\end{split}
\]
Then, using the union bound and~\eqref{bound:nbd}, we show that
\[
\Prb(\Omega^*_2) \leq |N^2_{\rd-1}(u)|\exp\left(- \Omega\left(\log^2\Delta\right)\right)
\leq k^4\Delta^2\exp\left(- \Omega\left(\log^2\Delta\right)\right)=\exp\left(- \Omega\left(\log^2\Delta\right)\right).
\]
\end{proof}

For every $2\le \ell \le k$, define 
\begin{equation}\label{def:indexset}
M_{\ell}:=\{(e, c):\ c\in \Pa_{\rd-1}(u),\ u\in e\in \Hc^{\rd-1}_{\ell, c}\}.
\end{equation}
\begin{lemma}\label{lem:exceout3}
For every $2\le \ell \le k$, denote by $\Omega^*_{3,\ell}:=\Omega^*_3(u, \ell)$ the set of events where there exists a color $c'\in \bigcup_{v\in N^2_{i-1}(u)}\Pa_{\rd-1}(v)$ such that\[
\left|\left\{(e, c)\in M_{\ell}:~c'\in \bigcup_{x\in e\setminus\{u\}}\Ac_{\rd}(x)\right\}\right|
\ge |\Pa_{\rd-1}(u)|(\cde\pa_{\rd-1})^{\ell-2}\log^2\Delta.
\] 
Then $\Prb(\Omega^*_{3,\ell}) \leq \exp\left(-\Omega\left(\log^2\Delta\right)\right).$
\end{lemma}
\begin{proof}
Fix an arbitrary color $c'\in \bigcup_{v\in N^2_{i-1}(u)}\Pa_{\rd-1}(v)$. 
For simplicity of notation, for every $(e,c)\in M_{\ell}$, let $\I_{e,c}$ be the indicator variable for the event $c'\in \bigcup_{x\in e\setminus\{u\}}\Ac_{\rd}(x)$, i.e., \begin{equation}\label{def:indi:out3}
 \I_{e,c}:=\I\left[c'\in \bigcup_{x\in e\setminus\{u\}}\Ac_{\rd}(x)\right],   \end{equation}
and observe that
\[
\left|\left\{(e, c)\in M_{\ell}:~c'\in \bigcup_{x\in e\setminus\{u\}}\Ac_{\rd}(x)\right\}\right|=\sum_{(e,c)\in M_{\ell}}\I_{e,c}:=\X_{c'}.
\] 
Note, by the definition of $\Ac_{\rd}(\cdot)$ (see Section~\ref{sec:alg:itera}, step 1), that $\E[\I_{e, c}]=\sum_{x\in e\setminus\{u\}}\Prb(\ac^i_{x,c'}=1)\leq (\ell-1)\pi_{\rd}$. 
Then we have 
\[
\begin{array}{lll}
\E[\X_{c'}]&=\sum_{c\in\Pa_{i-1}(u)}\sum_{u\in e\in\Hc^{\rd-1}_{\ell, c}}\E[\I_{e,c}]\leq \sum_{c\in\Pa_{i-1}(u)}d^{\rd-1}_{\ell}(u, c) (\ell-1)\pi_{\rd} & \\
&\leq|\Pa_{\rd-1}(u)|\frac{2t_{\rd-1}}{(\cde\pa_{\rd-1})^{k-\ell}}(\ell-1)\pi_{\rd} & \text{Proposition~\ref{prop:output}} \\
&=\Theta\left(|\Pa_{\rd-1}(u)|(\cde\pa_{\rd-1})^{\ell-2}\right).
& \pi_{\rd}=\Theta\left(\frac{(\cde\pa_{\rd-1})^{k-2}}{t_{\rd-1}}\right)
\end{array}
\]

\begin{claim}
$\X_{c'}$ is $(1, |\Pa_{\rd-1}(u)|(\cde\pa_{\rd-1})^{\ell-2})$-observable in $\Omega$. 
\end{claim}
\begin{proof}
Let $(e,c)\in M_{\ell}$. For any $\omega\in \Omega$ with $\I_{e, c}(\omega)=1$, by definition~\eqref{def:indi:out3}, under the outcome $\omega$, there exists a vertex $x \in e \setminus {u}$ such that $c' \in \Ac_{\rd}(x)$, i.e., $\ac^{\rd}_{x,c'}(\omega)=1$. 
We can then define the verifier of $\I_{e, c}$ as $R_{e, c}(\omega):=\{\ac^{\rd}_{x, c'}\}$. 
By Definition~\ref{def:verfiable}, $\I_{e, c}$ is $1$-verifiable with the verifier $R_{e, c}$.

Observe that a random variable $\ac^{\rd}_{x, c^*}$ is in the verifier of some $\I_{e,c}$, only if $x\in e\setminus u$. Then for every $\omega\in\Omega$ and random variable $\ac^{\rd}_{x, c^*}$,
\[
\begin{split}
|\{(e, c)\in M_{\ell}:\ \I_{e,c}(\omega)=1,~\ac^{\rd}_{x, c^*}\in R_{e,c}(\omega)\}|
&\le |\{(e, c)\in M_{\ell}:\ x\in e \setminus {u}\}|
\leq \sum_{c\in\Pa_{\rd-1}(u)}\delta_{2, \ell}(\Hc^{\rd-1}_{\ell,c})\\
&\leq |\Pa_{\rd-1}(u)|(\cde\pa_{\rd-1})^{\ell-2},
\end{split}
\]
where the last inequality follows from~Proposition~\ref{prop:output} (ii). By Definition~\ref{def:obser}, this shows that $\X_{c'}$ is $(1, |\Pa_{\rd-1}(u)|(\cde\pa_{\rd-1})^{\ell-2})$-observable.
\end{proof}

Set $\tau:=|\Pa_{\rd-1}(u)|(\cde\pa_{\rd-1})^{\ell-2}\log^2\Delta - \E[\X_{c'}]$, and note that $\tau \gg \E[\X_{c'}]$.
Applying Theorem~\ref{talagrand} on $\X_{c'}$ with $\tau$ and $\Omega^*=\emptyset$, we have
\[
\begin{split}
\Prb(\X_{c'} \geq |\Pa_{\rd-1}(u)|(\cde\pa_{\rd-1})^{\ell-2}\log^2\Delta)
&\leq \Prb(|\X_{c'} - \E[\X_{c'}]| \geq \tau)\\
& \leq 4\exp\left( -\frac{\tau^2}{8|\Pa_{\rd-1}(u)|(\cde\pa_{\rd-1})^{\ell-2})(4\E[\X_{c'}] + \tau)}\right) \\
& \leq 4\exp\left( -\frac{\tau}{16|\Pa_{\rd-1}(u)|(\cde\pa_{\rd-1})^{\ell-2})}\right)
= \exp\left( -\Omega(\log^2\Delta)\right).
\end{split}
\]

Lastly, note by~\eqref{bound:nbd} and the definition of palettes that $\sum_{v\in N^2_{i-1}(u)}|\Pa_{\rd-1}(v)| \le (k^2\Delta)^2\cdot \Cco\le k^4\Delta^3.$
This, together with the union bound, shows that 
\[
\Prb(\Omega^*_{3,\ell})\le \sum_{c'}\Prb(\X_{c'} \geq |\Pa_{\rd-1}(u)|(\cde\pa_{\rd-1})^{\ell-2}\log^2\Delta)
\leq k^4\Delta^3 \exp\left( -\Omega(\log^2\Delta)\right)=\exp\left(-\Omega\left(\log^2\Delta\right)\right),
\]
completing the proof.
\end{proof}

The next lemma establishes exceptional outcomes for the concentration analysis of $\X_{\ell, s}$.
\begin{lemma}\label{lem:outcome5}
For every $2\le \ell < s \le k$, denote by $\Omega^*_{4,s,\ell}:=\Omega^*_{4}(u, s,\ell)$ the set of events where there exists a color $c\in \Pa_{\rd-1}(u)$ such that 
\[
\left|\left\{e\in \Hc^{\rd-1}_{s, c} \mid u\in e,\ \sum_{v\in e-u}\ac^{\rd}_{v,c}\geq s-\ell\right\}\right| > (\cde p_{i-1})^{\ell-1}\log^{2(s-\ell)}\Delta.
\]
Then $\Prb(\Omega^*_{4,s,\ell}) \leq \exp\left(-\Omega(\log^2\Delta)\right).$
\end{lemma}

The proof of Lemma~\ref{lem:outcome5} is achieved by iteratively building exceptional outcome spaces through repetitive applications of Theorem~\ref{talagrand}. As a warm-up, let us first prove the following result.
\begin{lemma}\label{lem:outcome4}
Let $2\le \ell < s \le k$, $0\leq m\leq s-\ell-1$ and $1\leq a\leq s-m-1$. Denote by $\Omega^*_{s, \ell, m, a}$ the set of events where there exists a set $A\in N^2_{i-1}(u)$ of size $a$ and a color $c\in \Pa_{\rd-1}(u)$ such that 
\[
\left|\left\{e\in \Hc^{\rd-1}_{s, c} \mid A\cup\{u\}\subseteq e,\ \sum_{v\in e-A\cup\{u\}}\ac^{\rd}_{v,c}\geq m\right\}\right| > (\cde p_{i-1})^{s-(a+1)-m}\log^{2m}\Delta.
\]
Then $\Prb(\Omega^*_{s, \ell, m,a}) \leq \exp\left(-\Omega(\log^2\Delta)\right).$
\end{lemma}
\begin{proof}
Fix arbitrary $2\le \ell < s \le k$. 
We will prove the lemma by induction on $m\geq 0$.
For $m=0$ and $1\le a \le s-1$, the lemma is trivially true with $\Prb(\Omega^*_{s,\ell,0,a})=0$, as by Proposition~\ref{prop:output} (ii), we have $\delta_{a+1, s}(\Hc^{\rd-1}_{s, c}) \leq  (\cde p_{i-1})^{s-(a+1)}$ for every $c$ and $a\ge 1$.

Let $m\ge 1$ and assume by induction that the lemma holds for $m-1$.
In particular, for every $1\leq a \leq s-m-1$,
\begin{equation}\label{expout4:indas}
\Prb(\Omega^*_{s,\ell,m-1, a+1})\leq \exp\left(-\Omega(\log^2\Delta)\right).
\end{equation}
For simplicity of notation, for a set $A\subseteq N^2_{i-1}(u)$ and a color $c\in \Pa_{\rd-1}(u)$, let
 $\Ic(A, c):=\{e\in \Hc^{\rd-1}_{s, c} \mid A\cup\{u\}\subseteq e\}$ be an index set.
Then for every integer $m\geq 1$ and $e\in \Ic(A, c)$, denote by $\I_{e, m, A, c}$ the indicator variable for the event that $\sum_{v\in e-A\cup\{u\}}\ac^{\rd}_{v,c}\geq m$, i.e., 
\begin{equation}\label{def:indi:out4}
\I_{e, m, A, c}:=\I\left[\sum_{v\in e-A\cup\{u\}}\ac^{\rd}_{v,c}\geq m\right].
 \end{equation}
 and observe that
\[
\left|\left\{e\in \Hc^{\rd-1}_{s, c} \mid A\cup\{u\}\subseteq e,\ \sum_{v\in e-A\cup\{u\}}\ac^{\rd}_{v,c}\geq m\right\}\right| = \sum_{e\in \Ic(A, c)}\I_{e, m, A, c}:=\X_{m, A,c}.
\]
Note that for any set $A$ of size $a\ge 1$, 
\[
\begin{split}
\E[\X_{m, A, c}]
&\leq |\Ic(A, c)|\binom{s}{m}\pi_{\rd}^m
\leq \delta_{a+1, s}(\Hc^{\rd-1}_{s, c})\binom{s}{m}\pi_{\rd}^m
\leq \binom{s}{m}(\cde\pa_{\rd-1})^{s-(a+1)}\pi_{\rd}^m\\
&= \binom{s}{m}(\cde\pa_{\rd-1})^{s-(a+1)-m}(\cde\pa_{\rd-1}\pi_{\rd})^m
\leq \binom{s}{m}(\cde\pa_{\rd-1})^{s-(a+1)-m},
\end{split}
\]
where the third inequality again follows from~Proposition~\ref{prop:output}, and the last inequality uses $\cde\pa_{\rd-1}\pi_{\rd}\leq 1$, as given by~\eqref{indas4}.

\begin{claim}\label{outcome4:rdobe}
Let $1\leq a \leq s-m-1$. For every set $A\in N^2_{i-1}(u)$ of size $a$ and color $c\in \Pa_{\rd-1}(u)$,
\[
\text{$X_{m, A, c}$ is $\left(m, (\cde p_{i-1})^{s-(a+1)-m}\log^{2(m-1)}\Delta\right)$-observable with respect to $\Omega^*_{s,\ell,m-1, a+1}$}.
\]
\end{claim}
\begin{proof}
Let $e\in \Ic(A, c)$. For any $\omega\in\Omega\setminus\Omega^*_{s,\ell,m-1, a+1}$ with $\I_{e, m, A, c}(\omega)=1$, by definition~\eqref{def:indi:out4}, under the outcome $\omega$, there must exist a set $A'\subseteq e-A\cup\{u\}$ of size $m$ such that $\ac^{\rd}_{x,c}(\omega)=1$ for all $x\in A'$.
We can then define the verifier of $\I_{e, m, A, c}$ as $R_{e, m, A, c}(\omega):=\{\ac^{\rd}_{x,c},\ x\in A'\}$.
By Definition~\ref{def:verfiable}, $\I_{e, m, A, c}$ is $m$-verifiable with the verifier $R_{e, m, A, c}$.

Observe that a random variable $\ac^{\rd}_{x, c^*}$ is contained in some $R_{e, m, A, c}(w)$, only if $x\in e- A\cup\{u\}$, and $\sum_{v\in e-A\cup\{u, x\}}\ac^{\rd}_{v,c}(\omega)\geq m-1$. Then for every $\omega\in\Omega\setminus\Omega^*_{s,\ell,m-1, a+1}$ and random variable $\ac^{\rd}_{x, c^*}$,
\[
\begin{split}
&|\{
e\in \Ic(A,c) \mid \I_{e, m, A, c}(\omega)=1,\ \ac^{\rd}_{x,c^*}\in R_{e,m, A, c}(\omega)
\}|\\
\leq~&\left|\left\{e\in \Hc^{\rd-1}_{s, c} \mid A\cup\{u,x\}\subseteq e,\ \sum_{v\in e-A\cup\{u,x\}}\ac^{\rd}_{v,c}(\omega)\geq m-1\right\}\right|\\
\leq~&(\cde p_{i-1})^{s-(a+2)-(m-1)}\log^{2(m-1)}\Delta= (\cde p_{i-1})^{s-(a+1)-m}\log^{2(m-1)}\Delta,
\end{split}
\]
where the last inequality follows from $\omega\notin\Omega^*_{s,\ell,m-1, a+1}$ and the definition of $\Omega^*_{s,\ell,m-1, a+1}$.
By Definition~\ref{def:obser}, this shows that $X_{m, A, c}$ is $\left(m, (\cde p_{i-1})^{s-(a+1)-m}\log^{2(m-1)}\Delta\right)$-observable with respect to $\Omega^*_{s,\ell,m-1, a+1}$.
\end{proof}

Set 
\[
\tau_{m,a}:=\frac12(\cde p_{\rd-1})^{s-(a+1)-m}\log^{2m}\Delta,
\]
and observe that $\tau_{m,a}\gg \E[\X_{m, A, c}]$ for every $m\geq 1$ and set $A$ with $|A|=a \ge 1$. Applying Theorem~\ref{talagrand} on $\X_{m, A, c}$ with $\tau_{m, a}$ and $\Omega^*_{s, \ell,m-1, a+1}$, we obtain that
\[
\begin{split}
&\Prb\left(\X_{m, A, c}> (\cde p_{i-1})^{s-(a+1)-m}\log^{2m}\Delta\right)
\leq \Prb(|\X_{m, A, c} - \E[\X_{m, A, c}]|>\tau_{m, a})\\
\leq\ &4\exp\left(-\frac{\tau_{m, a}^2}{8m(\cde p_{i-1})^{s-(a+1)-m}\log^{2(m-1)}\Delta(4\E[\X_{m, A, c} ] + \tau_{m, a})}\right) + 4\Prb(\Omega^*_{s, \ell, m-1, a+1})\\
\leq\ &4\exp\left(-\frac{\tau_{m, a}}{16m(\cde p_{i-1})^{s-(a+1)-m}\log^{2(m-1)}\Delta}\right) + 4\exp\left(-\Omega(\log^2\Delta)\right)\\
\leq\ &4\exp\left(-\Omega(\log^2\Delta)\right) + 4\exp\left(-\Omega(\log^2\Delta)\right)=\exp\left(-\Omega(\log^2\Delta)\right),
\end{split}
\]
where the second inequality uses the induction hypothesis~\eqref{expout4:indas}.

Lastly, by~\eqref{bound:nbd} and the definition of palettes, we have $\binom{|N^2_{i-1}(u)|}{a}\cdot|\Pa_{\rd-1}(u)| \le (k^2\Delta)^{2a}\cdot \Cco\le k^{4k}\Delta^{2k+1}.$
This, together with the union bound, shows that 
\[
\Prb(\Omega^*_{s, \ell, m,a}) \leq \sum_{A\in\binom{N^2_{\rd-1}(u)}{a}}\sum_{c\in\Pa_{\rd-1}(u)}\Prb\left(\X_{m, A, c}>(\cde p_{i-1})^{s-(a+1)-m}\log^{2m}\Delta\right) \leq \exp\left(-\Omega(\log^2\Delta)\right),
\]
which completes the proof.
\end{proof}

Using $\Omega^*_{s, \ell, s-\ell-1, 1}$ from Lemma~\ref{lem:outcome4} as the exceptional outcome space, we now apply Theorem~\ref{talagrand} once again to prove Lemma~\ref{lem:outcome5}.

\begin{proof}[Proof of Lemma~\ref{lem:outcome5}]
Fix arbitrary $2\le \ell < s \le k$.
For simplicity of notation, for a color $c\in \Pa_{\rd-1}(u)$, let
 $\Ic(c):=\{e\in \Hc^{\rd-1}_{s, c} \mid u\in e\}$ be an index set. 
 Then for every $e\in \Ic(c)$, denote by $\I_{e, c}$ the indicator variable for the event that $\sum_{v\in e-u}\ac^{\rd}_{v,c}\geq s-\ell$, i.e.,
\begin{equation}\label{def:indi:out5}
\I_{e, c}:=\I\left[\sum_{v\in e-u}\ac^{\rd}_{v,c}\geq s-\ell\right],
 \end{equation}
and observe that
\[
\left|\left\{e\in \Hc^{\rd-1}_{s, c} \mid u\in e,\ \sum_{v\in e-u}\ac^{\rd}_{v,c}\geq s-\ell\right\}\right|=\sum_{e\in \Ic(c)}\I_{e, c}:=\X_{c}.
\]
Note that
\[
\begin{split}
\E[\X_{c}]
&\leq |\Ic(c)|\binom{s-1}{s-\ell}\pi_{\rd}^{s-\ell}
\leq d^{\rd-1}_{s}(u, c)\binom{s-1}{s-\ell}\pi_{\rd}^{s-\ell}
\leq \binom{s-1}{s-\ell}\frac{2t_{\rd-1}}{(\cde\pa_{\rd-1})^{k-s}}\pi_{\rd}^{s-\ell}\\
&=\binom{s-1}{s-\ell}\frac{2t_{\rd-1}\pi_{\rd}}{(\cde\pa_{\rd-1})^{k-\ell-1}}(\cde\pa_{\rd-1}\pi_{\rd})^{s-\ell-1}
\leq \binom{s-1}{s-\ell}\frac{2t_{\rd-1}\pi_{\rd}}{(\cde\pa_{\rd-1})^{k-\ell-1}}
=\binom{s-1}{s-\ell}\frac{\cac}{2}(\cde\pa_{\rd-1})^{\ell-1},
\end{split}
\]
where the third inequality follows from Proposition~\ref{prop:output}, and the last inequality uses $\cde\pa_{\rd-1}\pi_{\rd}\leq 1$, as given by~\eqref{indas4}.
\begin{claim}\label{outcome5:rdobe}
For every color $c\in \Pa_{\rd-1}(u)$,
\[
\text{$\X_{c}$ is $\left(s-\ell, (\cde p_{i-1})^{\ell-1}\log^{2(s-\ell-1)}\Delta\right)$-observable with respect to $\Omega^*_{s, \ell, s-\ell-1, 1}$}.
\]
\end{claim}
\begin{proof}
Let $e\in\Ic(c)$.
For any $\omega\in \Omega\setminus\Omega^*_{s, \ell, s-\ell-1, 1}$ with $\I_{e, c}(\omega)=1$, by definition~\eqref{def:indi:out5},
 there must exist a set $A\subseteq e-u$ of size $s-\ell$ such that $\ac^{\rd}_{x,c}=1$ for all $x\in A$.
We can then define the verifier of $\I_{e, c}$ as $R_{e, c}(\omega):=\{\ac^{\rd}_{x,c},\ x\in A\}$.
By Definition~\ref{def:verfiable}, $\I_{e, c}$ is $(s-\ell)$-verifiable with verifier $R_{e, c}$.

Observe that a random variable $\ac^{\rd}_{x, c^*}$ is contained in some $R_{e, c}(w)$, only if $\sum_{v\in e-\{u, x\}}\ac^{\rd}_{v,c}(\omega)\geq s-\ell-1$ and $x\in e- u$. Then for every $\omega\in\Omega\setminus\Omega^*_{s,\ell,s-\ell-1, 1}$ and random variable $\ac^{\rd}_{x, c^*}$,
\[
\begin{split}
|\{
e\in \Ic(c) \mid \I_{e, c}(\omega)=1,\ \ac^{\rd}_{x,c^*}\in R_{e, c}(\omega)
\}|
&\leq\left|\left\{e\in \Hc^{\rd-1}_{s, c} \mid \{u, x\}\subseteq e,\ \sum_{v\in e-\{u, x\}}\ac^{\rd}_{v,c}(\omega)\geq s-\ell-1\right\}\right|\\
&\leq(\cde p_{i-1})^{s-2-(s-\ell-1)}\log^{2(s-\ell-1)}\Delta\\
&= (\cde p_{i-1})^{\ell-1}\log^{2(s-\ell-1)}\Delta,
\end{split}
\]
where the last inequality follows from $\omega\notin\Omega^*_{s, \ell, s-\ell-1, 1}$ and the definition of $\Omega^*_{s, \ell, s-\ell-1, 1}$.
By Definition~\ref{def:obser}, this shows that $\X_{c}$ is $\left(s-\ell, (\cde p_{i-1})^{\ell-1}\log^{2(s-\ell-1)}\Delta\right)$-observable with respect to $\Omega^*_{s, \ell, s-\ell-1, 1}$.
\end{proof}
Set 
\[
\tau:=\frac12(\cde p_{\rd-1})^{\ell-1}\log^{2(s-\ell)}\Delta,
\]
and observe that $\tau \gg \E[\X_{c}]$. 
Applying Theorem~\ref{talagrand} on $\X_{c}$ with $\tau$ and $\Omega^*_{s,\ell,s-\ell-1, 1}$, we obtain that
\[
\begin{split}
&\Prb\left(\X_{c}> (\cde p_{i-1})^{\ell-1}\log^{2(s-\ell)}\Delta\right)
\leq \Prb(|\X_{c} - \E[\X_{c}]|>\tau)\\
\leq~&4\exp\left(-\frac{\tau^2}{8(s-\ell)(\cde p_{i-1})^{\ell-1}\log^{2(s-\ell-1)}\Delta(4\E[\X_{c} ] + \tau)}\right) + 4\Prb(\Omega^*_{s,\ell, s-\ell-1, 1})\\
\leq~&4\exp\left(-\frac{\tau}{16(s-\ell)(\cde p_{i-1})^{\ell-1}\log^{2(s-\ell-1)}\Delta}\right) + 4\exp\left(-\Omega(\log^2\Delta)\right)\\
\leq~&\exp\left(-\Omega(\log^2\Delta)\right),
\end{split}
\]
where the second inequality follows from Lemma~\ref{lem:outcome4}.

Lastly, by the union bound, we have
\[
\Prb(\Omega^*_{4, s, \ell}) \leq \sum_{c\in\Pa_{\rd-1}(u)}\Prb\left(\X_{c}> (\cde p_{i-1})^{\ell-1}\log^{2(s-\ell)}\Delta\right) \leq \exp\left(-\Omega(\log^2\Delta)\right),
\]
which completes the proof.
\end{proof}

\subsection{Proof of Lemma~\ref{lem:con:Xell}: Concentration of $\X_{\ell}$}\label{sec:con:ell}
Fix an arbitrary $2\leq \ell \leq k$, and recall from~\eqref{def:keedge} that
\[
\X_{\ell}=\sum_{c\in\Pa_{\rd-1}(u)}\sum_{\substack{e\in\Hc^{\rd-1}_{\ell, c}\\ u\in e}} \I\left[\left(\forall x\in e\setminus\{u\},~x\in U_{\rd}\right) \wedge \left(\forall y\in e,~c\in\nPa_{\rd}(y)\right)\right].
\]
Unfortunately, we are unable to directly show that $\X_{\ell}$ is $(r, d)$-certifiable with respect to any suitable set of exceptional outcomes. However, we can express $\X_{\ell}$ as a linear combination of several random variables that are $(r, d)$-certifiable, and then apply Theorem~\ref{talagrand} to each of them.

More formally, recall from~\eqref{def:indexset} that 
$
M_{\ell}=\{(e, c):\ c\in \Pa_{\rd-1}(u),\ u\in e\in \Hc^{\rd-1}_{\ell, c}\}.
$
Define
\[
\X^{1}_{\ell}:=\left|\left\{
(e, c)\in M_{\ell}:~\forall x\in e\setminus\{u\},~x\in U_{\rd}\right\}\right|
\]
and 
\[
\X^{2}_{\ell}:=\left|\left\{
(e, c)\in M_{\ell}:~\left(\forall x\in e\setminus\{u\},~x\in U_{\rd}\right)\wedge\left(\exists y\in e,~c\notin\nPa_i(y)\right)\right\}\right|
\]
Observe that $\X_{\ell}=\X_{\ell}^1 - \X_{\ell}^2$. As mentioned above, we will analyze the concentration of each $\X_{\ell}^1$ and $\X_{\ell}^2$ individually, and then combine the results (using the union bound) to obtain the concentration for $\X_{\ell}$.

First, let us consider $\X_{\ell}^1$. To analyze $\X_{\ell}^1$, we need to further decompose it into more variables. 
Note that for any hyperedge $e$ containing $u$, we can always represent it as a unique sequence $e = uv_1 \ldots v_{|e|-1}$ (based on an arbitrary prefixed ordering on the vertices of $V(\Hc)$).
For integers $i_1, i_2, \ldots, i_{\ell-1}$, let
\begin{multline*}
\rva_{i_1,\ldots,i_{\ell-1}}:=\left|\left\{(e, c)\in M_{\ell}:~e = uv_1 \ldots v_{\ell} \text{ s.t. } \forall r\in [\ell-1],~|\Ac_{\rd}(v_r)|=i_r\right.\right. \\
\left.\left.\text{ and } \left|\Ac_{\rd}(v_r)\cap\left(\Pa_{\rd-1}(v_r) - \nPa_{\rd}(v_r)\right)\right|=i_r\right\}\right|. 
\end{multline*}
Recall from Section~\ref{sec:alg:itera} (step 4) that a vertex $v\in U_i$ if and only if $\nPa_{\rd}(v) \cap \Ac_{\rd}(v)=\emptyset$. Then we have
\[
\X_{\ell}^1 = \sum_{0\leq i_1,\ldots,i_{\ell-1}\leq \Cco}\rva_{i_1 ,\ldots,i_{\ell-1}}.
\]
As $C\sim (\Delta/\log\Delta)^{1/(k-1)}$, $\X_{\ell}^1$ consists of too many variables, so even if we concentrate each of them very well, the error may still blow up when we combine them.  
This is where we use Lemma~\ref{lem:exceout2}: we can exclude a small set $\Omega^*_2$ of exceptional outcomes, such that in the remaining probability space, every vertex in $N^2_{\rd-1}(u)$ has at most $\log^2 \Delta$ activated colors.
More formally, we set
\begin{equation}\label{eq:Ydef}
\Y_{\ell}^1 := \sum_{0\leq i_1,\ldots,i_{\ell-1}\leq \log^2\Delta}\rva_{i_1 ,\ldots,i_{\ell-1}},
\end{equation}
and Lemma~\ref{lem:exceout2} then shows that
\begin{equation}\label{eq:appronum}
\Prb(\X_{\ell}^1 = \Y_{\ell}^1) \geq 1- \Prb(\Omega^*_2)\geq 1 -  \exp\left(- \Omega\left(\log^2\Delta\right)\right).
\end{equation}
Thus, we can instead analyze $\Y_{\ell}^1$ rather than $\X_{\ell}^1$.

Another issue is that, since $\rva_{i_1 ,\ldots,i_{\ell-1}}$ requires the corresponding sets to be of exact sizes, handling such variables is still challenging. To overcome this obstacle, we further introduce the following new variables: 
for integers $i_1, \ldots, i_{\ell-1}$ and $j_1, \ldots, j_{\ell-1}$, let
\begin{multline*}
\rvb_{i_1,\ldots,i_{\ell-1}}^{j_1,\ldots,j_{\ell-1}}:=\left|\left\{(e, c)\in M_{\ell}:~e = uv_1 \ldots v_{\ell} \text{ s.t. } \forall r\in [\ell-1],~|\Ac_{\rd}(v_r)|\ge i_r,\right.\right. \\
\left.\left.\text{ and } \left|\Ac_{\rd}(v_r)\cap\left(\Pa_{\rd-1}(v_r) - \nPa_{\rd}(v_r)\right)\right|\ge j_r\right\}\right|. 
\end{multline*}
\begin{prop}\label{prop:lincomb} For any integers $i_1, \ldots,i_{\ell-1} \geq 0$, 
\[
\rva_{i_1, \ldots,i_{\ell-1}}=\sum_{\substack{\sigma_r, \tau_r\in\{0,1\}\\  \forall r\in[\ell-1]}} (-1)^{f(\sigma_1, \tau_1, \ldots, \sigma_{\ell-1}, \tau_{\ell-1})}\rvb_{i_1+\sigma_1, \ldots, i_{\ell-1} + \sigma_{\ell-1}}^{i_1 +\tau_1,\ldots, i_{\ell-1} + \tau_{\ell-1} }, 
\]
where $f(\sigma_1, \tau_1, \ldots, \sigma_{\ell-1}, \tau_{\ell-1}):= |\{r\in [\ell-1]:~\sigma_r\neq \tau_r\}|.$
\end{prop}
The proof of Proposition~\ref{prop:lincomb} is elementary: for $\ell=2$, observe from the definition that $a_{i_1}=(b_{i_1}^{i_1} - b_{i_1}^{i_1+1}) - (b_{i_1+1}^{i_1} - b_{i_1+1}^{i_1+1})$; the cases for larger $\ell$ follow from applying the same argument to each coordinate $r\in [\ell-1]$, and we omit the details.

Combining~\eqref{eq:Ydef} and Proposition~\ref{prop:lincomb}, we obtain that
\begin{equation}\label{eq:Ydefnew}
\Y_{\ell}^1=\sum_{\substack{0\leq i_r\leq \log^2\Delta\\  \forall r\in[\ell-1]}}\sum_{\substack{\sigma_r, \tau_r\in\{0,1\}\\  \forall r\in[\ell-1]}} (-1)^{f(\sigma_1, \tau_1, \ldots, \sigma_{\ell-1}, \tau_{\ell-1})}\rvb_{i_1+\sigma_1, \ldots, i_{\ell-1} + \sigma_{\ell-1}}^{i_1 +\tau_1, \ldots, i_{\ell-1} + \tau_{\ell-1}}.
\end{equation}
We will apply Theorem~\ref{talagrand} to each $\rvb_{i_1, \ldots,i_{\ell-1}}^{j_1, \ldots, j_{\ell-1}}$ individually, where
\begin{equation}\label{eq:rrange}
0\leq i_r, j_r\leq \log^2\Delta + 1 \quad \text{for all } r\in[\ell-1].
\end{equation}
\begin{claim}\label{claim:rdver1}
For every integers $i_1, \ldots, i_{\ell-1}$ and $j_1, \ldots, j_{\ell-1}$ that satisfies~\eqref{eq:rrange}, 
\[
\text{$\rvb_{i_1, \ldots,i_{\ell-1}}^{j_1, \ldots, j_{\ell-1}}$ is $\left(\ell k\log^2\Delta,\ 2|\Pa_{\rd-1}(u)|(\cde\pa_{\rd-1})^{\ell-2}\log^2\Delta\right)$-observable with respect to $\Omega^*_{3,\ell}$,}
\]
where $\Omega^*_{3,\ell}$ is defined as in Lemma~\ref{lem:exceout3}.
\end{claim}
\begin{proof}
For every $(e,c)\in M_{\ell}$ with $e=uv_1\ldots v_{\ell}$, let $\I_{e,c}$ denote the indicator variable for the event that
\begin{itemize}
\item $|\Ac_{\rd}(v_r)|\geq i_r$ for all $r\in[\ell-1]$;
\item $|\Ac_{\rd}(v_r)\cap(\Pa_{\rd-1}(v_r) - \nPa_{\rd}(v_r))|\geq j_r$ for all $r\in[\ell-1]$.
\end{itemize}
Observe that
\[
\rvb_{i_1, \ldots,i_{\ell-1}}^{j_1, \ldots, j_{\ell-1}} = \sum_{(e, c)\in M_{\ell}} \I_{e,c}
\]

We first show that every $\I_{e,c}$ is $(\ell k\log^2\Delta)$-verifiable.
For any $\omega\in\Omega\setminus \Omega^*_{3,\ell}$ with $\I_{e,c}(\omega)=1$, by the definition of $\I_{e,c}$, under the outcome $\omega$, for every $r\in [\ell-1]$, there exist a set $C^1_{v_r}\subseteq \Pa_{\rd-1}(v_r)$ of size $i_r$, and a set $C^2_{v_r}\subseteq C^1_{v_r}$ of size $j_r$, such that
\begin{itemize}
\item $\ac^i_{v_r,c^*}(\omega)=1$ for every $c^*\in C^1_{v_r}$;
\item $c^*\notin \nPa_{i}(v_r)$ for every $c^*\in C^2_{v_r}$.
\end{itemize}
Recall from Section~\ref{sec:alg:itera} (step 3) that a color $c^*$ is not in $\nPa_{i}(v_r)$, either because of $\ke^{\rd}_{v_r, c^*}(\omega)=0$ or because there exists an edge $e_{r, c^*}\in\cup_{\ell\geq 2}\Hc^{\rd-1}_{\ell, c^*}$ s.t. 
$v_r\in e_{r, c^*}$ and $\ac^{\rd}_{v, c^*}(\omega)=1$ for all $v\in e_{r, c^*}\setminus\{v_r\}$. 
Then we can further partition $C^2_{v_r}$ into the following two sets:
\[
\hat{C}^2_{v_r}:=\{c^*\in C^2_{v_r} \mid \ke^{\rd}_{v_r, c^*}(\omega)=0\},
\] 
and
\[
\tilde{C}^2_{v_r}:=\left\{c^*\in C^2_{v_r} \mid  \exists e_{r, c^*}\in\cup_{\ell\geq 2}\Hc^{\rd-1}_{\ell, c^*} \text{ s.t. $\ac^{\rd}_{v, c^*}(\omega)=1$ for all $v\in e_{r, c^*}\setminus\{v_r\}$}\right\}. 
\]
We can then define the verifier of $\I_{e, c}$ as
\begin{equation}\label{def:ver1}
\begin{split}
R_{e, c}(\omega)
&:=\left( \bigcup_{r=1}^{\ell-1}\bigcup_{c^*\in C^1_{v_r}}\ac^{\rd}_{v_r, c^*}\right) 
\cup \left(\bigcup_{r=1}^{\ell-1}\bigcup_{c^*\in \tilde{C}^2_{v_r}}\bigcup_{v\in e_{r, c^*}\setminus\{v_r\}}\ac^{\rd}_{v, c^*}\right) \cup \left(\bigcup_{r=1}^{\ell-1}\bigcup_{c^*\in \hat{C}^2_{v_r}}\ke^{\rd}_{v_r, c^*}\right)\\
&:=R^1_{e, c}(\omega)\cup R^2_{e, c}(\omega)\cup R^3_{e, c}(\omega).
\end{split}
\end{equation}
Observe that 
\[
|R_{e, c}(\omega)|\leq \sum_{r=1}^{\ell-1}(i_r + j_r(k-1) + j_r)=\sum_{r=1}^{\ell-1}i_r + k\sum_{r=1}^{\ell-1}j_r\leq \ell k\log^2\Delta,
\]
where the last inequality follows from~\eqref{eq:rrange}.
Then by Definition~\ref{def:verfiable}, $\I_{e, c}$ is $(\ell k\log^2\Delta)$-verifiable with the verifier $R_{e, c}$.

Next, we check the observability.
Observe that a random variable $\ke^{\rd}_{v,c^*}$ is in the verifier of some $\I_{e,c}$, only if $v\in e-u$.
Then for every $\omega\in\Omega\setminus\Omega^*_{3,\ell}$ and random variable $\ke^{\rd}_{v, c^*}$, 
\begin{equation}\label{eq:ober1}
\begin{split}
|\{(e, c)\in M_{\ell}:\ \I_{e, c}(w)=1, ~\ke^{\rd}_{v, c^*}\in R_{e, c}(w)\}| & \le |\{(e, c)\in M_{\ell}:~\{u, v\}\in e\}|\\
&\leq \sum_{c\in \Pa_{\rd-1}(u)}\delta_{2,\ell}(\Hc^{\rd-1}_{\ell, c}) \leq |\Pa_{\rd-1}(u)|(\cde\pa_{\rd-1})^{\ell-2},
\end{split}
\end{equation}
where the last inequality follows from Proposition~\ref{prop:output}.

The calculation for activation variables is similar but more involved.
For a random variable $\ac^i_{v, c^*}$, observe that it appears in the verifier of some $\I_{e,c}$, as part of $R^1_{e, c}(w)$, only if $v\in e-u$. On the other hand, it appears in the verifier of some $\I_{e,c}$, as part of $R^2_{e, c}(w)$, only if \[
\text{$c^*\in \bigcup_{v\in N^2_{i-1}(u)}\Pa_{\rd-1}(v)$\quad and \quad $c^*\in \Ac_{\rd}(v)$ for some $v\in e-u$.}
\]
Therefore, for every $\omega\in\Omega\setminus\Omega^*_{3,\ell}$ and random variable $\ac^i_{v, c^*}$, 
\begin{equation}\label{eq:ober2}
\begin{split}
& |\{(e, c)\in M_{\ell}:\ \I_{e, c}(w)=1,~\ac^{\rd}_{v, c^*}\in R_{e, c}(w)\}| \\
\leq~&  |\{(e, c)\in M_{\ell}:\ \I_{e, c}(w)=1,~\ac^{\rd}_{v, c^*}\in R^1_{e, c}(w)\}| 
+
|\{(e, c)\in M_{\ell}:\ \I_{e, c}(w)=1,~\ac^{\rd}_{v, c^*}\in R^2_{e, c}(w)\}|\\
\leq~& |\Pa_{\rd-1}(u)|(\cde\pa_{\rd-1})^{\ell-2} + \left|\left\{(e, c)\in M_{\ell}:~c^*\in \bigcup_{x\in e\setminus\{u\}}\Ac_{\rd}(x)\right\}\right|\\
\leq~& |\Pa_{\rd-1}(u)|(\cde\pa_{\rd-1})^{\ell-2} +  |\Pa_{\rd-1}(u)|(\cde\pa_{\rd-1})^{\ell-2}\log^2\Delta \leq 2|\Pa_{\rd-1}(u)|(\cde\pa_{\rd-1})^{\ell-2}\log^2\Delta,    
\end{split}
\end{equation}
where the second inequality follows similarly to~\eqref{eq:ober1}, and the third inequality uses $\omega\notin\Omega^*_{3,\ell}$ along with the definition of $\Omega^*_{3,\ell}$.
By Definition~\ref{def:obser}, this completes the proof of the claim.
\end{proof}

Recall from~\eqref{def:tau1} that
\[
\tau_{\ell}=\alpha'_{\rd}\beta^{k}|\Pa_{\rd-1}(u)|\left(\frac{t_{\rd-1}}{(\cde\pa_{\rd-1})^{k-\ell}}\right)\Delta^{-\theta},
\]
Note that for every integers $i_1, \ldots, i_{\ell-1}$ and $j_1, \ldots, j_{\ell-1}$,  
\[
\E\left[\rvb_{i_1, \ldots,i_{\ell-1}}^{j_1, \ldots, j_{\ell-1}}\right] \leq |M_{\ell}|=\sum_{c\in \Pa_{\rd-1}(u)} d^{\rd-1}_{\ell}(u, c)
\leq |\Pa_{\rd-1}(u)|\cdot\frac{2t_{\rd-1}}{(\cde\pa_{\rd-1})^{k -\ell}},
\]
where the right-hand side is much larger than $\tau_{\ell}$.
Applying Theorem~\ref{talagrand} to $\rvb_{i_1, \ldots,i_{\ell-1}}^{j_1, \ldots, j_{\ell-1}}$ with $\tau_{\ell}$ and $\Omega^*_{3,\ell}$, we obtain that
\[
\begin{split}
& \Prb\left(\left|\rvb_{i_1, \ldots,i_{\ell-1}}^{j_1, \ldots, j_{\ell-1}}  - \E\left[\rvb_{i_1, \ldots,i_{\ell-1}}^{j_1, \ldots, j_{\ell-1}}\right]\right|> \tau_{\ell}\right)  \\
\leq~ &  4\exp\left(-\frac{\tau_{\ell}^2}{8\ell k\log^2\Delta\cdot 2|\Pa_{\rd-1}(u)|(\cde\pa_{\rd-1})^{\ell-2}\log^2\Delta \left(4\E\left[\rvb_{i_1, \ldots,i_{\ell-1}}^{j_1, \ldots, j_{\ell-1}}\right] + \tau_{\ell}\right)}\right) + 4\Prb(\Omega^*_{3,\ell}) \\
\leq~ & 4\exp\left(-\Omega\left(
\frac{t_{i-1}\Delta^{-2\theta}}{(\cde\pa_{\rd-1})^{k-2}\log^4\Delta}
\right)\right) + 4\exp\left(-\Omega\left(\log^2\Delta\right)\right) \\
\leq~ & 4\exp\left(-\Omega\left(
\frac{\Delta^{1/2(k-1) - 1/2k}}{\log^4\Delta}
\right)\right) + 4\exp\left(-\Omega\left(\log^2\Delta\right)\right)
= \exp\left(-\Omega\left(\log^2\Delta\right)\right).
\end{split}
\]
where the second inequality follows from Lemma~\ref{lem:exceout3}, and the last inequality uses~\eqref{indas5} along with $\theta=1/4k$.
Then, by the union bound and~\eqref{eq:Ydefnew}, we have
\[
\Prb\left(|\Y_{\ell}^1 - \E[\Y_{\ell}^1]| > (2\log\Delta)^{2(\ell-1)}\tau_{\ell}\right) \\
\leq (2\log\Delta)^{2(\ell-1)}\exp\left(-\Omega\left(\log^2\Delta\right)\right) = \exp\left(-\Omega\left(\log^2\Delta\right)\right).
\]
This, together with~\eqref{eq:appronum}, shows that
\begin{equation}\label{con:xell1}
\Prb\left(|\X_{\ell}^1 - \E[\X_{\ell}^1] | \leq (2\log\Delta)^{2(\ell-1)}\tau_{\ell}\right) 
\geq 1 - 2\exp\left(-\Omega\left(\log^2\Delta\right)\right) =1 - \exp\left(-\Omega\left(\log^2\Delta\right)\right).
\end{equation}

The argument for the concentration of $\X_{\ell}^2$ is very similar to that of $\X_{\ell}^1$. Therefore, from this point on, we will focus on the differences and be brief on the similar parts. 

For integers $i_1, \ldots, i_{\ell-1}$ and $j_1, \ldots, j_{\ell-1}$, define
\begin{multline*}
(\rvb')_{i_1,\ldots,i_{\ell-1}}^{j_1,\ldots,j_{\ell-1}}:=\left|\left\{(e, c)\in M_{\ell}:~e = uv_1 \ldots v_{\ell} \text{ s.t. } \forall r\in [\ell-1],~|\Ac_{\rd}(v_r)|\ge i_r,\right.\right. \\
\left.\left.\left|\Ac_{\rd}(v_r)\cap\left(\Pa_{\rd-1}(v_r) - \nPa_{\rd}(v_r)\right)\right|\ge j_r\right\}, \text{ and } \exists y\in e,~c\notin\nPa_i(y)\right|. 
\end{multline*}
and then let
\begin{equation}\label{def:Y2}
\Y_{\ell}^2:=\sum_{\substack{0\leq i_r\leq \log^2\Delta\\  \forall r\in[\ell-1]}}\sum_{\substack{\sigma_r, \tau_r\in\{0,1\}\\  \forall r\in[\ell-1]}} (-1)^{f(\sigma_1, \tau_1, \ldots, \sigma_{\ell-1}, \tau_{\ell-1})}(\rvb')_{i_1+\sigma_1, \ldots, i_{\ell-1} + \sigma_{\ell-1}}^{i_1 +\tau_1, \ldots, i_{\ell-1} + \tau_{\ell-1}}.
\end{equation}
A similar argument to that for $\X^1_{\ell}$ shows that
\begin{equation}\label{X2Y2apprx}
\Prb(\X_{\ell}^2 = \Y_{\ell}^2) \geq 1 -  \exp\left(- \Omega\left(\log^2\Delta\right)\right).
\end{equation}
\begin{claim}\label{claim:rdver2}
For every integers $i_1, \ldots, i_{\ell-1}$ and $j_1, \ldots, j_{\ell-1}$ that satisfies~\eqref{eq:rrange}, 
\[
\text{$(\rvb')_{i_1, \ldots,i_{\ell-1}}^{j_1, \ldots, j_{\ell-1}}$ is $\left(\ell k\log^2\Delta,\ 3|\Pa_{\rd-1}(u)|(\cde\pa_{\rd-1})^{\ell-2}\log^2\Delta\right)$-observable with respect to $\Omega^*_{3,\ell}$,}
\]
where $\Omega^*_{3,\ell}$ is defined as in Lemma~\ref{lem:exceout3}.
\end{claim}
\begin{proof}
For every $(e,c)\in M_{\ell}$ with $e=uv_1\ldots v_{\ell}$, let $\I'_{e,c}$ denote the indicator variable for the event that
\begin{itemize}
\item[(i)] $|\Ac_{\rd}(v_r)|\geq i_r$ for all $r\in[\ell-1]$;
\item[(ii)] $|\Ac_{\rd}(v_r)\cap(\Pa_{\rd-1}(v_r) - \nPa_{\rd}(v_r))|\geq j_r$ for all $r\in[\ell-1]$;
\item[(iii)] $c\notin\nPa_{\rd}(y)$ for some $y\in e$.
\end{itemize}
Note that $\I'_{e, c}$ is very similar to $\I_{e,c}$ in Claim~\ref{claim:rdver1}, except that it requires an additional condition, namely (iii).
Observe that
\[
(\rvb')_{i_1, \ldots,i_{\ell-1}}^{j_1, \ldots, j_{\ell-1}} := \sum_{(e, c)\in M_{\ell}} \I'_{e,c}.
\]

For any $\omega\in\Omega\setminus \Omega^*_{3,\ell}$ with $\I'_{e,c}(\omega)=1$, by the definition of $\I'_{e,c}$, under the outcome $\omega$, 
for every $r\in [\ell-1]$, there exists $C^1_{v_r}$, $\hat{C}^2_{v_r}$, $\tilde{C}^2_{v_r}$ as defined in the proof of Claim~\ref{claim:rdver1}; additionally, there exists a vertex $y\in e$ such that $c\notin \nPa_{\rd}(y)$.
Note that $c\notin \nPa_{\rd}(y)$, either because of $\ke^{\rd}_{y, c}(\omega)=0$, or because there exists an edge $e_{y, c}\in \bigcup_{\ell\geq 2}\Hc^{\rd-1}_{\ell, c}$ s.t. $y\in e_{y, c}$ and $ \ac^{\rd}_{v, c}(\omega)=1$ for all $v\in e_{y, c}\setminus \{y\}$. 
We can then define the verifier of $\I'_{e, c}$ as
\[
R'_{e,c}(\omega):=R_{e,c}(\omega) \cup R^4_{e,c}(\omega)= 
\left\{\begin{array}{lr}
        R_{e,c}(\omega) \cup \{\ke^{\rd}_{y, c}\} & \text{if } \ke^{\rd}_{y, c}(\omega)=0;\\
        R_{e,c}(\omega) \cup \left(\bigcup_{v\in e_{y, c}\setminus \{y\}}\ac^{\rd}_{v, c}\right) & \text{if } \ke^{\rd}_{y, c}(\omega)\neq 0,
        \end{array}\right.
\]
where $R_{e,c}(\omega)$ is as defined in~\eqref{def:ver1}.
A similar calculation as in Claim~\ref{claim:rdver1} shows that every $\I'_{e,c}$ is $(\ell k\log^2\Delta)$-verifiable.

Next, we check the observability.
For a random variable $\ke^i_{v, c^*}$, observe that it appears in the verifier of some $\I'_{e,c}$, as part of $R^4_{e, c}(w)$, only if $c=c^*$ and $\{v, u\}\in e$, where $v$ could possibly be equal to $u$. Combing this with~\eqref{eq:ober1}, we have that for every $\omega\in\Omega\setminus\Omega^*_{3,\ell}$ and random variable $\ke^{\rd}_{v, c^*}$, 
\[
\begin{split}
|\{(e, c)\in M_{\ell}:\ \I'_{e, c}(w)=1, ~\ke^{\rd}_{v, c^*}\in R'_{e, c}(w)\}| 
&\le |\Pa_{\rd-1}(u)|(\cde\pa_{\rd-1})^{\ell-2} +\max\left\{\delta_{2,\ell}(\Hc^{\rd-1}_{\ell, c^*}), d^{\rd-1}_{\ell}(u, c^*)\right\}\\
&\le|\Pa_{\rd-1}(u)|(\cde\pa_{\rd-1})^{\ell-2} + (\cde\pa_{\rd-1})^{\ell-2} + 2t_{\rd-1}/(\cde \pa_{\rd-1})^{k-\ell}\\
& \le 2|\Pa_{\rd-1}(u)|(\cde\pa_{\rd-1})^{\ell-2} +  2(k-1)(\cde\pa_{\rd-1})^{\ell-1}\log\Delta\\
&\ll |\Pa_{\rd-1}(u)|(\cde\pa_{\rd-1})^{\ell-2}\log^2\Delta
\end{split}
\]
where the second inequality follows from~Proposition~\ref{prop:output}, the third inequality uses $\zeta_{\rd-1}=\frac{t_{\rd-1}}{(\cde\pa_{\rd-1})^{k-1}}$ along with~\eqref{indas6}, and the last inequality uses $p_{i-1}\le(1 + o(1))|\Pa_{\rd-1}(u)|$, which follows from Proposition~\ref{prop:pbou} and the assumption \eqref{indas1}.

On the other hand, for a random variable $\ac^i_{v, c^*}$, observe that it appears in the verifier of some $\I'_{e,c}$, as part of $R^4_{e, c}(w)$, only if $c=c^*$ and $u\in e$.
Combing this with~\eqref{eq:ober2}, we have that for every $\omega\in\Omega\setminus\Omega^*_{3,\ell}$ and random variable $\ac^{\rd}_{v, c^*}$, 
\[
\begin{split}
|\{(e, c)\in M_{\ell}:\ \I'_{e, c}(w)=1, ~\ac^{\rd}_{v, c^*}\in R'_{e, c}(w)\}| 
&\le 2|\Pa_{\rd-1}(u)|(\cde\pa_{\rd-1})^{\ell-2}\log^2\Delta + d^{\rd-1}_{\ell}(u, c^*)\\
&\le 3|\Pa_{\rd-1}(u)|(\cde\pa_{\rd-1})^{\ell-2}\log^2\Delta.
\end{split}
\]
By Definition~\ref{def:obser}, this completes the proof of the claim.
\end{proof}

With~\eqref{def:Y2},~\eqref{X2Y2apprx}, and Claim~\ref{claim:rdver2}, following the same argument as for $\X^1_{\ell}$, we can show that
\begin{equation}\label{con:xell2}
\Prb\left(|\X_{\ell}^2 - \E[\X_{\ell}^2] |\leq (2\log\Delta)^{2(\ell-1)}\tau_{\ell}\right) 
\geq 1 - \exp\left(-\Omega\left(\log^2\Delta\right)\right).
\end{equation}
Finally, recall that $\X_{\ell}=\X_{\ell}^1- \X_{\ell}^2$. Combining~\eqref{con:xell1} and~\eqref{con:xell2}, and applying the union bound one more time, we obtain that
\[
 \Prb\left(|\X_{\ell} - \E[\X_{\ell}] |\leq 2(2\log\Delta)^{2(\ell-1)}\tau_{\ell}\right) \geq 1 - \exp\left(-\Omega\left(\log^2\Delta\right)\right).
\]
 This completes the proof of Lemma~\ref{lem:con:Xell}.

\subsection{Proof of Lemma~\ref{lem:con:Xsell}: Concentration of $\X_{\ell, s}$}\label{sec:con:ells}
Fix arbitrary $2\leq \ell <s\leq k$, and recall from~\eqref{def:inedge} that
\[
\X_{\ell, s}=\sum_{c\in\Pa_{\rd-1}(u)}\sum_{\substack{e\in \Hc^{\rd-1}_{s, c}\\ u\in e}}\sum_{Q\in\binom{e-u}{s-\ell}}
\I\left[\left(\forall x\in Q,~c\in\Ac_{\rd}(x) \right)\wedge\left(\forall y\in e,~c\in \nPa_{\rd}(y)\right)\right].
\]
Similarly to $\X_{\ell}$, we will express $\X_{\ell, s}$ as a linear combination of several random variables that are $(r, d)$-observable.
More formally, let
\[
\Ic_{\ell, s}:=\{(c, e, Q)\mid c\in \Pa_{\rd-1}(u),\ u\in e\in \Hc^{\rd-1}_{s, c},\ Q\subseteq e-u \text{ and } |Q|=s-\ell\}.
\]
Then we define
\[
\X^1_{\ell, s}:=\sum_{(c, e, Q)\in\Ic_{\ell, s}}\I\left[\left(\forall x\in Q,~c\in\Ac_{\rd}(x) \right)\right],
\]
and
\[
\X^2_{\ell, s}:=\sum_{(c, e, Q)\in\Ic_{\ell, s}}\I\left[\left(\forall x\in Q,~c\in\Ac_{\rd}(x) \right)\wedge\left(\exists y\in e,~c\notin \nPa_{\rd}(y)\right)\right].
\]
Observe that $\X_{\ell, s}=\X^1_{\ell, s}-\X^2_{\ell, s}$. 
\begin{claim}\label{claim:rdversl}
Both $\X^1_{\ell, s}$, $\X^2_{\ell, s}$ are 
\[
\text{$\left(2k,~\binom{s-1}{s-\ell}(\cde p_{i-1})^{\ell-1}\log^{2(s-\ell)}\Delta\right)$-observable with respect to $\Omega^*_{4, s,\ell}$,}
\]
where $\Omega^*_{4, s,\ell}$ is defined as in Lemma~\ref{lem:outcome5}.
\end{claim}
\begin{proof}
We only present the proof for $\X^2_{\ell, s}$; a similar (and even simpler) argument applies to $\X^1_{\ell, s}$.
For simplicity of notation, for every $(c, e, Q) \in \Ic_{\ell, s}$, we denote
\[
\I_{c, e,Q}:=\I\left[\left(\forall x\in Q,~c\in\Ac_{\rd}(x) \right)\wedge\left(\exists y\in e,~c\notin \nPa_{\rd}(y)\right)\right],
\]
and note that $\X^2_{\ell, s}=\sum_{(c, e, Q)\in\Ic_{\ell, s}}\I_{c, e,Q}$.

We first check the verifiability of each $\I_{c, e,Q}$.
For any $\omega\in\Omega\setminus \Omega^*_{4, s,\ell}$ with $\I_{c, e,Q}(\omega)=1$, by the definition of $\I_{c, e,Q}$, under the outcome $\omega$, we have $\ac^{\rd}_{x, c}(\omega)=1$ for all $x\in Q$, and there exists a vertex $y\in e$ such that $c\notin\nPa_{\rd}(y)$.
Note that, as shown before in Claim~\ref{claim:rdver2}, $c\notin \nPa_{\rd}(y)$, either because of $\ke^{\rd}_{y, c}(\omega)=0$, or because there exists an edge $e_{y, c}\in \bigcup_{\ell\geq 2}\Hc^{\rd-1}_{\ell, c}$ s.t. $y\in e_{y, c}$ and $ \ac^{\rd}_{v, c}(\omega)=1$ for all $v\in e_{y, c}\setminus \{y\}$. 
We can then define the verifier of $\I_{c, e,Q}$ as
\[
R_{c, e, Q}(\omega):=
\begin{cases}
\left(\bigcup_{x\in Q}\ac^{\rd}_{x, c}\right)\cup\{\ke^{\rd}_{y, c}\},  & \text{if } \ke^{\rd}_{y, c}(\omega)=0; \\
\left(\bigcup_{x\in Q}\ac^{\rd}_{x, c}\right)\cup \left(\bigcup_{v\in e_{y, c}\setminus \{y\}}\ac^{\rd}_{v, c}\right),  & \text{if } \ke^{\rd}_{y, c}(\omega)\notin 0.
\end{cases}
\]
Observe that $|R_{c, e, Q}(\omega)|\le |Q|+k-1\le 2k$.
Then by Definition~\ref{def:verfiable}, $\I_{c, e, Q}$ is $2k$-verifiable with the verifier $R_{c, e, Q}$.

Next, we check the observability. Observe that a random variable $\ke^i_{v,c^*}$ is in the verifier of some $\I_{c, e, Q}$, only if $c=c^*$ and at least $s - \ell$ vertices in $e - u$ have the color $c^*$ activated.
Therefore, for every $\omega\in \Omega \setminus\Omega^*_{4, s,\ell}$ and random variable $\ke^{\rd}_{v,c^*}$,  
we have 
\[
\begin{split}
|\{
(c, e, Q)\in \Ic_{\ell, s}\mid \I_{c,e,Q}(\omega)=1,\ \ke^{\rd}_{v,c^*}\in R_{c, e, Q}(\omega)
\}|&
\leq \left|\left\{e\in \Hc^{\rd-1}_{s, c^*} \mid u\in e,\ \sum_{v\in e-u}\ac^{\rd}_{v,c^*}(\omega)\geq s-\ell\right\}\right|\binom{s-1}{s-\ell}\\
&\leq\binom{s-1}{s-\ell}(\cde p_{i-1})^{\ell-1}\log^{2(s-\ell)}\Delta,
\end{split}
\]
where the last inequality follows from $\omega\notin \Omega^*_{4, s, \ell}$ and the definition of $\Omega^*_{4, s, \ell}$.

Similarly, a random variable $\ac^i_{v,c^*}$ is in the verifier of some $\I_{c, e, Q}$, only if $c=c^*$ and at least $s - \ell$ vertices in $e - u$ have the color $c^*$ activated. Here we did not use all restrictions, but the bound obtained from this limited information is sufficient.
Then for every $\omega\in \Omega \setminus\Omega^*_{4, s,\ell}$ and random variable $\ac^{\rd}_{v,c^*}$,  
\[
|\{
(c, e, Q)\in \Ic_{\ell, s}\mid \I_{c,e,Q}(\omega)=1,\ \ac^{\rd}_{v,c^*}\in R_{c, e, Q}(\omega)
\}|
\leq \binom{s-1}{s-\ell}(\cde p_{i-1})^{\ell-1}\log^{2(s-\ell)}\Delta.
\]
By Definition~\ref{def:obser}, this completes the proof of the claim for $\X^2_{\ell, s}$.
\end{proof}

Observe that
\[
\begin{split}
\E[\X^2_{\ell, s}] \leq \E[\X^1_{\ell, s}]
&= |\Ic_{\ell, s}|\pi_{\rd}^{s-\ell} 
= \sum_{c\in \Pa_{\rd-1}(u)}d^{\rd-1}_{s}(u, c)\binom{s-1}{s-\ell}\pi_{\rd}^{s-\ell}
\leq \binom{s-1}{s-\ell} |\Pa_{\rd-1}(u)|\frac{2t_{\rd-1}}{(\cde\pa_{\rd-1})^{k-s}}\pi_{\rd}^{s-\ell}\\
&=\binom{s-1}{s-\ell} |\Pa_{\rd-1}(u)|\frac{2t_{\rd-1}}{(\cde\pa_{\rd-1})^{k-\ell}}(\cde\pa_{\rd-1}\pi_{\rd})^{s-\ell}
\leq \binom{s-1}{s-\ell} |\Pa_{\rd-1}(u)|\frac{2t_{\rd-1}}{(\cde\pa_{\rd-1})^{k-\ell}},
\end{split}
\]
where the first inequality follows from Proposition~\ref{prop:output}, and the last inequality follows from $\cde\pa_{\rd-1}\pi_{\rd}\le 1$, as given by~\eqref{indas4}.
Recall from~\eqref{def:tau1} that
\[
\tau_{\ell}= \alpha'_{\rd}\beta^{k}|\Pa_{\rd-1}(u)|\left(\frac{t_{\rd-1}}{(\cde\pa_{\rd-1})^{k-\ell}}\right)\Delta^{-\theta},
\]
which is much smaller than the upper bound of expectations.
Applying Theorem~\ref{talagrand} on $\X^1_{\ell, s}$, $\X^2_{\ell, s}$ with $\tau_{\ell}$ and $\Omega^*_{4,s,\ell}$, we obtain that
\[
\begin{split}
&\Prb(|\X^1_{\ell, s} - \E[\X^1_{\ell, s}]|>\tau_{\ell}),\ \Prb(|\X^2_{\ell, s} - \E[\X^2_{\ell, s}]|>\tau_{\ell})\\
\leq\ &4\exp\left(-\frac{\tau_{\ell}^2}{8\cdot 2k\cdot \binom{s-1}{s-\ell}(\cde p_{i-1})^{\ell-1}\log^{2(s-\ell)}\Delta(4\E[\X^1_{\ell, s}] + \tau_{\ell})}\right) + 4\Prb(\Omega^*_{4,s,\ell}) \\
\leq\ & \exp\left(-\Theta\left(\frac{|\Pa_{\rd-1}(u)|t_{\rd-1}\Delta^{-2\theta}}{(\cde\pa_{\rd-1})^{k-1}\log^{2(s-\ell)}\Delta}\right)\right)
+ \exp\left(-\Omega(\log^2\Delta)\right)\\
\leq\ & \exp\left(-\Theta\left(\frac{t_{\rd-1}\Delta^{-2\theta}}{(\cde\pa_{\rd-1})^{k-2}\log^{2(s-\ell)}\Delta}\right)\right)
+ \exp\left(-\Omega(\log^2\Delta)\right)\\
\leq~&\exp\left(-\Omega\left(\frac{\Delta^{1/2(k-1)-1/2k}}{\log^{2(s-\ell)}\Delta}\right)\right)
+ \exp\left(-\Omega(\log^2\Delta)\right)=\exp\left(-\Omega(\log^2\Delta)\right),
\end{split}
\]
where the second inequality uses Lemma~\ref{lem:outcome5}, the third inequality uses $|\Pa_{\rd-1}(u)|\ge (1- o(1))p_{i-1}$ (given by Proposition~\ref{prop:pbou} and the assumption \eqref{indas1}), and the last inequality follows from~\eqref{indas5} along with $\theta=1/4k$.

Finally,  recall that $\X_{\ell,s}=\X_{\ell,s}^1- \X_{\ell,s}^2$. Applying the union bound, we obtain that
\[
 \Prb\left(|\X_{\ell, s} - \E[\X_{\ell, s}]|\leq 2\tau_{\ell}\right) \geq 1- \exp\left(-\Omega(\log^2\Delta)\right).
\]
This completes the proof of Lemma~\ref{lem:con:Xsell}.

\section{Proof of Theorem~\ref{mainthm1}}\label{sec:mainthm1}
Our plan is to apply the codegree reduction algorithm (see Section~\ref{sec:coredu}) to the original hypergraph $\Hc$, reducing codegrees until the conditions of Theorem~\ref{mainthm2} are satisfied for some $\Delta$.

Let $V:=V(\Hc)$.
We start the algorithm with $\Hc^0:=\Hc$ and $\Lambda_{k} := (k2^k)^k\Delta_k$. In the $i$-th iteration round, for every vertex $u$ and $2\leq k-i < \ell \leq k$, let 
\[
F_{k-i, \ell}(u):=\left\{ S\subseteq V:\ |S|=k-i,\ u\in S,~ \text{and}\  \deg_{\ell}(S, \mathcal{H}^{i-1}) \geq (\Lambda_{\ell}/\log\Lambda_{\ell})^{\frac{\ell - (k -i)}{k-1}}\right\}.
\]
Define $\mathcal{H}^i$ as the following:
\[
E(\mathcal{H}^i) := E(\mathcal{H}^{i-1}) - \bigcup_{u}\bigcup_{\ell>k-i}\bigcup_{S\in F_{k-i, \ell}(u)}\{e\in \mathcal{H}^{i-1}:\  e \supseteq S,\ |e|=\ell \} + \bigcup_{u}\bigcup_{\ell>k-i}F_{k-i, \ell}(u).
\]
We then take $\Lambda_{k-i}$ such that
\begin{equation}\label{sec:reduce:def:Lambda}
\Delta_{k-i}(\Hc^{i})=\frac{1}{(k2^k)^{k-i}}\cdot \Lambda_{k-i}^{1 - \frac{i}{k-1}} (\log \Lambda_{k-i})^{\frac{i}{ k-1}},
\end{equation}
and move to the next round.
The algorithm terminates after $k-2$ rounds.

Observe that this process produces a sequence of hypergraphs $\Hc^0, \ldots, \Hc^{k-2}$ that satisfies the following two properties:
\begin{itemize}
\item[(1)] for every $0\leq i \leq k-2$,
\begin{multline*}
\Delta_{k-i}(\Hc^0)=\ldots=\Delta_{k-i}(\Hc^{i-1}) \leq \Delta_{k-i}(\Hc^{i})=\frac{1}{(k2^k)^{k-i}}\cdot \Lambda_{k-i}^{1 - \frac{i}{k-1}} (\log \Lambda_{k-i})^{\frac{i}{ k-1}}\\
\geq \Delta_{k-i}(\Hc^{i+1})\geq \ldots \geq \Delta_{k-i}(\Hc^{k-2});
\end{multline*}
\item[(2)] for every $2\leq k-i < \ell \leq k$,
\[
\delta_{k-i, \ell}(\Hc^{k-2})  \leq \ldots \leq \delta_{k-i, \ell}(\Hc^{i+1}) \leq \delta_{k-i, \ell}(\Hc^{i}) 
\leq (\Lambda_{\ell}/\log\Lambda_{\ell})^{\frac{\ell - (k-i)}{k-1}}.
\]
\end{itemize} 
Moreover, note that $\Hc^{k-2}$ is indeed an $f$-reduction of $\Hc$ with the function $f(s,  \ell):=(\Lambda_{\ell}/\log\Lambda_{\ell})^{\frac{\ell - s}{k-1}}$, which satisfies the assumption of Proposition~\ref{prop:redu}. Therefore, by Proposition~\ref{prop:redu}, we have two more properties:
\begin{itemize}
\item[(3)] any proper coloring of $\Hc^{k-2}$ is also proper for $\Hc$;
\item[(4)] $\Hc^{k-2}$ is triangle-free.
\end{itemize}
Define 
\[
\Delta := \max_{0\leq i \leq k-2} \Lambda_{k-i}.
\]
By Properties (1)(2)\&(4), $\mathcal{H}^{k-2}$ satisfies all the assumptions of Theorem~\ref{mainthm2} with $\Delta$.
Applying Theorem~\ref{mainthm2} to $\mathcal{H}^{k-2}$, we find that there exists a constant $c'$ such that 
\begin{equation}\label{sec:reduce:chibd}
\chi_{\ell}(\mathcal{H}) 
\leq \chi_{\ell}(\mathcal{H}^{k-2})
\leq c' \left( \frac{\Delta}{\log \Delta}\right)^{\frac{1}{k-1}}.
\end{equation}

Take $i$ such that $\Delta = \Lambda_{k-i}$. 
First, by Property (1), we have that for any $\ell$,
\[
\Delta_\ell(\Hc^{i-1}) \le \Delta_\ell(\Hc^{k-\ell})=\frac{1}{(k2^k)^{\ell}}\cdot\Lambda_{\ell}^{1 - \frac{k-\ell}{k-1}} (\log \Lambda_{\ell})^{\frac{k-\ell}{ k-1}}.
\]
On the other hand, by the definition of $F_{k-i, \ell}(u)$, for every vertex $u$ and $\ell > k-i$,
\[
\Delta_\ell(\Hc^{i-1})\geq d_{\ell}(u, \mathcal{H}^{i-1})
\geq \frac{1}{\binom{\ell-1}{k-i-1}} |F_{k-i, \ell}(u)|(\Lambda_{\ell}/\log\Lambda_{\ell})^{\frac{\ell - (k -i)}{k-1}}.
\]
Combining the two inequalities above, we obtain that
\[
|F_{k-i, \ell}(u)|
\leq \frac{\binom{\ell-1}{k-i-1}}{(k2^k)^{\ell}}\Lambda_{\ell}^{1 - \frac{i}{k-1}}(\log \Lambda_{\ell})^{\frac{i}{ k-1}}
\leq \frac{1}{2k(k2^k)^{\ell-1}}\Lambda_{k-i}^{1 - \frac{i}{k-1}}(\log \Lambda_{k-i})^{\frac{i}{ k-1}}
\leq \frac{1}{2k}\Delta_{k-i}(\Hc^i),
\]
where the second inequality uses the maximality of $\Lambda_{k-i}$, and the last inequality follows from~\eqref{sec:reduce:def:Lambda} along with $\ell> k-i$.
Then by the definition of $E(\mathcal{H}^i)$, we have that
\[
\Delta_{k-i}(\Hc^i) \leq \Delta_{k-i}(\Hc) + \max_{u}\sum_{\ell>k-i}|F_{k-i, \ell}(u)| \leq
\Delta_{k-i}(\Hc) +  \Delta_{k-i}(\Hc^i)/2,
\]
and thus $\Delta_{k-i}(\Hc^i) \leq 2\Delta_{k-i}(\Hc)$.
This, together with (\ref{sec:reduce:def:Lambda}) and (\ref{sec:reduce:chibd}), shows that
\[
\chi_{\ell}(\mathcal{H}) 
\leq c' \left( \frac{\Lambda_{k-i}}{\log \Lambda_{k-i}}\right)^{\frac{1}{k-1}}
\leq \frac{c}{2}\left( \frac{\Delta_{k-i}(\Hc^i)}{\log \Delta_{k-i}(\Hc^i)}\right)^{\frac{1}{k-i-1}}
\leq c\left( \frac{\Delta_{k-i}(\Hc)}{\log \Delta_{k-i}(\Hc)}\right)^{\frac{1}{k-i-1}},
\] 
for some sufficiently large constant $c$, thereby completing the proof of Theorem~\ref{mainthm1}.

\section{Proof of Theorem~\ref{sparsethm}}\label{sec:spars}
We say a rank $k$ hypergraph $\Hc$ is \textit{$(\Delta, \omega_2,\ldots, \omega_{k})$-sparse}, if $\Hc$ has
maximum $k$-degree at most $\Delta$, and for all $1 \leq s < \ell \leq k$, $\Hc$ has maximum $(s, \ell)$-codegree $\delta_{s, \ell} \leq \Delta^{\frac{\ell-s}{k-1}}\omega_{\ell}$.

To prove Theorem~\ref{sparsethm}, we use the following partition lemma from~\cite[Lemma 7]{cooper2016coloring}.

\begin{lemma}\label{lem:parti}
Fix $k\geq 2$. Let $\Hc$ be a rank $k$ hypergraph, and $\Fc$ be a finite family of fixed, connected hypergraphs. Let $f=\Delta^{O(1)}$, where $f$ is sufficiently large. Suppose that
\begin{itemize}
\item $\Hc$ is $(\Delta, \omega_2,\ldots, \omega_{k})$-sparse, where $\omega_{\ell}=\omega_{\ell}(\Delta)=f^{o(1)}$ for all $2\leq \ell \leq k$;
\item for all $F\in\Fc$, $\Delta_{F}(\Hc)\leq \Delta^{\frac{v(F)-1}{k-1}}/f^{v(F)}$.
\end{itemize}
Then $V(\Hc)$ can be partitioned into $\bO\left(\Delta^{\frac{1}{k-1}}/f\right)$ parts such that the hypergraph induced by each part is $\Fc$-free and has maximum $\ell$-degree at most $2^{2k}f^{\ell-1}\omega_{\ell}$ for each $2 \leq \ell\leq k$.
\end{lemma}

\begin{proof}[Proof of Theorem~\ref{sparsethm}]
Without loss of generality, we assume that $1\ll f =\Delta^{\bO(1)}$, as otherwise, the conclusion easily follows from Theorem~\ref{mainthm1} or a direction application of the Local Lemma.
Take $\Delta$ such that
\[
\left(\frac{\Delta}{\log f}\right)^{\frac{1}{k-1}}=\max_{2\leq \ell \leq k} \left\{\left(\frac{\Delta_{\ell}}{\log f}\right)^{\frac{1}{\ell-1}}\right\},
\]
and set
\[
f_1:=f/(\log f)^{(k-2)(3k - 4)/(k-1)}.
\]
By the maximality, we have that for all $2\leq \ell\leq k$,
\begin{equation}\label{debound}
\Delta_{\ell} \leq \Delta^{\frac{\ell-1}{k-1}}(\log f)^{1 - \frac{\ell-1}{k-1}},
\end{equation}
and therefore for all $T\in \Tc$,
\begin{equation}\label{tribound}
\begin{split}
\Delta_{T}(\Hc)& \leq \left(\max_{2\leq\ell\leq k}\Delta_{\ell}^{1/(\ell-1)}\right)^{v(T)-1}/f
\leq \left(\max_{2\leq\ell\leq k}\Delta^{1/(k-1)}(\log f)^{\frac{k-\ell}{(k-1)(\ell-1)}}\right)^{v(T)-1}/f\\
& \leq \left(\Delta^{1/(k-1)}(\log f)^{\frac{k-2}{k-1}}\right)^{v(T)-1}/f
\leq \left(\Delta^{1/(k-1)}\right)^{v(T)-1}/f_1,
\end{split}
\end{equation}
where the last inequality follows from the fact that $v(T)\leq 3k-3$.

Now we apply our codegree reduction algorithm (described in Section~\ref{sec:coredu}) to $\Hc$. 
Let $\Hc'$ be a $g$-reduction of $\Hc$, where $g(x, y):=\Delta^{\frac{y-x}{k-1}}$. By Proposition~\ref{prop:redu} and the mechanics of the algorithm, this codegree reduction process produces a sequence of hypergraphs $\{\Hc=\Hc^0, \ldots, \Hc^{i}, \ldots, \Hc^{k-2}=\Hc'\}$, which satisfies the following properties:
\begin{itemize}
\item[(1)] any proper coloring of $\Hc'$ is also proper for $\Hc$;
\item[(2)] for every $0\leq i \leq k-2$,
\begin{equation}\label{itm:de}
\Delta_{k-i}=\Delta_{k-i}(\Hc^0)=\ldots=\Delta_{k-i}(\Hc^{i-1}) \leq \Delta_{k-i}(\Hc^{i})
\geq \Delta_{k-i}(\Hc^{i+1})\geq \ldots \geq \Delta_{k-i}(\Hc^{k-2});
\end{equation}
\item[(3)] for every $2\leq k-i < \ell \leq k$,
\begin{equation}\label{itm:code}
\delta_{k-i, \ell}(\Hc')  \leq  g(k-i, \ell)\leq \Delta^{\frac{\ell - (k-i)}{k-1}}.
\end{equation}
\end{itemize} 
Moreover, we have the following claim.
\begin{claim}\label{claim:De}
For every $0\leq i \leq k-2$, 
\[
\Delta_{k-i}(\Hc^{i}) \leq 2\Delta^{\frac{k-i-1}{k-1}}(\log f)^{1 - \frac{k-i-1}{k-1}}.
\]
\end{claim}
\begin{proof}
We prove it by induction on $i$. The base case $i=0$ holds trivially by~\eqref{debound}.
Let $i \ge 1$ and assume that the claim holds for all $j<i$. This, together with~\eqref{itm:de}, indicates that for all $\ell>k-i$,
\begin{equation}\label{claim:assump}
\Delta_{\ell}(\Hc^{i-1})\leq \Delta_{\ell}(\Hc^{k-\ell}) \leq 2\Delta^{\frac{\ell-1}{k-1}}(\log f)^{1- \frac{\ell-1}{k-1}}.
\end{equation}

Now let us focus on the $i$-th round of the algorithm. By the definition of $F_{k-i, \ell}(u)$ (see Section~\ref{sec:coredu}), we have that for every vertex $u$ and $\ell> k- i$,
\begin{equation*}
\frac{1}{\binom{\ell-1}{k-i-1}} |F_{k-i, \ell}(u)|\Delta^{\frac{\ell - (k -i)}{k-1}}
\leq d_{\ell}(u, \mathcal{H}^{i-1}) \leq \Delta_\ell(\Hc^{i-1}).
\end{equation*}
This, together with~\eqref{claim:assump}, shows that
\[
|F_{k-i, \ell}(u)|
\leq \binom{\ell-1}{k-i-1}\Delta_\ell(\Hc^{i-1})\Delta^{-\frac{\ell - (k -i)}{k-1}}
\leq 2\binom{\ell-1}{k-i-1}\Delta^{\frac{k-i-1}{k-1}}(\log f)^{1 -\frac{\ell-1}{k-1}}.
\]
Then by the definition of $E(\mathcal{H}^i)$, we obtain
\[
\Delta_{k-i}(\Hc^i) \leq \Delta_{k-i} + \max_{u}\sum_{\ell>k-i}|F_{k-i, \ell}(u)| \leq
2\Delta^{\frac{k-i-1}{k-1}}(\log f)^{1 - \frac{k-i-1}{k-1}},
\]
where the last inequality follows from~\eqref{debound} and $f\gg 1$.
\end{proof}

Claim~\ref{claim:De} together with~\eqref{itm:code} shows that
\begin{equation}
\text{$\Hc'$ is $(2\Delta, \omega_2, \ldots, \omega_k)$-sparse, where $\omega_{\ell}=(\log f)^{1 -\frac{\ell-1}{k-1}}$ for each $2\leq \ell\leq k$.}
\end{equation}
After establishing the sparseness, we next estimate the number of triangles in $\Hc'$. 
\begin{claim}\label{claim:tri}
For every $0\leq i \leq k-2$ and $T\in\Tc$, 
\[
\Delta_{T}(\Hc^{i}) \leq \left(i!\cdot k^{i(i+1)/2}\right)^3\left(\Delta^{1/(k-1)}\right)^{v(T)-1}/f_1.
\]
\end{claim}
\begin{proof}
We prove it by induction on $i$. The base case $i=0$ holds trivially by~\eqref{tribound}.
Let $i \ge 1$ and assume that the claim holds for $i-1$, i.e., 
\[
\Delta_{T}(\Hc^{i-1}) \leq \left((i-1)!\cdot k^{(i-1)i/2}\right)^3\left(\Delta^{1/(k-1)}\right)^{v(T)-1}/f_1.
\]
for all $T\in\Tc$.
As all edges in $\Hc^{i} - \Hc^{i-1}$ are of size $k-i$, it is sufficiently to consider all triangles $T$ who contains at least one edge of size $k-i$. 
To simplify the discussion, we further assume that $T$ contains exactly one edge of size $k-i$; other cases will follow by applying the same argument on each size $k-i$ edge.

Let $\Fc$ be the family of copies of $T$ which are in $\Hc^{i}$ but not in $\Hc^{i-1}$. For every $F\in \Fc$, denote by $e_F$ the edge of $F$ which is not in $\Hc^{i-1}$, and note that $|e_{F}|=k-i$. Then by the definition of $H^{i}$, there exists an integer $\ell_F > k-i$, such that there are at least $\Delta^{\frac{\ell_F -(k-i)}{k-1}}$ edges $e'$ of size $\ell_F$ in $\Hc^{i-1}$ with $e'\supseteq e_F$. Moreover, every such $\{e', f, g\}$ forms a copy of triangle $T_{\ell_F}$ in $\Hc^{i-1}$, where $T_{\ell_F}$ is the triangle obtained by replacing the 
size $k-i$ edge $e$ in $T$ with a size $\ell_F$ edge containing $e$ and vertices outside of $T$.

For every $k-i < \ell \le k$, define
\[
\Fc_{\ell}:=\{F\in \Fc:\ \ell_F=\ell\}.
\]
Then from the above discussion, we have
\[
\frac{1}{\binom{\ell}{k-i}}|\Fc_{\ell}| \cdot \Delta^{\frac{v(T_{\ell}) - v(T)}{k-1}}
=\frac{1}{\binom{\ell}{k-i}}|\Fc_{\ell}| \cdot \Delta^{\frac{\ell -(k-i)}{k-1}}\le \Delta_{T_{\ell}}(\Hc^{i-1})\leq \left((i-1)!\cdot k^{(i-1)i/2}\right)^3\left(\Delta^{1/(k-1)}\right)^{v(T_{\ell})-1}/f_1,
\]
and therefore,
\[
|\Fc_{\ell}|\leq \binom{\ell}{k-i}\left((i-1)!\cdot k^{(i-1)i/2}\right)^3\left(\Delta^{1/(k-1)}\right)^{v(T)-1}/f_1.
\]
Finally, we conclude that
\[
\begin{split}
\Delta_{T}(\Hc^{i}) 
&\leq \Delta_{T}(\Hc^{i-1}) + \sum_{k-i<\ell \le k}|\Fc_{\ell}| 
\leq \left(1 + \sum_{k-i<\ell \le k}\binom{\ell}{k-i}\right)\left((i-1)!\cdot k^{(i-1)i/2}\right)^3\left(\Delta^{1/(k-1)}\right)^{v(T)-1}/f_1\\
&
\leq ik^i\left((i-1)!\cdot k^{(i-1)i/2}\right)^3\left(\Delta^{1/(k-1)}\right)^{v(T)-1}/f_1
\le \left(i!\cdot k^{i(i+1)/2}\right)^3\left(\Delta^{1/(k-1)}\right)^{v(T)-1}/f_1,
\end{split}
\]
which completes the proof. We note that while the exponent 3 may seem excessive here, it is necessary because, for other $T$, we may need to apply the same calculation three times to each of its edges.
\end{proof}

Set $f':=\left(f_1/\left(k!\cdot k^{k^2/2}\right)^3\right)^{1/(3k-3)}$, and note that $\log f' = \Theta(\log f)$. Thus, Claim~\ref{claim:tri} implies that
\[
\Delta_{T}(\Hc') \leq \left(\Delta^{1/(k-1)}\right)^{v(T)-1}/(f')^{v(T)}.
\]
for all $T\in\Tc$.
Applying Lemma~\ref{lem:parti} on $\Hc'$ with $f'$, we obtain a partition of $V(\Hc)$ into $\bO\left(\Delta^{1/(k-1)}/f'\right)$ parts such that the hypergraph induced by each part is triangle-free and has maximum $\ell$-degree at most
at most $2^{2k}(f')^{\ell-1}\omega_{\ell}$ for every $2 \leq \ell\leq k$, where $\omega_{\ell}=(\log f)^{1 -\frac{\ell-1}{k-1}}.$
By Theorem~\ref{mainthm1}, we can properly color each part with lists of 
\[
\bO\left(
\max_{2\leq \ell \leq k} \left\{\left(  \frac{(f')^{\ell-1}\omega_{\ell}}{\log \left((f')^{\ell-1}\omega_{\ell}\right)}\right)^{\frac{1}{\ell-1}} \right\}
\right) \leq 
\bO\left( f'\max_{2\leq \ell \leq k}\left\{
\left(\frac{\omega_{\ell}}{\log f}\right)^{\frac{1}{\ell-1}}
\right\}\right)
= \bO\left(f'
\left(\log f\right)^{-\frac{1}{k-1}}
\right)
\]
colors.
Finally, we conclude that
\[
\chi(\Hc)\leq \chi(\Hc') \leq \bO\left(\Delta^{1/(k-1)}/f'\right)\cdot \bO\left(f'
\left(\log f\right)^{-\frac{1}{k-1}}
\right)
= \bO\left(
\max_{2\leq \ell \leq k} \left\{\left(\frac{\Delta_{\ell}}{\log f}\right)^{\frac{1}{\ell-1}}\right\}
\right),
\]
where the last equality follows from the definition of $\Delta$. 
\end{proof}

\section{Open problems}\label{sec:conclu}
In this paper, we showed that by forbidding all triangles, one can improve the trivial bound of the chromatic number and therefore independence number of hypergraphs by some polylogarithmic factor.
We remark that answering negatively a question of Ajtai, Erd\H{o}s, Koml{\'o}s and Szemer{\'e}di~\cite{ajtai1981turan}, Cooper and Mubayi~\cite{cooper2017sparse} constructed a 3-uniform, $K^-_4$-free hypergraphs with independence number at most $2n/\sqrt{\Delta}$, and thereby showed that forbidding some single triangle is not enough to improve the trivial independence number from the Tur\'an theorem, and thereofore the trivial chromatic number from the Local Lemma.
It would be interesting to determine whether our results (Theorems~\ref{mainthm1} and~\ref{mainthmind}) can be extended to a larger class of $\Tc'$-free hypergraphs, for some smaller forbidden set $\Tc'\subsetneq \Tc$ (where $\Tc$ is the collection of rank $k$ triangles for some given rank $k$).

A related but more difficult problem than that considered in this paper is to obtain analogous results for \textit{hypergraph DP-colorings}. 
The concept of \textit{DP-coloring}, or so called \textit{correspondence colorings} was developed by Dvo{\v{r}}{\'a}k and Postle~\cite{dvovrak2018correspondence} in order to generalize the notion of list coloring on graphs. This concept  was later generalized to hypegraphs due to the work of Bernshteyn and Kostochka~\cite{bernshteyn2019dp}.
For a detailed definition of hypergraph DP-colorings, we refer interested readers to \cite{bernshteyn2019dp}.
Unfortunately, our approach in this paper does not readily generalize to DP-colorings, and we believe new ideas are needed.
Intuitively speaking, when our approach moves to DP-colorings, the major new challenge is that the `hyperedge shrinking' trick we employed all the time is not applicable; indeed, applying such `shrinkage' might generate a set of forbidden correspondence on edges, which is no longer a hypergraph matching, and thus no longer forms an instantce of DP-coloring.
Moreover, unlike list colorings where the random events always keep independence among different colors, there is no guarantee of such independence in DP-colorings, which certainly creates more technical difficulties in concentration analysis.

\section*{Acknowledgments}
We are grateful to the anonymous referees for pointing out an error in an earlier version and for providing helpful comments and suggestions.

\bibliographystyle{abbrv}
\bibliography{ref} 

\end{document}